\documentclass[11pt]{amsart}
\usepackage[foot]{amsaddr}
\usepackage{amssymb, mathrsfs, amsfonts, amsmath}
\usepackage{graphicx,color}
\usepackage{epsfig}
\usepackage{float}
\usepackage[all]{xy}
\usepackage{subcaption}
\usepackage{verbatim}
\usepackage{hyperref}
\usepackage{enumitem}
\usepackage{tikz}
\usepackage{tikz-cd}
\usepackage{bm}
\usepackage{amssymb}
\usepackage{ulem}
\usepackage{pgfplots}
\usepackage{tikz-cd}
\DeclareMathOperator{\re}{Re}
\DeclareMathOperator{\im}{Im}

\newcommand{\C}{\mathbb{C}}       
\newcommand{\R}{\mathbb{R}}       
\newcommand{\N}{\mathbb{N}}       

\newcommand{\rme}{\mathrm{e}}
\newcommand{\rmi}{\mathrm{i}}

\newcommand{\tord}{\mathrm{tord}}

\newtheorem{theorem}{Theorem}[section]
\newtheorem{lemma}[theorem]{Lemma}
\newtheorem{corollary}[theorem]{Corollary}
\newtheorem{proposition}[theorem]{Proposition}
\newtheorem{definition}[theorem]{Definition}
\newtheorem{example}[theorem]{Example}
\newtheorem{remark}[theorem]{Remark}
\newtheorem{afirm}[theorem]{Claim}

\newtheorem{obs}[theorem]{Remark}
\newtheorem{corolario}[theorem]{Corollary}

\begin{document}
\title[LIPSCHITZ GEOMETRY OF MIXED POLYNOMIALS]{LIPSCHITZ GEOMETRY OF MIXED POLYNOMIALS}
\author[D. L. Medeiros]{Davi Lopes Medeiros$^{\dagger,\sharp}$}
\address{$\dagger$Departamento de Matem\'atica, Instituto de Ci\^encias Matem\'aticas e de Computa\c{c}\~ao (ICMC-USP). Avenida Trabalhador São-carlense, 400, Centro, 13566-590. São Carlos, SP, Brasil}
\email{davi\_lopes90@hotmail.com}

\author[J. E. Sampaio]{Jos\'e Edson Sampaio$^{\sharp}$}
\address{$\sharp$Departamento de Matem\'atica, Universidade Federal do Ceará (UFC), Campus do Pici, Bloco 914, Cep.~60455-760, Fortaleza-Ce, Brasil}
\email{edsonsampaio@mat.ufc.br}
\author[E. L. Sanchez Quiceno]{Eder Leandro Sanchez Quiceno$^{\star,\sharp}$}
\thanks{The first named author was supported by the Serrapilheira Institute (grant number Serra -- R-2110-39576) and by FAPESP (grant number FAPESP: 2024/13488-6). The second named author was partially supported by CNPq-Brazil grant 303375/2025-6 and supported by the Serrapilheira Institute (grant number Serra -- R-2110-39576). The last named author was supported by CNPq-Brazil grant 173313/2023-0 and by FAPESP (grant number FAPESP: 2023/11366-8).}
\address{$\star$Departamento de Matemática, Universidade Federal de S\~ao Carlos\\ Rodovia Washington Lu\'is, Km 235, CEP 13560-905, Caixa Postal 676 S\~ao Carlos- SP, Brasil}
\email{ederleansanchez@alumni.usp.br}
\begin{abstract}
We investigate the (ambient) bi-Lipschitz $V$-equivalence of two-variable mixed polynomials satisfying the Newton inner non-degeneracy condition. Concerning triviality, we show that ambient bi-Lipschitz $V$-triviality for families $\{f+\varepsilon\theta\}_{\varepsilon \in \mathbb{R}}$ is guaranteed when $f$ is semi-radially weighted homogeneous and the weighted radial degree of every monomial in $\theta$ is greater than the weighted radial degree associated with $f$. However, in the general case, we prove that it is not guaranteed, even though ambient topological $V$-triviality still holds.

For the classification problem, we define two simple metric links and prove that they suffice to determine bi-Lipschitz $V$-equivalence within the class of mixed polynomials that are $\Gamma_{\rm inn}$-nice. A key outcome is that neither the Newton boundary $\Gamma(f)$ nor the $C$-face diagram $\Gamma_{\mathrm{inn}}(f)$ constitutes an invariant of this equivalence for such mixed polynomials. To outcome this, we introduce new data extracted from the two face diagrams under consideration and prove that, under certain generic conditions, these data become fundamental invariants for the bi-Lipschitz equivalences. This provides a fundamental step toward a bi-Lipschitz classification of these mixed polynomials.
\end{abstract}

\maketitle
\vspace{-0.5cm}
\tableofcontents
\section{Introduction}
Mixed polynomials in $n$ variables are complex-valued functions defined as polynomials in $n$ complex variables $z = (z_1, \ldots, z_n)$ and their complex conjugates $\bar{z} = (\bar{z}_1, \ldots, \bar{z}_n)$. This class of functions is equivalent to the class of real polynomial maps $f: \mathbb{R}^{2n} \to \mathbb{R}^2$. The seminal work \cite{Oka2010} of Oka shows that this complex-variable approach has the advantage of bringing tools from the theory of holomorphic functions and complex algebraic geometry to investigate the topology of its zeros, such as Newton polyhedra, Newton and Khovanskii inner non-degeneracy, and resolution of singularities. The development of this methodology is today a very active research area, as shown by the recent works of Araújo dos Santos, Bode, Sanchez Quiceno, Espinel and Fukui in \cite{AraujoBodeSanchez,AraujoSanchez2024,Bode:part2,BodeSanchez2023,EspinelSanchez2025,Fukui2025}.

Specifically, the recent works \cite{AraujoBodeSanchez} and \cite{Bode:part2} develop a study of two-variable mixed polynomials under conditions imposed on their Newton boundary: the first one is Newton inner non-degeneracy, and the second condition is the so-called $\Gamma$-niceness. These conditions allow the authors to characterize the link $L_f:=\mathbb{S}^3_{\rho} \cap f^{-1}(0)$ (in the sense of Milnor in \cite{Milnor1968}) of a mixed polynomial $f$ in terms of the links associated with the regular part of the zeros of their face functions $f_{\Delta_i}$, where $\Delta_i$ is a one-dimensional face of the Newton boundary. They conclude that adding monomials whose support lies strictly above the Newton boundary does not change the isotopy type of $L_f$ or the topological type of  $V(f) := f^{-1}(0)$. This result is improved in \cite{EspinelSanchez2025}, where the authors show that the link of $f$ is determined solely by its restriction to a refined Newton $C$-diagram, $\Gamma_{\rm inn}(f)$. 

Such results on link triviality shows that the topology of the zeros of mixed polynomials and holomorphic functions has many similarities. Motivated by a theorem of Pham and Teissier (see \cite{Pham-Teissier}), which states that exists a meromorphic bi-Lipschitz map between two germs of complex analytic plane curves $X$ and $Y$ if only if $X$ and $Y$ are topologically equivalent (it was extended to any bi-Lipschitz map by Fernandes in \cite{Fernandes2003}, and to any ambient bi-Lipschitz map by Neumann and Pichon in \cite{neumann-pichon}), it is natural to ask if the same bi-Lipschitz equivalence holds for mixed polynomials. Unfortunately, the Lipschitz geometry of real surface germs obtained from a real analytic map from $\R^2$ to $\R^4$ can be very complicated, as it was shown by Birbrair, Mendes and Nuño-Ballesteros in \cite{a}. When we deal with subanalytic surface germs in $\R^4$, things can be very chaotic, and the extension of bi-Lipschitz maps between such germs to the ambient space is a very challenging problem. The work \cite{BBG} of Birbrair, Brandenbusky and Gabrielov leads to a surprisingly infinite classification of surface germs up to ambient bi-Lipschitz equivalence, even when they are topologically equivalent and outer bi-Lipschitz equivalent (Birbrair, Denkowski, Medeiros and Sampaio showed in \cite{microknot} that the same holds for germs whose outer and inner metrics are equivalent). Therefore, the relation between topology and Lipschitz geometry of mixed polynomials is not so obvious, and thus it deserves a deep investigation.

To achieve this goal, a good start is the special family of inner non-degenerate mixed polynomials introduced by Bode in \cite{Bode2019}. He constructed, for any given braid $B$, an inner non-degenerate and $z_1$-semiholomorphic polynomial (a mixed polynomial $f$ independent of $\bar{z}_1$) whose associated link is isotopic to the closure of the square of $B$ (closure of $B\cdot  B$). This result demonstrates that the class of inner non-degenerate mixed polynomials encompasses a significantly wider range of topological types. Further exploration of the diverse topological types realised by these mixed polynomials can be found in \cite{Bode:part2}. The family of mixed polynomials constructed by Bode is also radially weighted homogeneous mixed polynomials (radial mixed polynomials, for short) or, more generally, weighted homogeneous real polynomials. It follows from the paper \cite{Kerner-Mendes} by Kerner and Mendes that, outside of a horn neighbourhood of the obstruction locus $\Sigma_L(f)$ (\cite[Section 3.2]{Kerner-Mendes}) of a radial mixed polynomial $f$, the Lipschitz geometry of $V(f)$ coincides with that of their higher degree deformations. More precisely, they studied the ambient bi-Lipschitz $V$-triviality for families of real map germs $\mathcal{F}=\{f+\varepsilon \theta\}_{\varepsilon \in [0,1]}$, where  $$f=(f_1,\dots, f_m):(\R^n,0) \to (\R^m,0),\  n\geq m,$$  is weighted homogeneous of weight-type $(\boldsymbol{w};\boldsymbol{d})=(w_1, \dots, w_n
; d_1,\dots, d_m)$, and $\theta=(\theta_1,\dots, \theta_m)$ is any real map germ. They defined an obstruction locus $\Sigma_{L}(f)$, and proved that under the conditions $\Sigma_{L}(f)=\{0\}$ and $\operatorname{ord}_{\boldsymbol{w}}(\theta_i)> \operatorname{ord}_{\boldsymbol{w}}(f_i)=d_i, \ i=1,\dots ,m,$ the family $\mathcal{F}$ is ambient bi-Lipschitz $V$-trivial. In particular, this implies that the family is also bi-Lipschitz \(V\)-trivial. Another study concerning radial mixed polynomials appears in the work of Rabelo \cite{inacio2025}, which focuses on the Lipschitz classification problem of mixed Pham-Brieskorn polynomials.

We approach the bi-Lipschitz $V$-equivalence problem from the perspective of mixed polynomials in two variables that satisfy the inner non-degeneracy condition. We are motivated by the following two general questions: 
\textit{ 
\begin{itemize}
    \item[$\mathcal{Q}_1$.] Are families $\{f+\varepsilon \theta\}_{\varepsilon \in \R}$ of mixed polynomials (ambient) bi-Lipschitz $V$-trivial whenever $f$ is inner non-degenerate  and $\theta$ lies above or on the radial Newton boundary of $f$? 
    \item[$\mathcal{Q}_2$.] Is the radial Newton boundary of an inner non-degenerate mixed polynomial an invariant of (ambient) bi-Lipschitz $V$-equivalence?  
\end{itemize}}
The first question ($\mathcal{Q}_1$) has a negative answer in general. To reach this conclusion, we develop a study of contact order between the components of surfaces arising from inner non-degenerate mixed polynomials (see Proposition~\ref{nontangencycriteria}). The main idea for this negative answer arises from two conditions, \eqref{Eq:propertyA} and \eqref{Eq:propertyB}. We construct such a family of mixed polynomials in Example~\ref{contraexemploimportante}, which, although topologically $V$-trivial, is not bi-Lipschitz $V$-trivial. However, ($\mathcal{Q}_1$) is true if we assume that $f$ is semi-radially weighted homogeneous mixed polynomials (or semi-radial for short). These are those mixed polynomials that admit a decomposition into $f=f_P+\tilde{f}$, where $f_P$ satisfies $Sing(V(f_P))=\{0\}$ and it is radial with respect to a positive weight vector $P=(p_1,p_2)$ and a fixed radial-degree $d$ (referred to as radial-type $(P;d)$), and all monomials in the remainder $\tilde{f}$ have a radial-degree greater than $d$. For such a semi-radial mixed polynomial $f$, we define the quantity $k_f:=\frac{p_{1}}{p_{2}}$.	
\begin{theorem}\label{MainThm:1.1}
Let $\mathcal{F}=\{f+\varepsilon \theta\}_{\varepsilon\in \R}$ be a family of mixed polynomials satisfying: 
\begin{enumerate}
\item[(i)] $f$ is semi-radial of radial-type $(P;d)$, 
\item[(ii)] \(d(P;\theta)\geq d\).
\end{enumerate}
If the inequality in (ii) is strict,  then the family $\mathcal{F}$ is (ambient) bi-Lipschitz trivial. If we have equality in (ii), then the family is (ambient) bi-Lipschitz trivial along a small neighbourhood of the origin.
 \end{theorem}
As a consequence of this theorem, in Corollary~\ref{cor:bilipVequivsemi-radialprincipalpart}, we prove that every semi-radial mixed polynomial of radial-type \((P; d)\) is ambient bi-Lipschitz \(V\)-equivalent to its principal part \(f_P\). A second consequence is that the condition $\Sigma_{L}=\{0\}$ is not necessary for ambient bi-Lipschitz $V$-triviality; this is demonstrated by the class of semi-radial mixed polynomials of Type-III (see Definition~\ref{deftypes} and Lemma~\ref{radial-obst}).
\vspace{0.2cm}

 The answer to the second question ($\mathcal{Q}_2$) is also negative in general. We demonstrate this by studying two specific simple metric links: the 1-braid closure (Definition~\ref{1-braidclosure}) and the non-tangent Hopf-link (Definition~\ref{Def:Lip-transverse-Hopf-link}). We prove that the bi-Lipschitz equivalence class for these metric links does not depend on either their Newton boundary $\Gamma(f)$ or the refined  $C$-face diagram $\Gamma_{\rm inn}(f)$.
 \begin{theorem}\label{1.4}
 Let \( f \) and \( g \) be inner non-degenerate mixed polynomials that are $\Gamma_{\mathrm{inn}}$-nice and such that their links \( L_f \) and \( L_g \) satisfy one of the following conditions:
\begin{enumerate}
		\item[(i)] Both \( L_f \) and \( L_g \) are empty.
	\item[(ii)] Both \( L_f \) and \( L_g \) are metric 1-braid closures.
	\item[(iii)] Both \( L_f \) and \( L_g \) are non-tangent Hopf-links.
    \end{enumerate}
Then, \( f \) and \( g \) are (ambient) bi-Lipschitz \( V \)-equivalent.
\end{theorem}
The question ($\mathcal{Q}_2$) is shown to be negative in general through counterexamples involving 1-braid closures (Example~\ref{1braidclosures_type_I_II}) and non-tangent Hopf-links (Example~\ref{firstsemiradiaexamples1}). Counterexamples related to the $C$-face diagram $\Gamma_{\rm inn}(f)$ are also provided in Example~\ref{ex:typeiinotmetrics}. Despite this, we demonstrate that for semi-radial mixed polynomials (excluding those with simple metric links), certain data from the principal part is invariant. We define types (Type I-III), which are invariant under bi-Lipschitz $V$-equivalence, and show that the coefficient $k_f$ is also preserved (up to a multiplicative inverse).  
\begin{theorem}\label{MainThm:1.3}
Let \(f\) and \(g\) be semi-radial mixed polynomials, and  assume that each link $L_f$ and $L_g$ is neither empty nor a metric 1-braid closure. If $f$ is  bi-Lipschitz $V$-equivalent to $g$, then:
\begin{enumerate}
\item[(i)]  If $f$ is Type-I, then $g$ must be Type-I and \( k_f = k_g \text{ or } k_f = k_g^{-1}\).
\item[(ii)] If additionally their links are not non-tangent Hopf-links, then: 
\begin{itemize}
\item[(a)] If $f$ is Type-II, then $g$ must be Type-II. 
\item[(b)]  If $f$ is Type-III, then $g$ must be Type-III and \( k_f = k_g \text{ or } k_f = k_g^{-1}\).
\end{itemize}
\end{enumerate}
\end{theorem}

 We also show that, by restricting to certain wider classes of inner non-degenerate mixed polynomials, relevant data from $\Gamma_{\mathrm{inn}}(f)$ become invariant. We define $k_{i,f}:=\frac{p_{i,1}}{p_{i,2}}$, where $P_{i}=(p_{i,1},p_{i,2})$ belongs to the set of weight vectors $\mathcal{P}_{\mathrm{inn}}(f)=\{P_1,\dots, P_{N_{f}}\}$ associated with the 1-faces of $\Gamma_{\mathrm{inn}}(f)$. Using these values, we construct a contact data set $\mathcal{NC}(f) := \{N_f, \mathcal{C}(f)\}$, where the set  $\mathcal{C}(f)$ is detailed in Equation~\eqref{contactdata}. In general, we observe that $\mathcal{NC}(f)$ is not necessarily an invariant under bi-Lipschitz $V$-equivalence (see Example~\ref{1braidclosures_type_I_II}). To address this and other technical issues, we introduce two conditions, \eqref{Eq:propertyC} and \eqref{Eq:propertyD}, formulated as broadly as possible while ensuring the invariance of $\mathcal{C}(f)$. For the invariance of $\mathcal{NC}(f)$, we restrict our consideration to the class of $\Gamma_{\mathrm{inn}}$-true mixed polynomials. Building on the previous results of test arcs by Fernandes (see \cite{Fernandes2003}) and tangent cones by Sampaio (see \cite{Sampaio:2016}), we prove the following: 
\begin{theorem}
Let $f$ and $g$ be inner non-degenerate mixed polynomials that are $\Gamma_{\mathrm{inn}}$-nice and satisfy \eqref{Eq:propertyC} and \eqref{Eq:propertyD}. If $f$ is bi-Lipschitz $V$-equivalent to $g$, then $\mathcal{C}(f)=\mathcal{C}(g).$ Furthermore, if additionally $f$ and $g$ are $\Gamma_{\mathrm{inn}}$-true, then  $\mathcal{NC}(f)=\mathcal{NC}(g).$ 
\end{theorem}

This paper is organised into several sections. Section~\ref{Sect:IND} establishes the preliminaries needed for understanding the paper, making it self-contained. Subsection \ref{subsect2.1} introduces the concepts related to mixed polynomials and the Newton polygon, as well as the definitions of convenient polynomials and inner non-degeneracy, which are important to investigate the topology of the zeros of such polynomials. Subsection~\ref{Subsect:topVtriviality} presents the important notion of semi-radial mixed polynomials and shows some results on the topological \(V\)-triviality of families of mixed polynomials \(\{f + \varepsilon \theta\}_{\varepsilon \in I}\) obtained in \cite{AraujoBodeSanchez} and \cite{EspinelSanchez2025}, formulated in terms of inner vertices of \(\Gamma_{\mathrm{inn}}(f)\). In Subsection~\ref{nested}, we review the concept of nested links associated with the mixed polynomials discussed herein, their parametrizations, and how we can see them as braids. Subsection~\ref{Subsec:Lipschitz-geometry} provides the basic definitions in Lipschitz geometry and some of its main tools and properties, such as half-branches, contact order, H\"older triangles, and horns. Finally, Subsection~\ref{subsec:2.5} is devoted to showing the necessary theorems in ambient Lipschitz geometry and the extension of such maps, as well as their relation with tangent cones.

In Section~\ref{somepreparationresults}, we present auxiliary results that are essential for the main theorems and are also applicable in a more general context. Subsequently, Subsection~\ref{obstructionlocus} discusses the notion of singular locus $\Sigma_{L}(f)$ given by Kerner and Mendes in \cite{Kerner-Mendes}, and gives a criterion in mixed polynomials for when it must be trivial. Subsection~\ref{separatinghorns} aims to address the challenge of obtaining an ambient bi-Lipschitz triviality in the case $\Sigma_{L}(f)\ne \{0\}$. It is devoted to establishing several results in which we partition the ambient space into horn neighbourhoods in order to deal with each part separately later. In Subsection~\ref{Lipcontactcriterion}, we study the zeros of a broad class of mixed polynomials by calculating estimates of the contact orders of their components. We finish this section with Subsection~\ref{unionlema}, where we begin to investigate the tangent cone of several mixed polynomials and prove Lemma~\ref{Lemma:Union}, which allows us to extend a family of bi-Lipschitz maps to the ambient, by deforming them to the identity around some horn neighbourhoods. 

In Section \ref{Sect:Bi-LipVtriviality}, we focus on the ambient bi-Lipschitz geometry of semi-radial mixed polynomials. Subsection \ref{sub:ambientbilipchitztriviality} addresses the ambient bi-Lipschitz V-trivialization, defining the three types of semi-radial mixed polynomials, proving Propositions \ref{radial-trivial} and \ref{Prop:biliptrivialitytypeIII} and using them to establish Theorem~\ref{MainThm:1.1}. Subsection~\ref{Subsect:necessarybi_lipVequivsemi-radial} examines the relation between the zero set of semi-radial mixed polynomials and their tangent cones, culminating in their ambient bi-Lipschitz equivalence in most cases.

Section~\ref{Sect:1-braids} aims to obtain ambient bi-Lipschitz triviality beyond semi-radial mixed polynomials. To achieve this, we introduce the definitions of metric 1-braid closure and non-tangent Hopf-link, which are links whose topology and metric are simple. Then, we prove Proposition \ref{1.4}, which implies bi-Lipschitz V-triviality in families of mixed polynomials whose zeros have links with such properties.

Section~\ref{Sect:necessarybi_lipVequiv} explores necessary conditions for bi-Lipschitz \(V\)-equivalence when the link is neither a 1-braid closure nor a non-tangent Hopf-link. Subsection \ref{subsec6.1} shows that the type of such semi-radial mixed polynomials is a Lipschitz invariant, culminating in Theorem~\ref{MainThm:1.3}. This section concludes by complementing the preceding result by defining the sets \(\mathcal{C}(f)\) and \(\mathcal{NC}(f)\) in Subsection \ref{subsec6.2}, where we introduce conditions \eqref{Eq:propertyC} and \eqref{Eq:propertyD}, and prove Theorem~\ref{MainThm:1.4}. Finally, Proposition~\ref{suficientcondforCandD} highlights the significance of such a theorem. Section \ref{section7} presents two examples that explicitly show why the Lipschitz geometry of mixed polynomials is different from its topology, contrasting with the holomorphic case studied by Pham, Teissier, Fernandes, Neumann and Pichon in \cite{Pham-Teissier,Fernandes2003,neumann-pichon}.

\vspace{0.05cm}


\section{Preliminaries}\label{Sect:IND}

\subsection{Mixed polynomials and Newton boundaries}\label{subsect2.1}
In this section, we review some background on mixed singularities and mixed hypersurfaces. A more detailed account of the concepts can be found in \cite{AraujoBodeSanchez} and in a generalised way in \cite{Oka2010,EspinelSanchez2025}.

\begin{definition}
A \textbf{mixed polynomial in $\C^2$} is a complex valued map \(f:\C^2\to \C\) that can be represented as 
\[f(\bm{z})=\sum_{j=1}^{n} c_j \bm{z}^{\nu_j}\bar{\bm{z}}^{\mu_j}, \]
for \(\ c_{j} \in \C^*\), \(\nu_j , \mu_j \in \mathbb{Z}_{\ge0}^2\) with \(\nu_{j,i}+\mu_{j,i} \ge1\), for each $j=1,\dots,n$ and $i=1,2$. Here, \(\bm{z} = (u,v)\),  \(\bar{\bm{z}}= (\bar{u},\bar{v})\), \(\bm{z}^{\nu_j} = {u}^{\nu_{j,1}}v^{\nu_{j,2}}\) and \(\bar{\bm{z}}^{\mu_j}=\bar{u}^{\mu_{j,1}}{\bar{v}}^{\mu_{j,2}}\) (notice that $f(0)=0$).
\end{definition}

\begin{remark}\label{Remark:homogeneous-real}
    Any mixed polynomial $f$ can be seen as a real polynomial map $(\re(f),\im(f))$ in four variables $\re(u),\im(u),\re(v),\im(v)$ via the identifications $f(\bm{z}) =  \re (f(\bm{z})) +\rmi \im (f(\bm{z}))$ and $z=(\re(u)+\rmi \im(u),\re(v)+\rmi \im(v))$. Moreover, we can transform any real polynomial map 
\(F: (\R^4,0) \to (\R^2,0)\) into a mixed polynomial by defining it with $$f(u,v)=F_1(\boldsymbol{x})+\rmi F_2(\boldsymbol{x}), \ \boldsymbol{x}=\left(\frac{u+\bar{u}}{2},\frac{u-\bar{u}}{2\rmi},\frac{v+\bar{v}}{2},\frac{v-\bar{v}}{2\rmi}\right).$$
\end{remark}

\begin{definition}
We say that the mixed polynomial $f$ is  \textbf{radially weighted homogeneous} of radial-type \((P;d_r)\), where \(P=(p_1,p_2) \in {(\mathbb{Z}_{>0})}^2\) and  $d_r \in \mathbb{Z}_{\ge 0}$ (or radial mixed polynomial, for short), if, for all $j=1,\dots, n$, we have
\begin{equation}\label{radialhomog}
p_1(\nu_{j,1}+\mu_{j,1})+p_2(\nu_{j,2}+\mu_{j,2})=d_r.
\end{equation}  
The weighted sum in Equation~\eqref{radialhomog} is called the  \textbf{radial degree} with respect to $P$ of the monomial
\( c_j z^{\nu_j} \bar{z}^{\mu_j}\), and it is denoted by \(\mathrm{rdeg}_P (c_jz^{\nu_j} \bar{z}^{\mu_j})\). Equivalently,  the mixed polynomial \(f\) is radially weighted homogeneous of radial-type $(P;d_r)$, if \(f\) satisfies 
\begin{equation*}
 f(\lambda^{p_1} u, \lambda^{p_2} v)=\lambda^{d_r}f (u,v), \ \lambda \in \R _{>0}.
 \end{equation*}
\end{definition}

\begin{remark}
    The radial mixed polynomials constitute a special class of weighted homogeneous polynomial map \(F:\R^4 \to \R^2\). Here, $F$ is homogeneous of weight-type $(\boldsymbol{w};\boldsymbol{d})=(w_1,w_2,w_3,w_4;d_1,d_2)$ if, for every $\boldsymbol{x}=(x_1,x_2,x_3,x_4) \in \R^4$ and $\lambda \in \mathbb{R}_{>0}$.
\begin{equation*}\label{realhomog}
 F(\lambda^{w_1} x_1,\lambda^{w_2} x_2, \lambda^{w_3} x_3, \lambda^{w_4} x_4)=(\lambda^{d_1} F_1(\boldsymbol{x}), \lambda^{d_2} F_2(\boldsymbol{x})),
 \end{equation*} 
where \(\boldsymbol{x}=(x_1,x_2,x_3,x_4)\),  \(w_j\in \mathbb{Q}_{>0}, d_1,d_2 \in \N \).   
In this case, the real polynomial map $F$ is called of  \textbf{weighted homogeneous} of weight-type \((\boldsymbol{w};\boldsymbol{d})\). If $f$ is a radially weighted homogeneous of radial-type $(p_1,p_2;d_r)$, then $$F=(\re (f),\im (f)):(\R^4,0) \to (\R^2,0)$$ is  weighted homogeneous  of weight-type $(p_1,p_1,p_2,p_2; d_r,d_r)$.
\end{remark}

\begin{remark}
While we can transform any real polynomial map 
\(F: (\R^4,0) \to (\R^2,0)\) into a mixed polynomial (see Remark \ref{Remark:homogeneous-real}), it is important to notice that not all weighted homogeneous real polynomial map $F$ can be seen as a radial mixed polynomial. For example,
\(F(x_1,x_2,x_3,x_4)=(x_1,x_1^2+x_2^2+x_3^2+x_4^2)\) is weighted homogeneous of weight-type \( (1,1,1,1;1,2)\), but it does not yield a radial mixed polynomial, since $1=d_1 \ne d_2=2$.  
\end{remark}

The homogeneity property of radial mixed polynomials allows us to derive several properties of $f$ and $V(f)$. However, radial mixed polynomials represent only a limited class in the category of mixed polynomials. To explore a broader and well-behaved class of mixed polynomials, we utilise the radial Newton polygon of 
$f$ (introduced in \cite{Oka2010}) and the inner non-degeneracy condition (introduced in \cite{AraujoBodeSanchez} and amplified in \cite{EspinelSanchez2025}):

\begin{definition}
Given a set $S\subseteq (\R_{\ge0})^2$ of points with rational coordinates, we define the \textbf{rational diagram of $S$} (denoted by \(\Gamma_{+}(S)\)) as the convex hull of the set
\begin{equation*}
\bigcup_{(p,q)\in S} \{(p+t_1,q+t_2) \mid t_1, t_2 \in \R_{\ge 0}\}.
\end{equation*}

The \textbf{rational boundary of $S$}, denoted as $\Gamma(S)$, is defined as the union of the compact faces of $\Gamma_{+}(S)$. A \textbf{face \(\Delta\) 
of \(\Gamma(S)\)} is a compact face of $\Gamma_{+}(S)$. The rational diagram $\Gamma_+(S)$ is \textbf{convenient} if $\Gamma(S)$ intersects both the $x$-axis and the $y$-axis. A \textbf{$C$-face diagram} (or $C$-diagram, for short) is a convenient rational boundary.
\end{definition}

\begin{remark}
Notice that, in the two-variable case, the rational boundary is formed by 1-faces (edges) and 0-faces (vertices).
Moreover, every $C$-diagram contains exactly one 0-face in the $x$-axis and one 0-face in the $y$-axis.
\end{remark}

\begin{definition}
Given a $C$-diagram $\Gamma_+(S)$, we say that a face \(\Delta\) of $\Gamma_+(S)$ is an \textbf{inner face} if it is not 
contained in the $x$-axis or in the $y$-axis. Given a positive weight vector $P =(p_1, p_2)$, with $p_1, p_2 \in \mathbb{Z}_{>0}$, we define $\Delta (P; \Gamma)$ as the face of $\Gamma(S)$ where $l_P(\nu) = p_1 \nu_1+p_2 \nu_2, \ \nu=(\nu_1,\nu_2) \in \Gamma(S),$ attains its minimal value. We denote this minimal value of $l_P$ by $d(P; \Gamma)$.
\end{definition}

\begin{definition}
    Given a mixed polynomial $f(\bm{z})=\sum c_{\nu, \mu} z^{\nu}\bar{z}^{\mu}$, we define the \textbf{radial Newton polygon of \(f\)}, denoted by \(\Gamma_+(f)\), as the set $\Gamma_+(\operatorname{supp}(f))$, where $$\operatorname{supp}(f):=\{(\nu_1+\mu_1,\nu_2+\mu_2) \mid \nu=(\nu_1,\nu_2) \, ; \, \mu=(\mu_1,\mu_2) \, ; \, c_{\nu,\mu} \ne 0\}.$$
    The \textbf{Newton boundary of $f$}, denoted as $\Gamma(f)$, is defined as $\Gamma(\operatorname{supp}(f))$. A \textbf{face \(\Delta\) of \(\Gamma(f)\)} is a compact face of $\Gamma_{+}(f)$. A mixed polynomial $f$ is called \textbf{$u$-convenient} if $\Gamma(f)$ intersects the $x$-axis, and $f$ is called \textbf{$v$-convenient} if $\Gamma(f)$ intersects the $y$-axis. If $f$ is both $u$-convenient and $v$-convenient, then $f$ is called \textbf{convenient}.
    
    For a positive weight vector $P =(p_1, p_2)$, we define $\Delta (P; f)$ as the face of $\Gamma(f)$ where $l_P(\nu) = p_1 \nu_1+p_2 \nu_2, \ \nu=(\nu_1,\nu_2) \in \Gamma(f),$ attains its minimal value. We denote such minimal value by $d(P; f)$. Associated with the vector $P$ (and, alternatively, with the face $\Delta=\Delta (P; f)$), we define the  \textbf{face function} $f_P(\bm{z})$ by \begin{equation*}f_P(\bm{z})=f_{\Delta}(\bm{z})= f_{\Delta(P;f)}(\bm{z}):=\sum_{\nu+\mu \in \Delta(P;f)} c_{\nu,\mu}z^{\nu}\bar{z}^{\mu}.
    \end{equation*}
    Notice that the face function $f_P (\bm{z})$ is a radial mixed polynomial of radial-type $(P;d(P;f))$. Finally, associated with $\Gamma(f)$, we define the  \textbf{Newton principal part}  $f_{\Gamma}(\bm{z})$ as
    \begin{equation*}
    f_{\Gamma}(\bm{z}):= \sum_{\nu+\mu \in \Gamma(f)} c_{\nu,\mu}z^{\nu}\bar{z}^{\mu}.
    \end{equation*}
\end{definition}

 For \(I \subseteq \{1,2\} \) and  \(\mathbb{K}= \R \text{ or } \C\), consider
\[\mathbb{K}^I:=\{(x_1,x_2)\in \mathbb{K}^2 \mid x_i=0 \text{ if } i \not\in I\} \simeq \mathbb{K}^{|I|}\]
 \[(\mathbb{K}^*)^I:= \left\{x \in \mathbb{K}^2 \mid \prod _{i \in I} x_i \neq 0 \right \} \cap \mathbb{K}^I \simeq (\mathbb{K}^* )^{|I|}.\]

\begin{definition}[{\cite{EspinelSanchez2025}}]\label{caracterizationinner}
We say that a mixed polynomial $f$ is \textbf{inner Khovanskii non-degenerate} (or \textbf{inner non-degenerate} - IND for short\footnote{This name was also used in \cite{AraujoBodeSanchez} for a non-degeneracy condition that is equivalent to this definition \cite[Prop. 2.11]{EspinelSanchez2025}.}) if there is a $C$-diagram $\Gamma$ such that the following holds:
\begin{enumerate}
	\item[(i)] no point of \(\operatorname{supp}
	(f)\) lies below \(\Gamma\), i.e., \(\operatorname{supp}(f) \subseteq \Gamma+\R^n_{\geq 0}\); 
	\item[(ii)] for every inner face \(\Delta\) of \(\Gamma\) and for every nonempty subset \(I\) of \{1,2\}, 
	\[ \Delta \cap \R^I\neq  \emptyset \implies Sing (V(f_\Delta)) \cap (\C^*)^{I}= \emptyset.\]
\end{enumerate} 
\end{definition}
The $C$-diagram satisfying (i) and (ii) is not necessarily unique; see Example~\ref{exemplo-poligono-newton}. However, if the mixed polynomial $f$ is IND, we are interested in a specific diagram, denoted by $\Gamma_{\mathrm{inn}}(f)$, which captures the information of the link of the singularity and facilitates the presentation of our results. Henceforth, when we state that $f$ is IND, the diagram satisfying properties (i) and (ii) will be understood to be $\Gamma_{\mathrm{inn}}(f)$.

We define for a mixed polynomial $f$ and a $C$-diagram $D$ the set 
 $$I_{\mathrm{nc}}(f_D):=\{i \in \{1,2\}: \mathrm{supp}(f_D) \cap \R^{\{i\}}= \emptyset\}.$$
\begin{definition}[\cite{EspinelSanchez2025}]\label{def:gammainn}
    Given an IND mixed polynomial $f$, let $\Gamma_{\mathrm{inn}}(f)$ be the $C$-diagram satisfying Definition~\ref{caracterizationinner} and 
    \begin{itemize}
        \item[(i)]  for $i \in I_{\rm{nc}}(f_{\Gamma_{\mathrm{inn}}(f)})$ and any inner face $\Delta(P;\Gamma_{\mathrm{inn}}(f)), \ P \in \mathbb{Q}^n_{>0}$, of $\Gamma_{\mathrm{inn}}(f)$ satisfying  \(\Delta(P;\Gamma_{\mathrm{inn}}(f)) \cap \R^{\{i\}}\neq \emptyset,\) we get 
$    p_{i} \leq p_{j},\text{ for all } j.$
        \item[(ii)]  if $D$ is another $C$-diagram satisfying (i) and  Definition \ref{caracterizationinner}  then  $$D+(\R_{\geq 0})^{2} \subseteq \Gamma_{\mathrm{inn}}(f)+(\R_{\geq 0})^{2}.$$
    \end{itemize}
    
 Define $\mathcal{P}_{\mathrm{inn}}(f)$ as a set of weight vectors $P_i=(p_{i,1},p_{i,2})$, with  $\gcd(p_{i,1},p_{i,2})=1$, associated with the 1-faces of \(\Gamma_{\mathrm{\mathrm{inn}}}(f)\): 
\begin{equation*}\label{1weightvectors}
\mathcal{P}_{\mathrm{inn}}(f):=\{ P=(p_{1},p_{2}) \mid \gcd(p_1,p_2)=1 \ \& \ \Delta(P;\Gamma_{\mathrm{inn}}(f)) \text{ is a 1-face of } \Gamma_{\mathrm{inn}}(f)\}.
\end{equation*}
Let $\mathcal{P}_{\mathrm{inn}}(f)=\{P_1,P_2,\dots,P_{N_f}\}$ and, for $i=1,2,\dots,N_f$, define \(k_{i,f}:=\frac{p_{i,1}}{p_{i,2}}\) (this number is the slope of the line passing by $\Delta(P_i;\Gamma_{\mathrm{inn}}(f))$). We always consider the elements of $\mathcal{P}_{\mathrm{inn}}(f)$ ordered in such a way that $k_{1,f}>k_{2,f}>\ldots>k_{N_f,f}$.
\end{definition}

\begin{example}\label{exemplo-poligono-newton}
Consider the mixed polynomial
\[
f(\bm{z}) = u^8 + v^3u^2 + \bar{v}^6u^2 + \bar{v}^5u + v^4\bar{v}^4.
\]
The support of \( f \) is
\[
\mathrm{supp}(f) = \{(8,0), (2,3), (2,6), (1,5), (0,8)\}.
\]

The Newton polygon of \( f \) is shown in Figure~\ref{newtonboundaryholomorphicq}. We observe that \( \Gamma(f) \) forms a $C$-diagram with five inner faces (two of them being 0-faces). To prove that \( f \) is IND with respect to \( \Gamma(f) \), that means $f$ satisfies Definition~\ref{caracterizationinner} for $\Gamma(f)$, it suffices to verify condition (ii) of Definition~\ref{caracterizationinner} for each inner face. The inner faces that are vertices trivially satisfy condition (ii). It remains to check this condition for the 1-faces associated with $P_1=(3,1)$, $P_2=(2,1)$, and $P_3=(1,2)$.  
 The face functions associated with these faces are:
\[
\begin{aligned}
f_{\Delta(P_1;f)}(\bm{z}) &= \bar{v}^5u + v^4\bar{v}^4, \\
f_{\Delta(P_2;f)}(\bm{z}) &= v^3u^2 + \bar{v}^5u, \\
f_{\Delta(P_3;f)}(\bm{z}) &= u^8 + v^3u^2.
\end{aligned}
\]
We illustrate the argument for \( \Delta(P_2;f) \); the other two cases are analogous. Since \( f_{\Delta(P_2;f)} \) is $u$-semiholomorphic (i.e. independent of $\bar{u}$), according to \cite{AraujoSanchez2024}, the singular locus \( Sing(V(f)) \) is the solution of
\[
f = f_u = f_v\overline{f_v} - f_{\bar{v}}\overline{f_{\bar{v}}} = 0.
\]

For \( f_{\Delta(P_2;f)} \), this yields 
\[
v^3u^2 + \bar{v}^5u = 2v^3u + \bar{v}^5=9(v\bar{v})^2(u\bar{u})^2-25(v\bar{v})^4(u\bar{u}) = 0.
\]

Hence, from $v^3u^2 + \bar{v}^5u = 2v^3u + \bar{v}^5=0$, we get 
\begin{equation}\label{eq:singvp2}
Sing(V(f_{\Delta(P_2;f)})) = \{(u,v) \in \C^2 : v = 0\}.
\end{equation}

Since \( \Delta(P_2;f) \cap \R^{\{1,2\}} \neq \emptyset \) and from Equation \eqref{eq:singvp2}, we clearly have
\[
Sing(V(f_{\Delta(P_2;f)})) \cap (\C^*)^{\{1,2\}} = \emptyset,
\]
we conclude that \( f \) is IND and convenient.

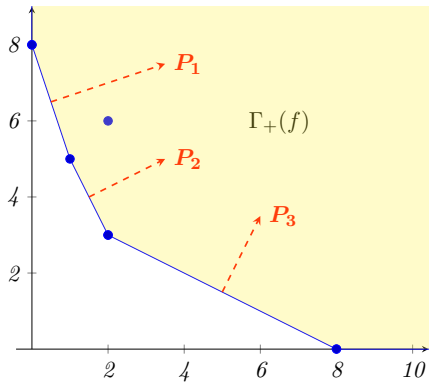
\begin{figure}[h!]
 \centering
 \begin{tikzpicture}[scale=0.8]
\begin{axis}[axis lines=middle,axis equal,yticklabels={0,,2,4,6,8,10}, xticklabels={0,,2,4,6,8,10},domain=-10:10,     xmin=0, xmax=10,
                    ymin=0, ymax=9,
                    samples=1000,
                    axis y line=center,
                    axis x line=center]
\addplot coordinates{(11,0)(8,0) (2,3) (1,5) (0,8) (0,10)};
\draw[line width=2pt,red,-stealth,thick,dashed](50,15)--(60,35)node[anchor=west]{$\boldsymbol{P_3}$};
\draw[line width=2pt,red,-stealth,thick,dashed](5,65)--(35,75)node[anchor=west]{$\boldsymbol{P_1}$};
\draw[line width=2pt,red,-stealth,thick,dashed](15,40)--(35,50)node[anchor=west]{$\boldsymbol{P_2}$};
 \filldraw[black] (50,110) node[anchor=north ] {$\Gamma_+(f)$};
  \filldraw[black] (65,65) node[anchor=north ] {$\Gamma_+(f)$};
\filldraw[blue] (20,30) circle (2pt) node[anchor=south west] {};
\filldraw[blue] (20,60) circle (2pt) node[anchor=south west] {};
\filldraw[blue] (0,80) circle (2pt) node[anchor=south west] {};
\filldraw[blue] (80,0) circle (2pt) node[anchor=south west] {};
\fill[yellow!90,nearly transparent] (0,80) --(10,50)--(20,30) -- (80,0) -- (260,0) -- (260,300) -- (0,300) --cycle;
\end{axis}
\end{tikzpicture}
\caption{The Newton polygon $\Gamma_+(f)$ in Example \ref{exemplo-poligono-newton}.
\label{newtonboundaryholomorphicq}}
\end{figure}

\vspace{0.2cm}

Consider the $C$-diagram \( \Gamma \) constructed from the set \( S = \{(8,0), (2,3), (0,7)\} \), as illustrated in Figure~\ref{diagramf}. We can verify that \( f \) is IND with respect to \( \Gamma \) as well.

The $C$-diagram \( \Gamma \) has three inner faces. From the previous case, we already know that \( f \) satisfies condition (ii) of Definition~\ref{caracterizationinner} for the inner face that is a vertex, as well as for \( \Delta(Q_2; \Gamma) = \Delta(P_3; f) \), where $Q_2=P_3$.

It remains to check the inner face \( \Delta(Q_1; \Gamma) \), where $Q_1=P_2$, for which the following holds:
\[
\Delta(Q_1; \Gamma) \cap \mathbb{R}^{\{1,2\}} \neq \emptyset \quad \text{and} \quad \Delta(Q_1; \Gamma) \cap \mathbb{R}^{\{1\}} \neq \emptyset.
\]

Since 
\[
f_{\Delta(Q_1; \Gamma)}(\bm{z}) = f_{\Delta(P_2; f)}(\bm{z}),
\]
 it follows from Equation~\eqref{eq:singvp2} that
\[
Sing(V(f_{\Delta(Q_1; \Gamma)})) \cap (\mathbb{C}^*)^{\{1,2\}}=Sing(V(f_{\Delta(Q_1; \Gamma)})) \cap (\mathbb{C}^*)^{\{1\}} = \emptyset.
\]
Therefore, \( f \) is IND with respect to  \( \Gamma \). We also observe that
\[
\Gamma_{\mathrm{inn}}(f) = \Gamma \quad \text{and} \quad \mathcal{P}_{\mathrm{inn}}(f) = \{Q_1, Q_2\}.
\]

    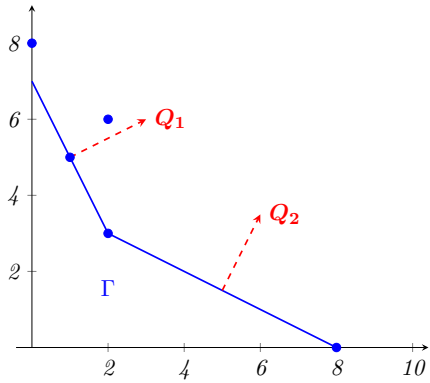
\begin{figure}[h!]
 \centering
 \begin{tikzpicture}[scale=0.8]
\begin{axis}[axis lines=middle,axis equal,yticklabels={0,,2,4,6,8,10}, xticklabels={0,,2,4,6,8,10},domain=-10:10,     xmin=0, xmax=10,
                    ymin=0, ymax=9,
                    samples=1000,
                    axis y line=center,
                    axis x line=center]
                    \addplot[mark=none, thick, blue] coordinates{(8,0) (2,3) (0,7)};
\addplot[only marks, mark=*, blue, mark size=2pt] coordinates {(8,0)};
\draw[line width=2pt,red,-stealth,thick,dashed](50,15)--(60,35)node[anchor=west]{$\boldsymbol{Q_2}$};
\draw[line width=2pt,red,-stealth,thick,dashed](10,50)--(30,60)node[anchor=west]{$\boldsymbol{Q_1}$};
  \filldraw[blue] (20,20) node[anchor=north ] {$\Gamma$};
\filldraw[blue] (20,30) circle (2pt) node[anchor=south west] {};
\filldraw[blue] (10,50) circle (2pt) node[anchor=south west] {};
\filldraw[blue] (0,80) circle (2pt) node[anchor=south west] {};
\filldraw[blue] (20,60) circle (2pt) node[anchor=south west] {};
\end{axis}
\end{tikzpicture}
\caption{The $C$-diagram $\Gamma$ in Example \ref{exemplo-poligono-newton}.
\label{diagramf}}
\end{figure}
\end{example}

\subsection{Link constancy in families of mixed polynomials}\label{Subsect:topVtriviality}
It was proved in \cite{AraujoBodeSanchez} that IND mixed polynomials  define  
germs of isolated surface singularities. Therefore, by the local conical structure theorem (see \cite{Milnor1968}), associated with $f$ (or with the surface $V(f)$) we have a well-defined smooth link $L_f:= \mathbb{S}_{\rho}^3 \cap V(f)$, called the  \textbf{link of the singularity of $f$}. It does not depend on $\rho>0$ if it is small enough. We will see that, for certain smooth families $f_{\varepsilon} : \C^2 \to \C$, $\varepsilon \in [0,1]$ (that is, there exist a neighbourhood $U$ of $0$ in $\C^2$ and a smooth map $F: U \times [0,1] \to \C$ such that $F(0, \varepsilon) = 0$ and $f_\varepsilon(\bm{z}) = F(\bm{z}, \varepsilon)$, for all  $
\varepsilon \in [0,1]$, and for all $\bm{z} \in U$), the link type of each member remains the same. 

  \begin{definition}[Link-constancy]
Let $K \subseteq \R$ be a connected set containing the origin. Suppose that for all $\varepsilon \in K$ we have that $f_\varepsilon$ has a well-defined link (in the sense of Milnor \cite{Milnor1968}). We say that the family \(\{f_\varepsilon\}_{\varepsilon \in K}\)  is \textbf{link-constant along $K$} if for all $\varepsilon \in K$, $L_{F_\varepsilon}$ is ambient isotopic to $L_{F_0}$, that is, there exists a continuous map	$
		H : \mathbb{S}_\epsilon \times [0,1] \to \mathbb{S}_\epsilon$, where $L_{F_0,\epsilon}$ and  $L_{F_\varepsilon,\epsilon}$ are well-defined for sufficiently small $\epsilon>0$, 
		such that for each $t \in [0,1]$, the map $H_t := H(\cdot, t)$ is a homeomorphism,
		$H_0 = \text{id}$, and
		$
		H_1(L_{F_0}) = L_{F_\varepsilon}.
		$    If $K = \R$, we simply say that $F$ is link-constant.       
 
\end{definition}     
  A  simple class of IND mixed polynomials is the semi-radial mixed polynomials.     
\begin{definition}\label{definitionsemi-radial}
\noindent We say that a mixed polynomial is \textbf{semi-radially weighted homogeneous} (for short, semi-radial) if there exists a positive  weight vector \( P=(p_1,p_2) \) such that 
\[ f(\bm{z}) = f_{P}(\bm{z}) + \tilde{f}(\bm{z}), \]
where \( d(P;\tilde{f}) > d(P;f) \) and  Sing(V(\( f_P \)))=\{0\}. In this case, \( f \) is called \textbf{semi-radial} of radial-type \( (P; d(P; f)) \) and we associate with each such $f$ a number \( k_f := \frac{p_1}{p_2} \).
\end{definition}

In \cite{EspinelSanchez2025}, Espinel and Sanchez Quiceno proved that certain deformations of semi-radial mixed polynomials are link-constant.

\begin{theorem}\label{thm:link-trivialitysemi}
Let $\mathcal{F}=\{f+\varepsilon \theta\}_{\varepsilon\in \R}$ be a family of mixed polynomials satisfying: 
\begin{enumerate}
\item[(i)] $f$ is semi-radial of radial-type $(P;d)$, 
\item[(ii)] \(d(P;\theta)\geq d\).
\end{enumerate}
Then, the family $\mathcal{F}$ is link-constant along some neighbourhood $K \subseteq \R$ of the origin. If the inequalities in (ii) are strict,  then the deformation is link-constant.
 \end{theorem}
 
To improve Theorem~\ref{thm:link-trivialitysemi} for a wider class of mixed polynomials, the authors in \cite{EspinelSanchez2025} studied the link-constancy of  families $\{f+\varepsilon \theta\}$ where $f$ is IND.
\begin{theorem}\label{thm:link-triviality}
Let $\mathcal{F}=\{f+\varepsilon \theta\}_{\varepsilon \in \R}$ be a family of mixed polynomials satisfying: 
\begin{enumerate}
\item[(i)] $f$ is IND, 
\item[(ii)] \(d(P_i;\theta) \geq d(P_i;f) \text{ for any  } P_i \in \mathcal{P}_{\mathrm{inn}}(f).\) 
\end{enumerate} 
Then, the family $\mathcal{F}$ is link-constant along some neighbourhood $K \subseteq \R$ of the origin. If the inequalities in (ii) are strict,  then the deformation is link-constant.
 \end{theorem}
 \begin{obs}\label{rem-topological}
 \begin{enumerate}

 \item[(a)]  Condition (ii) in Theorem~\ref{thm:link-trivialitysemi} and \ref{thm:link-triviality} improve the condition  \(\Gamma_+(\theta) \subseteq \Gamma_+(f),\) which is commonly used by other authors.    

 \item[(b)] From Theorem~\ref{thm:link-triviality}, we conclude that the link $L_{f}$ is isotopic to $L_{f_{\Gamma_{\mathrm{inn}}}}$.
 \end{enumerate}
 \end{obs}
\subsection{Nested links}\label{nested}
In \cite{AraujoBodeSanchez}, Araújo, Bode and Sanchez Quiceno studied links of singularities of IND mixed polynomials, under an extra condition called $\Gamma$-niceness. Following \cite{AraujoBodeSanchez} and the reformulation made in \cite{EspinelSanchez2025}, we describe the characterization of the links of IND mixed polynomials that are $\Gamma_{\mathrm{inn}}$-nice. In this subsection we will review the preliminaries involving the proof of Theorem~\ref{thm:link-triviality} under the assumption of $\Gamma_{\mathrm{inn}}$-niceness, using the set $\mathcal{P}_{\mathrm{inn}}(f)$ that was introduced in \cite{EspinelSanchez2025}, which will be useful in this paper.   

\begin{definition}
We say that a mixed polynomial $f$ is \textbf{$\Gamma_{\mathrm{inn}}$-nice} if, for every inner face $\Delta \in \Gamma_{\mathrm{inn}}(f)$ that is a vertex, we have  $V({f_{\Delta}})\cap (\C^*)^{2}=\emptyset$.
\end{definition}

Let $f$ be an IND mixed polynomial that is $\Gamma_{\mathrm{inn}}$-nice and let $\mathcal{P}_{\mathrm{inn}}(f)=\{P_1,\dots ,P_{N}\}$, where $N\geq 2$. Consider the polar coordinates $u=R \rme^{\rmi \varphi}$ and $v= r \rme^{\rmi t}$. The functions \(f_i:\C \times \R_{\geq 0} \times \mathbb{S}^1 \to \C \) and \( f_{\underline{i}}: \R_{\geq 0} \times \mathbb{S}^1 \times \C \to \C \)  are defined by 

\begin{equation*}\label{rdecomp}
 f_i(u,r,\rme^{\rmi t}):=r^{\frac{-d(P_i;f)}{p_{i,2}}}f(r^{k_i}u,r \rme^{\rmi t})
 \end{equation*}
 and 
 \begin{equation*}\label{Rdecomp}
    f_{\underline{i}}(R, \rme^{\rmi \varphi}, v):=R^{\frac{-d(P_i;f)}{p_{i,1}}}f(R\rme^{\rmi \varphi}, R^{\frac{1}{k_i}}v).
    \end{equation*}
Notice that for the face function \(f_{P_i}\) we have

\begin{equation*}\label{rPidecomp}
 (f_{P_i})_i(u,r,\rme^{\rmi t})=\lim_{r\to 0^+ }f_i(u,r,\rme^{\rmi t}),
\end{equation*}
and 
\begin{equation*}\label{RPidecomp}
(f_{P_i})_{\underline{i}}(R, \rme^{\rmi \varphi}, v)=\lim_{R \to 0^+}f_{\underline{i}}(R, \rme^{\rmi \varphi}, v).
\end{equation*}  
 Thus, the function  \((f_{P_i})_i \) does not depend on \(r\), and \((f_{P_i})_{\underline{i}} \) does not depend on \(R\). 

 By the radial homogeneity of $f_{P_i}$, we have 
 \begin{align*}
  V(f_{P_i})\cap (\C^*)^2=& \{ (r^{k_i}u_*,r \rme^{\rmi t_*}) \mid (u_*,0, \rme^{\rmi t_*}) \in V((f_{P_i})_i) \cap (\C^* \times \{0\} \times \mathbb{S}^1),\ r>0 \} \\
  =&  \{(R \rme^{\rmi \varphi_*},R^{\frac{1}{k_i}}v_*) \mid (0,\rme^{\rmi \varphi_*},v_*) \in V((f_{P_i})_{\underline{i}}) \cap (\{0\}\times \mathbb{S}^1\times \C^*),\ R>0 \} \\
=& \{(\lambda^{p_{i,1}} s_1,  \lambda^{p_{i,2}} s_2) \mid (s_1,s_2)\in V(f_{P_i})\cap (\C^*)^2\cap \mathbb{S}^3,\ \lambda>0 \}.
  \end{align*}
The three characterisations of $V(f_{P_i})\cap (\mathbb{C}^*)^2$ yield three associated sets $L_i$, $L_{\underline{i}}$\footnote{The underlined index notation is based on $f_{\underline{i}}$, which was originally established in \cite{BodeSanchez2023}. Note that in \cite{AraujoBodeSanchez}, $L_{N}$ denotes what we refer to as $L_{\underline{N}}$; our set $L_N$ is not used there.} and $L_{P_i}$.       
If $i\neq 1, N$, the properties of $\Gamma_{\mathrm{inn}}$-niceness and inner non-degeneracy imply that  
\begin{align*}
 &L_i:=V((f_{P_i})_{i}) \cap (\C^* \times \{0\} \times \mathbb{S}^1) \subset \C \times \{0\} \times \mathbb{S}^1 \cong \C \times \mathbb{S}^1,   \\
&L_{\underline{i}}:=V((f_{P_i})_{\underline{i}}) \cap (\{0\} \times \mathbb{S}^1 \times \C^*)  \subset \{0\} \times \mathbb{S}^1 \times \C \cong \mathbb{S}^1 \times \C,     \\
&L_{P_i}:= V_{P_i} \cap \mathbb{S}^3 \subset \mathbb{S}^3, \text{ with } V_{P_i}:=V(f_{P_i}) \cap (\C^*)^2, 
\end{align*}
are smooth links on the respective ambients. Moreover, \(L_i\) 
and  \(L_{\underline{i}}\)  can be embedded on \(\mathbb{S}^3\) such that they  are isotopic to \(L_{P_i}\).  For \(i=1,N\), we cannot always define \(L_i\) and $L_{\underline{i}}$ in the same way as \(L_j\) and $L_{\underline{j}}$, \(j=2,\dots,N-1.\) This is because \(V((f_{P_1})_{1}) \cap (\C^* \times \{0\} \times \mathbb{S}^1)\) and \(V((f_{P_N})_{\underline{N}}) \cap (\{0\} \times \mathbb{S}^1 \times \C^*)\) may not be compact on \(\C^* \times \{0\} \times \mathbb{S}^1\) and  \(\{0\} \times \mathbb{S}^1 \times \C^*\), respectively. To fix that, we define
\begin{align*}
 &L_1:=V((f_{P_1})_{1}) \cap (\C \times \{0\} \times \mathbb{S}^1), \\
 &L_{P_1}:= V_{P_1} \cap \mathbb{S}^3 \subset \mathbb{S}^3, \text{ with } V_{P_1}:=V(f_{P_1}) \cap (\C\times \C^*),  
\end{align*}
and  
\begin{align*}
&L_{\underline{N}}:=V((f_{P_N})_{\underline{N}}) \cap (\{0\} \times \mathbb{S}^1 \times \C),\\
 &L_{P_N}:= V_{P_N} \cap \mathbb{S}^3 \subset \mathbb{S}^3, \text{ with } V_{P_N}:= V(f_{P_N}) \cap (\C^*\times \C).
\end{align*} 
We also have that the links \(L_1\) and \(L_{\underline{N}}\)  can be embedded on \(\mathbb{S}^3\) such that they are isotopic to the links \(L_{P_1}\) and \(L_{P_N}\), respectively.  
\vspace{0.2cm}

Up to a homeomorphism, we consider $L_i$ and $L_{\underline{i}}$ to be subsets of $\mathbb{C} \times \mathbb{S}^1$ and $\mathbb{S}^1 \times \mathbb{C}$, respectively. Since that $f$ is IND and $\Gamma_{\mathrm{inn}}$-nice \cite{AraujoBodeSanchez}: the link $L_i,\  i=1,\dots, N-1$ is empty or parametrized by
\begin{equation*}\label{uparametrization}
\bigcup_{j=1}^{M_i}\{(u_{i,j}(\tau),\rme^{\rmi t_{i,j}(\tau)}) \mid \,\tau\in[0,2\pi]\}\subset\mathbb{C}\times \mathbb{S}^1,
\end{equation*}
where $j=1,\ldots,M_i$ is indexing the $M_i$ components of $L_i$, $u_{i,j}:[0,2\pi]\to\mathbb{C}$ and $t_{i,j}:[0,2\pi]\to\mathbb{R}/2\pi$ are appropriate functions with $u_{i,j}(\tau)\neq 0$ for all $i\neq 1$, $j$ and $\tau$.  The link $L_{\underline{i}},\  i=2,\dots,N$ is empty or parametrized by
\begin{equation*}\label{vparametrization}
\bigcup_{j=1}^{M_i}\{(\rme^{\rmi \varphi_{i,j}(\tau)},v_{i,j}(\tau))\mid \,\tau\in[0,2\pi]\}\subset \mathbb{S}^1\times \mathbb{C},
\end{equation*}
where $j=1,\ldots,M_i$ is indexing the $M_i$ components of $L_i$, $v_{i,j}:[0,2\pi]\to\mathbb{C}$ and $\varphi_{i,j}:[0,2\pi]\to\mathbb{R}/2\pi$ are appropriate functions with $v_{i,j}(\tau)\neq 0$ for all $i\neq N$, $j$ and $\tau$. 
 
\begin{definition}\label{DEF: decompostion links}
Let \((L_1,L_2, \dots, L_{\ell}, \dots, L_N)\) be a sequence of possible empty links associated with \(f\), then we define $\ell$ as the minimal $i$ such that $k_i<1$ and the links $[L_1,L_2,\ldots,L_{\ell-1}]$ on $\C \times \mathbb{S}^1$ and $[L_N,L_{N-1},\ldots,L_{\ell}]'$ on $\mathbb{S}^1 \times \C$ are as follows:
\begin{enumerate}
\item[(i)]  the link $[L_1,L_2,\ldots,L_{\ell-1}]$ is defined by
\begin{equation*}
\bigcup_{i=1}^{N-1}\bigcup_{j=1}^{M_i}\{(\varepsilon_u^{k_i} u_{i,j}(\tau),\rme^{\rmi t_{i,j}(\tau)})\mid  \,\tau\in[0,2\pi]\}\subset\mathbb{C}\times \mathbb{S}^1,
\end{equation*}
for some sufficiently small $\varepsilon_u>0$, 
\item[(ii)] the link $[L_{\underline{N}},L_{\underline{N-1}},\ldots,L_{\underline{\ell}}]'$ is defined by
\begin{equation*}
\bigcup_{i=1}^{N-1}\bigcup_{j=1}^{M_i}\{(\rme^{\rmi \varphi_{i,j}(\tau)},\varepsilon_v^{\frac{1}{k_i}} v_{i,j}(\tau)) \mid \tau\in[0,2\pi]\}\subset\mathbb{S}^1\times \mathbb{C},
\end{equation*}
for some sufficiently small $\varepsilon_v>0$.
\end{enumerate}
\end{definition}  
We introduce the notation $\text{[*]}'$, which is necessary for the objectives of the present work. We also use the notation $\text{[*]}$ from \cite{AraujoBodeSanchez}. However, $\text{[*]}'$ creates the additional possibilities that our study demands, which is unnecessary within the scope of \cite{AraujoBodeSanchez}.
\begin{figure}[h]
         \centering
\includegraphics[width=0.7\linewidth]{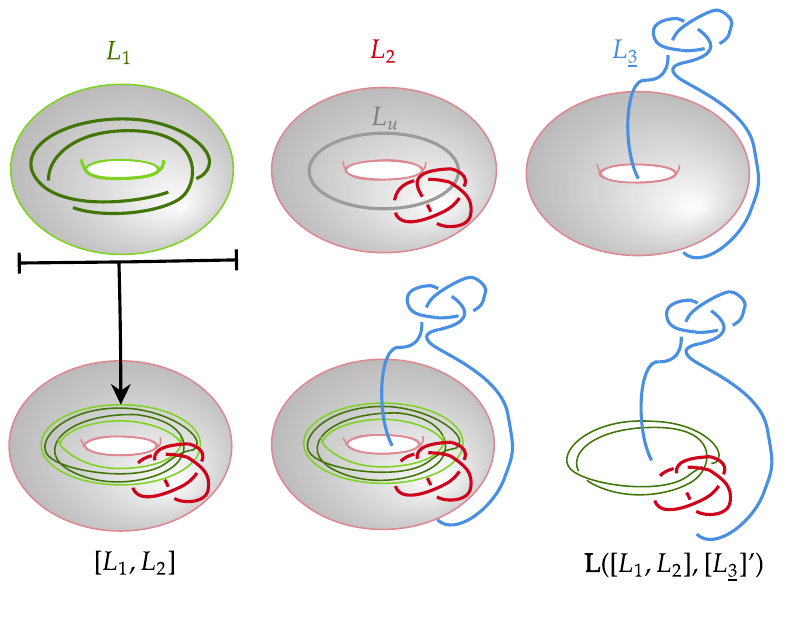}
         \caption{Nested link. Notice that $[*]$ represents the components in $\C\times \mathbb{S}^1$ and $[*]'$ represents the components in $ \mathbb{S}^1\times \C$}
         \label{fig:Nestedlink}
     \end{figure}
From links $K_1 \subset \C \times \mathbb{S}^1$ and $K_2 \subset \mathbb{S}^1 \times \C$, in \cite{AraujoBodeSanchez} is defined a union that yields a link $\mathbf{L}(L_1,L_2)$ on the 3-sphere, which is called a  \textbf{nested link} (For example, Figure~\ref{fig:Nestedlink}). 
\begin{definition}[\cite{AraujoBodeSanchez}]\label{nestedlink}
Let $K_1$ be a link in $\C\times \mathbb{S}^1$ that is either empty or parametrized by
\begin{equation}\label{linkparametrization}
\bigcup_{j=1}^M\{(u_{j}(\tau),\rme^{\rmi t_{j}(\tau)})\mid\tau\in[0,2\pi]\}
\end{equation}
for appropriate functions $u_j$, $t_j$,
and $K_2$ be a link in $\mathbb{S}^1\times\C$ that is either empty or parametrized by
\begin{equation*}
\bigcup_{j=1}^{M'}\{(\rme^{\rmi \varphi_{j}(\tau)},v_{j}(\tau))\mid \tau\in[0,2\pi]\}
\end{equation*}
for appropriate functions $v_j$, $\varphi_j$.
We define the link $\widetilde{K_i}$, $i=1,2$, in $\mathbb{S}^3$, which is empty if $K_i$ is empty and otherwise parametrized by 
\begin{equation*}
\bigcup_{j=1}^M\left\{\left(\varepsilon u_{j}(\tau),\sqrt{1-\varepsilon^2|u_j(\tau)|^2}\rme^{\rmi t_{j}(\tau)}\right)\mid \tau\in[0,2\pi]\right\}
\end{equation*}
and
\begin{equation*}
\bigcup_{j=1}^{M'}\left\{\left(\sqrt{1-\varepsilon^2|v_j(\tau)|^2}\rme^{\rmi \varphi_{j}(\tau)},\varepsilon v_{j}(\tau)\right)\mid \tau\in[0,2\pi]\right\}
\end{equation*}
for some small $\varepsilon>0$, respectively.
Then we write $\mathbf{L}(K_1,K_2)$ for the link in $\mathbb{S}^3$ given by $\widetilde{K_1}\cup\widetilde{K_2}$. If $K_1$ (resp. $K_2$) is empty, we denote  $\mathbf{L}(K_1,K_2)$ by $\mathbf{L}(K_2)$ (resp.  $\mathbf{L}(K_1)$).
\end{definition}

From \cite{AraujoBodeSanchez}, we deduce the characterisation of IND mixed polynomials that are $\Gamma_{\mathrm{inn}}$-nice (see also \cite{EspinelSanchez2025}). 
\begin{theorem}\label{thm:link-triviality1}
Let $f$ be an IND mixed polynomial that is $\Gamma_{\mathrm{inn}}$-nice, and take $L_i  ,\ i=1,\dots, N-1,$ and $L_{\underline{N}}$ as the links associated with the weight vectors of $\mathcal{P}_{\mathrm{inn}}(f)=\{P_1,\dots, P_N\}$. Then, the link of the singularity of $f$ is ambient isotopic to:
\begin{itemize}
\item[(i)] the link $L_{f_{P_1}}$, if $N=1$,
\item[(ii)] the nested link $\mathbf{L}([L_1,\ldots,L_{N-1}],[L_{\underline{N}}]')$, if $N\geq 2$.
\end{itemize}
\end{theorem}
\begin{remark}\label{geralzao}
Consider a mixed polynomial $f$ that is IND and $\Gamma_{\mathrm{inn}}$-nice and let $\mathcal{P}_{\mathrm{inn}}(f)=\{P_1,\dots,P
	_N\}$ with $N\geq 2$. By a modification of the nested representation on the link in  Theorem~\ref{thm:link-triviality1}, we obtain that $L_f$ is isotopic to the link $\mathbf{L}([L_1,L_2,\dots,L_{\ell-1}],[L_{\underline{N}},L_{\underline{N-1}}, \dots, L_{\underline{\ell}}]'])$, where $\ell$ is the minimal index in $\{1,\dots,N\}$ such that $k_{\ell}<1$. By construction, any sublink $J$ of $L_i$, $L_{\underline{i}}$, or $L_{P_i}$ lies in a component of $V(f_{P_i})$, denoted by $V(J;f_{P_i})$, and it is isotopic (provided by Theorem~\ref{thm:link-triviality1}) to a link lying in a component of $V(f)$, denoted by $V(J;f)$.

If $i=1,\dots ,\ell-1$, and $J$ is a sublink of $L_i$, the zeros  $V(J;f)$ and $V(J;f_{P_i})$ can be parametrized as $$(r^{k_i}u_{i,J}(r,\tau),r\rme^{\rmi t_{i,J}(r,\tau)}) \text{ and } (r^{k_i}u_{i,J}(0,\tau),r\rme^{\rmi t_{i,J}(0,\tau)}),$$
respectively, where $u_{i,J}$ and $t_{i,J}$ are subanalytic functions such that (see \cite[Equation~(59)]{AraujoBodeSanchez}) $$(u_{i,J}(r,\tau),\rme^{\rmi t_{i,J}(r,\tau)}) \to (u_{i,J}(0,\tau),\rme^{\rmi t_{i,J}(0,\tau)}) \in J \subseteq L_i \text{ when } r\to 0^+.$$ 

If $i=\ell, \dots, N$, and $J$ is a sublink of $L_{\underline{i}}$, the zeros $V(J;f)$ and $V(J;f_{P_i})$ can be parametrized as $$(R \rme^{\rmi \varphi_{i,J}(R,\tau)},R^{1/k_i}v_{i,J}(R,\tau)) \text{ and } (R \rme^{\rmi \varphi_{i,J}(0,\tau)},R^{1/k_i}v_{i,J}(0,\tau)),$$ respectively, where $v_{i,J}$ and $\varphi_{i,J}$ are subanalytic functions such that (see \cite[Equation~(60)]{AraujoBodeSanchez}) 

$$(\rme^{\rmi \varphi_{i,J}(R,\tau)},v_{i,J}(R,\tau)) \to (\rme^{\rmi \varphi_{i,J}(0,\tau)},v_{i,J}(0,\tau)) \in J \subseteq  L_{\underline{i}} \text{ when } R\to 0^+.$$

Note that if $L_i$, $L_{\underline{i}}$, and $L_{P_i}$ are compact in their respective spaces, then $V(L_i;f_{P_i})=V(L_{\underline{i}};f_{P_i})=V(L_{P_i};f_{P_i})$ and $V(L_i;f)=V(L_{\underline{i}};f)=V(L_{P_i};f)$.
\end{remark}

\begin{obs} \label{equivalenceoflinks}
	The representation of the link $L_f$ as a nested link can be changed,  depending on data from $\Gamma_{\mathrm{inn}}(f)$. For instance:
	\begin{enumerate} 
	\item[(i)]  If $k_i>1$ for all $P_i \in \mathcal{P}_{\mathrm{inn}}(f)$, then \(f_{\Gamma_{\mathrm{inn}}}\) is \(u\)-convenient,  $L_N:=V((f_{P_N})_{N}) \cap (\C^* \times\{0\}\times  \mathbb{S}^1) \subset \C \times \{0\} \times \mathbb{S}^1$ is compact and 
\[L_f \simeq \mathbf{L}([L_1,\dots,L_N]).\]
\item[(ii)]  If $k_i<1$ for all $P_i \in \mathcal{P}_{\mathrm{inn}}(f)$, then \(f_{\Gamma_{\mathrm{inn}}}\) is $v$-convenient,  $L_{\underline{1}}:=V((f_{P_1})_{\underline{1}}) \cap (\{0\} \times \mathbb{S}^1 \times \C^*) \subset \{0\} \times \mathbb{S}^1 \times \C$ is compact and 
\[L_f \simeq 	\mathbf{L}([L_{\underline{N}},\dots, L_{\underline{1}}]').\]
\item[(iii)] If $k_i\leq 1$ for all $P_i \in \mathcal{P}_{\mathrm{inn}}(f)$ and  \(f_{\Gamma_{\mathrm{inn}}}\) is not \(u\)-convenient (respectively, $k_i\geq 1$ for all $P_i \in \mathcal{P}_{\mathrm{inn}}(f)$ and \(f_{\Gamma_{\mathrm{inn}}}\) is not \(v\)-convenient), then \(L_{P_N} \simeq \mathbf{L}([L_N]')\) is isotopic to the trivial knot 
$L_{v}:=\mathbb{S}^3 \cap \{v=0\}$ (respectively, \(L_{P_1}\simeq \mathbf{L}([L_1])\) is isotopic to  $L_{u}:=\mathbb{S}^3 \cap \{u=0\}$). 
\end{enumerate}	
     \end{obs}
     
     \begin{obs}
        Sometimes it is useful to distinguish what function $f$ the links $L_i$, $L_{\underline{i}}$ and $L_{P_i}$ are associated to. In these cases, we write $L_{i,f}:=L_i$, $L_{\underline{i},f}:=L_{\underline{i}}$ and $L_{P_i,f}:=L_{P_i}$. For example, by Theorem~\ref{thm:link-triviality}, any IND mixed polynomial $f$ that is $\Gamma_{\rm{inn}}$-nice can be deformed to an IND mixed polynomial $\hat{f}$ such that $L_{f}$ is isotopic to $\mathbf{L}([L_{1,\hat{f}},\dots,L_{N,\hat{f}}])$ and $\mathbf{L}([L_{\underline{N},\hat{f}},\dots,L_{\underline{1},\hat{f}}]')$. As we will see through the paper, this does not necessarily imply that $f$ and $\hat{f}$ are bi-Lipschitz $V$-equivalent (see Definition \ref{V}), which shows the importance of such a distinction in some cases.
     \end{obs}
     \begin{example}
      Consider the polynomial  $f(u,v) = u^8 + v^3u^2 + \bar{v}^6u^2 + \bar{v}^5u + v^4\bar{v}^4$ from Example~\ref{exemplo-poligono-newton}. 
We observe that $\mathcal{P}_{\mathrm{inn}}(f)=\{Q_1=(2,1),Q_2=(1,2)\}$ (see Figure~\ref{diagramf}). By Theorem~\ref{thm:link-triviality1}, the link of the singularity of $f$ is isotopic to the nested link $\mathbf{L}([L_1],[L_{\underline{2}}]')$, where $L_1$ and $L_{\underline{2}}$ are the links associated with the weight vectors $Q_1$ and $Q_2$, respectively. The functions that define explicitly the links are:
\begin{equation*}
 (f_{Q_1})_1(u,r,\rme^{\rmi t}):=\lim_{r \to 0^+}r^{-7}f_{Q_1}(r^{2}u,r \rme^{\rmi t})=\rme^{3\rmi t}u^2+\rme^{-5\rmi t}u
 \end{equation*}
 and 
 \begin{equation*}
    (f_{Q_2})_{\underline{2}}(R, \rme^{\rmi \varphi}, v):\lim_{R\to 0^+}R^{-8}f(R\rme^{\rmi \varphi}, R^{2}v)=\rme^{8 \rmi \varphi}+v^3\rme^{2\rmi \varphi}.
    \end{equation*}
The link $L_1$ is formed (up to a homomorphism) by the union of two knots: $\{(0,\rme^{\rmi \tau})\mid \tau \in [0,2\pi]\}$ and $\{(\rme^{-8\rmi \tau},\rme^{\rmi \tau})\mid \tau \in [0,2\pi]\}$. The link $L_{\underline{2}}$ is formed (up to a homomorphism) by the union of the three knots: $\{(\rme^{\rmi \tau},\mathrm{e}^{\mathrm{i} \left( \frac{\pi}{3} + 2 \tau \right)})\mid \tau \in [0,2\pi]\}$,  $\{(\rme^{\rmi \tau},\mathrm{e}^{\mathrm{i} \left( \pi + 2 \tau \right)})\mid \tau \in [0,2\pi]\}$ and  $\{(\rme^{\rmi \tau},\mathrm{e}^{\mathrm{i} \left( \frac{5\pi}{3} + 2 \tau \right)})\mid \tau \in [0,2\pi]\}$.

Since $f$ is $u$-convenient and $u$-semiholomorphic (meaning $f_{\bar{u}}\equiv 0$), we can represent the link $L_f$ as the nested link $\mathbf{L}([L_1,L_2])$. The links $L_1$, $L_2$ and $[L_1,L_2]$ are illustrated in Figure~\ref{fig1:main}. We calculate the component $L_2$ from the associated function: 
\begin{equation*}
    (f_{Q_2})_{2}(u,r, \rme^{\rmi t}):=r^{-4}f(r^{\frac{1}{2}}u,r\rme^{\rmi t})=u^8+\rme^{3 \rmi t}u^2.
    \end{equation*}
    \begin{figure}[h]
    \centering
    \begin{subfigure}[b]{0.34\textwidth}
        \centering
        \includegraphics[width=\textwidth]{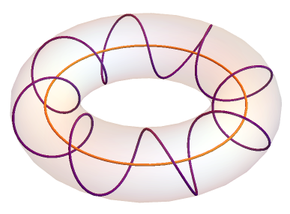}
        \caption{$L_1 \subset \C \times \mathbb{S}^1$.}
        \label{fig1:sub-a}
    \end{subfigure}
    \hfill 
    \begin{subfigure}[b]{0.34\textwidth}
        \centering
        \includegraphics[width=\textwidth]{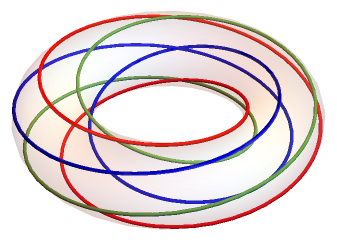}
        \caption{$L_{2} \subset  \C  \times \mathbb{S}^1$.}
        \label{fig1:sub-b}
    \end{subfigure}
     \begin{subfigure}[c]{0.34\textwidth}
        \centering
        \includegraphics[width=\textwidth]{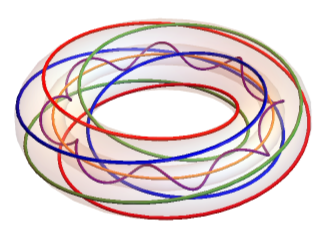}
        \caption{ $[L_1,L_{{2}}] \subset  \C \times \mathbb{S}^1$.}
        \label{fig1:sub-c}
    \end{subfigure}
    \caption{}
    \label{fig1:main}
\end{figure}
    \begin{obs}\label{braidclosure}
In \cite{AraujoBodeSanchez}, Araújo, Bode and Sanchez Quiceno also investigated the link of IND semiholomophic polynomials. In this context, these polynomials are automatically $\Gamma_{\mathrm{inn}}$-nice, and their links can be represented as the closure of a braid. Specifically, a braid $B$ in $\C \times [0,2\pi]$ is carried to a link in $\C \times \mathbb{S}^1$ by identifying $\C \times \{0\}$ and $\C \times \{2\pi\}$; this process yields the \textbf{closure of $B$ in $\C \times \mathbb{S}^1$}. The sphere $\mathbb{S}^3$ can be decomposed as the union $\mathbb{T}_1=\{(u,v) \in \mathbb{S}^3 \mid |v|\geq |u|\}$ and $\mathbb{T}_2=\{(u,v) \in \mathbb{S}^3 \mid |u|\geq |v|\}$, and $\C \times \mathbb{S}^1$ and $\mathbb{S}^1 \times \C$ are homeomorphic (by unknotted embeddings) to $\operatorname{Int}(\mathbb{T}_1)$ and $\operatorname{Int}(\mathbb{T}_2)$  respectively. 
The \textbf{closure of $B$ in $\mathbb{S}^3$} is the image of the closure of $B$ in $\C \times \mathbb{S}^1$ by the unknotted embedding in $\operatorname{Int}(\mathbb{T}_1)$. This closure of $B$ has $L_v$ as braid axis. 
Analogously, we construct for a braid $B$ in $[0,2\pi] \times \C$ the \textbf{closure of $B$ in $\mathbb{S}^1 \times \C$} and the  \textbf{closure of $B$ in $\mathbb{S}^3$}.  This closure of $B$ lies in $\operatorname{Int}(\mathbb{T}_2)$ and has $L_u$ as braid axis. 

For instance, the braid in Figure~\ref{fig2:sub-a} closes to the link $L_1$ in Figure~\ref{fig1:sub-a}, the braid in Figure~\ref{fig2:sub-b} closes to the link $L_2$ in Figure~\ref{fig1:sub-b}, and the braid in Figure~\ref{fig2:sub-c} closes to the link $[L_1,L_2]$ in Figure~\ref{fig1:sub-c}. The braid in Figure~\ref{fig2:sub-c} is referred to as a \textbf{nested braid} and is denoted by $\mathbf{B}(B_1,B_2)$, where $B_1$ and $B_2$ are the braids  shown in Figures \ref{fig2:sub-a} and \ref{fig2:sub-b}, respectively.
 \end{obs}
\begin{figure}[h]
    \centering
    \begin{subfigure}[b]{0.35\textwidth}
        \centering
        \includegraphics[width=\textwidth]{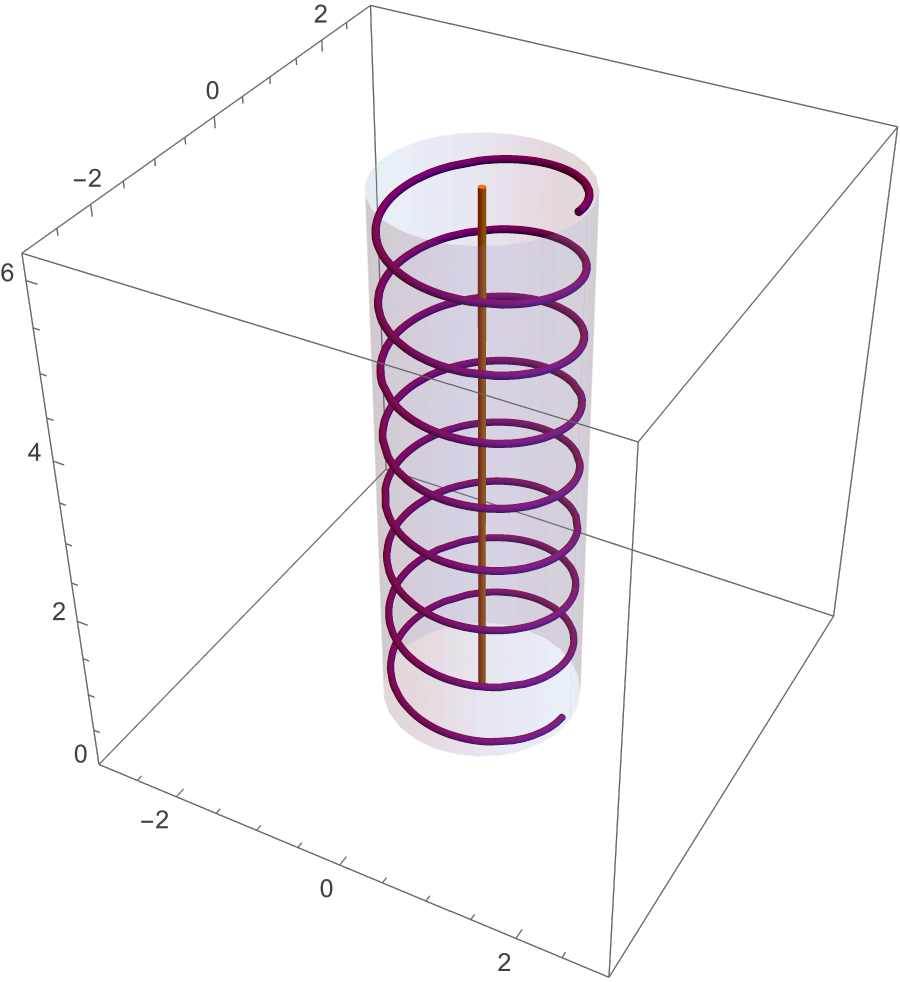}
        \caption{Braid $B_1$}
        \label{fig2:sub-a}
    \end{subfigure}
    \hfill 
    \begin{subfigure}[b]{0.35\textwidth}
        \centering
        \includegraphics[width=\textwidth]{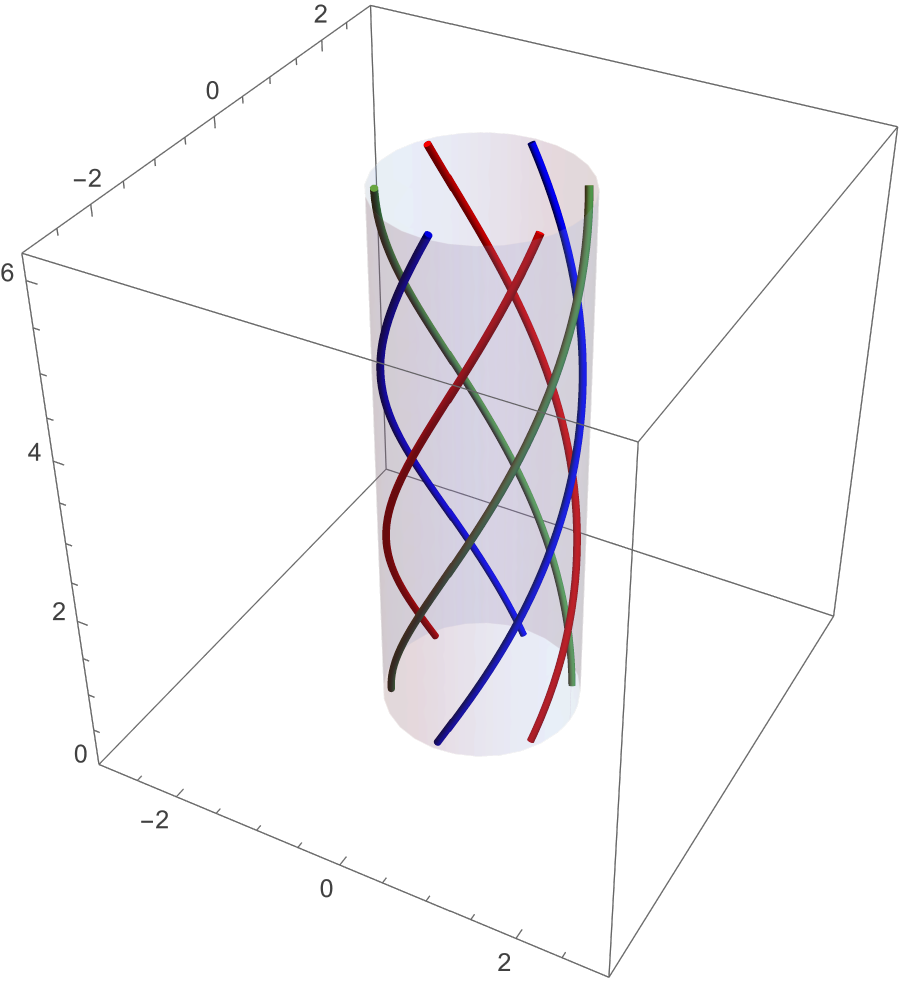}
        \caption{Braid $B_2$.}
        \label{fig2:sub-b}
    \end{subfigure}
     \begin{subfigure}[c]{0.35\textwidth}
        \centering
        \includegraphics[width=\textwidth]{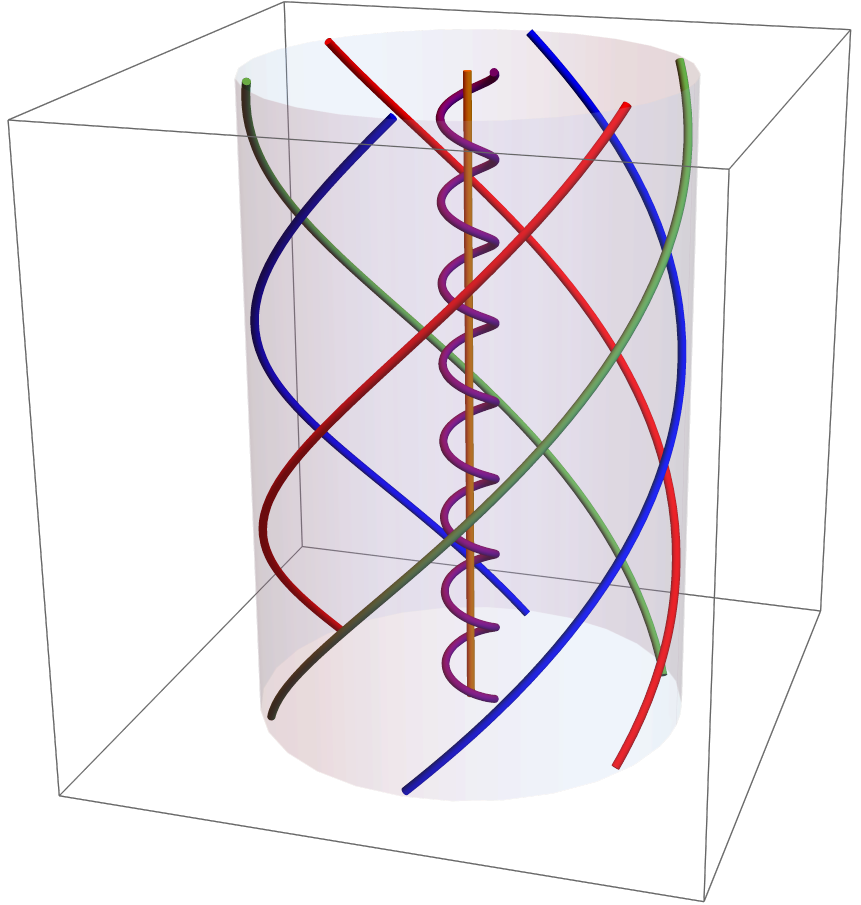}
        \caption{Braid $\mathbf{B}(B_1,B_2)$.}
        \label{fig2:sub-c}
    \end{subfigure}
    \caption{}
    \label{fig2:main}
\end{figure} 
     \end{example}

To finish this subsection, we define the following class of mixed polynomials, which is inspired by the link associations presented in \cite{Oka2010}, and which will be useful to improve necessary conditions for bi-Lipschitz equivalence in some cases.

\begin{definition}\label{Def:truepolynomials}
	Let \( f \) be an IND mixed polynomial that is $\Gamma_{\mathrm{inn}}$-nice, and let \( P_i \in \mathcal{P}_{\mathrm{inn}}(f) \). We say that \( f \) is \textbf{\(\Gamma_{\mathrm{inn}}\)-true} for \( P_i \) if:
	\begin{itemize}
		\item[(i)] \(i=1\), the link \( L_{P_1;f} \) is non-empty when \( N_f > 1 \), or the link \( L_{f_{P_1}} \) is non-empty when \( N_f = 1 \).
		\item[(ii)] \( i \neq 1 \), the link \( L_{P_i;f} \) is non-empty.
	\end{itemize}
	We say that \( f \) is \textbf{\(\Gamma_{\mathrm{inn}}\)-true} if \( f \) is \textbf{\(\Gamma_{\mathrm{inn}}\)-true} for all \( P_i \in \mathcal{P}_{\mathrm{inn}}(f) \).
\end{definition}
Let $f$ be an IND mixed polynomial that is $\Gamma_{\mathrm{inn}}$-nice. Let $N_f$ be the number of compact 1-faces of $\Gamma_{\mathrm{inn}}(f)$ and let $k_{i,f}=\frac{p_{i,1}}{p_{i,2}}$ be the slopes of these compact 1-faces associated with $P_i=(p_{i,1},p_{i,2})\in \mathcal{P}_{\mathrm{inn}}(f)$.  
Define 
\begin{equation}\label{Def:If}
	I_f:=\{i \in \{1,\dots ,N_f \} \mid f \text{ is } \Gamma_{\mathrm{inn}}\text{-true for } P_i\}.
\end{equation}

In general, $I_f \subsetneq \{1,\dots, N_f\}$. But, if $f$ is $\Gamma_{\mathrm{inn}}$-true (for instance, semiholomorphic), then we have $I_f=\{1,\dots, N_f\}$. We also define 
\begin{equation}\label{Def:Ifk}
		I_{f}^{(k>1)}:=\{i \in I_f \mid  k_i>1\} \text{  and  }  I_{f}^{(k\geq 1)}:=\{i \in I_f \mid   k_i\geq 1\}.
\end{equation}
Analogously, we define $ I_{f}^{(k< 1)}$,  $I_{f}^{(k\leq 1)}$, $I_{f}^{(k=1)}$, and $I_{f}^{(k\neq 1)}$. 

\subsection{Excerpts of Lipschitz geometry}\label{Subsec:Lipschitz-geometry}
Most of the preliminaries in Lipschitz geometry used in this paper are included with more detail in the survey paper \cite{livroLip}. We included the necessary ones to make the paper self-contained.

\begin{definition}\label{Lip}
	Given two metric spaces $(X_1, d_1)$ and $(X_2, d_2)$, we say that a homeomorphism $\phi: X_1 \to X_2$ is \textbf{bi-Lipschitz for the metrics $d_1$ and $d_2$} (or bi-Lipschitz, for short) if there is $K \ge 1$ such that
	\begin{equation*}
		\frac{1}{K} \cdot d_1(p, q) \leq d_2(\phi(p), \phi(q)) \leq K \cdot d_1(p, q), \quad \forall p, q \in X_1.
	\end{equation*}
	For each $K \ge 1$ that satisfies such a condition, we say that $\phi$ is $K$-bi-Lipschitz. Furthermore, we say that the metric spaces $(X_1, d_1)$ and $(X_2, d_2)$ are \textbf{bi-Lipschitz equivalent} if there is a bi-Lipschitz map $\phi: X_1 \to X_2$.
\end{definition}

\begin{definition}\label{Outer-metric}
	Given $X \subset \mathbb{R}^{n}$, we define the following two metrics in $X$:
	\begin{itemize}
		\item \textbf{outer metric of $X$}: we define the outer metric of $X$ as the distance $d: X \times X \to \mathbb{R}_{+}$, given by $d(x, y) = \|x - y\|$, for all $x, y \in X$;
		\item \textbf{inner metric of $X$}: if $X$ is path-connected, we define the inner metric of $X$ as the distance $d_{X}: X \times X \to \mathbb{R}_{+}$, given by $d_{\rm inn}(x, y) = \inf{\{ l(\alpha) \}}$, for all $x, y \in X$. The infimum is taken over all rectifiable paths $\alpha \subset X$ from $x$ to $y$, and $l(\alpha)$ is the length of $\alpha$ (if such a rectifiable path $\alpha$ does not exist, we define $d_{\rm inn}(x, y) = \infty$).
	\end{itemize}
\end{definition}

\begin{definition}\label{Inner-metric}
	Given two sets $X_1 \subset \mathbb{R}^{m}$ and $X_2 \subset \mathbb{R}^{n}$, we define the following bi-Lipschitz maps from $X_1$ to $X_2$:
	\begin{itemize}
		\item \textbf{outer bi-Lipschitz map}: an outer bi-Lipschitz map $\phi: X_1 \to X_2$ is a map that is bi-Lipschitz for the outer metrics of $X_1$ and $X_2$;
		\item \textbf{inner bi-Lipschitz map}: an inner bi-Lipschitz map $\phi: X_1 \to X_2$ (assuming $X_1$ and $X_2$ are path-connected) is a map that is bi-Lipschitz for the inner metrics of $X_1$ and $X_2$.
		\item \textbf{ambient bi-Lipschitz map:} an ambient bi-Lipschitz map between $X_1$ and $X_2$ is a bi-Lipschitz map $\phi : \mathbb{R}^{m} \to \mathbb{R}^{n}$ with respect to the outer metric, such that $\phi (X_1) = X_2$. Notice that if such an ambient bi-Lipschitz map exists, then $m=n$.
	\end{itemize}
\end{definition}

\begin{definition}\label{outer-lip}
	Given two sets $X_1 \subset \mathbb{R}^m$, $X_2 \subset \mathbb{R}^n$, we say that $X_1$ and $X_2$ are \textbf{outer bi-Lipschitz equivalent} (resp. inner, ambient) if there is an outer bi-Lipschitz (resp. inner, ambient) map between $X_1$ and $X_2$. Given $p\in X_1$, $q\in X_2$, we say that two germs $(X_1,p)$, $(X_2,q)$ are \textbf{outer bi-Lipschitz equivalent} (resp. inner, ambient) if there are neighbourhoods $U$ of $p$ and $V$ of $q$ and an outer bi-Lipschitz (resp. inner, ambient) map between $X_1\cap U$ and $X_2 \cap V$.
\end{definition}

\begin{remark}\label{amb-out-inn}
	If $X_1$ and $X_2$ are ambient bi-Lipschitz equivalent, then $X_1$ and $X_2$ are outer bi-Lipschitz equivalent, and if $X_1$ and $X_2$ are outer bi-Lipschitz equivalent, then $X_1$ and $X_2$ are inner bi-Lipschitz equivalent. However, the converses are generally false (see \cite{BBG} for counterexamples).
\end{remark}

\begin{remark}
    Throughout this paper, when we refer to a bi-Lipschitz map $\phi: X \to Y$ (or a bi-Lipschitz map $\phi: (X,0)  \to (Y,0)$), we mean that $\phi$ is an outer bi-Lipschitz map. Moreover, when we say that $X$ and $Y$ (or $(X,0)$ and $(Y,0)$) are bi-Lipschitz equivalent, this means that they are outer bi-Lipschitz equivalent.
\end{remark}

\begin{definition}\label{LNE}
	We say that a path-connected set $X \subset \mathbb{R}^{n}$ is \textbf{Lipschitz normally embedded} (or LNE, for short) if the outer metric and inner metric are bi-Lipschitz equivalent, i.e., there is a constant $K \ge 1$ such that
	\begin{equation*}
		d_{\rm inn}(x,y) \leq K \cdot d(x,y), \quad \forall \; x,y \in X.
	\end{equation*}
	In this case, we say that $X$ is $K$-LNE. Given $p\in X$, we say that $X$ is \textbf{Lipschitz normally embedded at $p$} (or LNE at $p$) if there exists a neighbourhood $U$ of $p$ such that $X \cap U$ is LNE, or equivalently, the germ $(X,p)$ is LNE. If $X\cap U$ is $K$-LNE for some $K\ge 1$, we say that $(X,p)$ is $K$-LNE.
\end{definition}

The following result, proved by Mendes and Sampaio in \cite{LLNE-LNE}, is a very useful criterion for detecting LNE set germs in terms of their links.

\begin{theorem}\label{Edson-Rodrigo}
Let $X \subset \mathbb{R}^{n}$ be a closed subanalytic set, such that $0 \in X$ and the link of $X$ is connected. Then, $X$ is LNE at $0$ if and only if there is a constant $K\ge 1$ such that $X_t$ is $K$-LNE for all $t>0$ small enough.
\end{theorem}

\begin{definition}\label{Def: arc}
		An  \textbf{arc} in $\mathbb{R}^{n}$ (or \textbf{real half-branch}) is a germ at the origin of a map $\gamma \colon [0,\epsilon) \longrightarrow \mathbb{R}^{n}$ such that $\gamma(0) = 0$. Unless otherwise specified, we suppose that arcs are parametrized by the distance to the origin, i.e., $||\gamma(t)||=t$. For a germ (at 0) of a set $X$, we define the set of all arcs $\gamma \subset X$ as the  \textbf{Valette link} of $X$ (see \cite{Valette}).
\end{definition}

\begin{definition}\label{DEF: order of tangency}
	 The \textbf{tangency order of two arcs} $\gamma_{1}$ and $\gamma_{2}$ in the Valette link of $X$ is the exponent $q$ where $||\gamma_{1}(t) - \gamma_{2}(t)|| = ct^{q} + o(t^{q})$ with $c> 0$. We denote the tangency order as $\tord(\gamma_{1},\gamma_{2})$. By convention, for every arc $\gamma$, we set $\tord(\gamma,\gamma)=\infty$).
     
     For an arc $\gamma$ and a set of arcs $Z$ in the Valette link of $X$, \textbf{the tangency order} of $\gamma$ and $Z$ is the supremum of $\tord(\gamma, \lambda)$ over all arcs $\lambda$ in the Valette link of $Z$. We denote this tangency order by either $\tord(\gamma, Z)$ or $\tord(Z,\gamma)$. 
     
     Finally, the \textbf{contact order of two set germs} $Z$ and $Z'$ at $0$ is the supremum of $\tord(\gamma, Z')$ over all arcs $\gamma$ in the Valette link of $Z$. We denote the contact order as either $\tord(Z,Z')$ or $\operatorname{Cont}(Z,Z')$.
\end{definition}

\begin{remark}
    It was shown in \cite{B.F.METRIC-THEORY} that, for any outer bi-Lipschitz map $\phi:(X,0)\to (Y,0)$ and for any arcs $\gamma$ in the Valette link of $X$, $\lambda$ in the Valette link of $Y$, we have ${\rm tord}(\gamma,\lambda)={\rm tord}(\phi(\gamma),\phi(\lambda))$. As a consequence, notice also that if $\phi: (X \cup Y,0) \to (X'\cup Y',0)$ is an outer bi-Lipschitz map such that $\phi(X)=X'$ and $\phi(Y)=Y'$, then $\operatorname{Cont}(X,Y)=\operatorname{Cont}(X',Y')$. In particular, if $\operatorname{Cont}(X,Y)=1$, then $\tord(\gamma,\lambda)=1$, for all $\gamma$ in the Valette link of $X$ and $\lambda$ in the Valette link of $Y$.
\end{remark}

\begin{remark}\label{Rem: non-archimedean property}
	 An interesting fact about the tangency order of arcs in $\R^n$ is the so called ``non-archimedean property'' (it first appeared in \cite{B.F.METRIC-THEORY} as ``isosceles property''): given arcs $\gamma_1,\gamma_2,\gamma_3$ in $\R^n$, we have $\tord(\gamma_2,\gamma_3) \ge \min(\tord(\gamma_1,\gamma_2),\tord(\gamma_1,\gamma_3))$. Moreover, if $\tord(\gamma_1,\gamma_2)\ne \tord(\gamma_1,\gamma_3)$ then $$\tord(\gamma_2,\gamma_3) = \min(\tord(\gamma_1,\gamma_2),\tord(\gamma_1,\gamma_3)).$$
\end{remark}

\begin{definition}\label{DEF: standard Holder triangle}
	  Let $\beta \in \mathbb{Q}_{\ge 1}$. We define the \textbf{standard $\beta$-H\"older triangle} as the germ at 0 of 
      \begin{equation*}
		T_\beta = \{(x,y)\in \R^2 \mid x\ge 0, \; 0\le y \le x^\beta\}.
	\end{equation*}
    The germs at 0 of $\gamma_0=\{(x,0) \in \R^2 \mid x\ge0\}$ and $\gamma_1=\{(x,x^{\beta}) \in \R^2 \mid x\ge0\}$ are defined as the \textbf{boundary arcs of $T_\beta$}. A germ of a set $T$ that is inner bi-Lipschitz equivalent to the standard $\beta$-H\"older triangle is called a \textbf{$\beta$-H\"older triangle}. If $\phi: T_\beta \to T$ is an inner bi-Lipschitz map, then $(\phi(\gamma_0),0)$ and $(\phi(\gamma_1),0)$ are defined as the \textbf{boundary arcs of $T$}. The \textbf{exponent of $T$} is $\beta$.
    \end{definition}

    \begin{definition}
      The \textbf{standard $\beta$-horn} is defined as the germ at 0 of the set \begin{equation*}\label{Formula:Standard Holder triangle}
		H_\beta = \{(x,y,z)\in \R^3 \mid z\ge 0, \; x^2+y^2=z^{2\beta}\}.
	\end{equation*}
    A germ at 0 of a set $H \subset \mathbb{R}^{n}$ that is inner bi-Lipschitz equivalent to the standard $\beta$-horn $H_\beta$ is called a \textbf{$\beta$-horn}. The number $\beta \in \mathbb{F}$ is called the \textbf{exponent of $H$}. For a set $X\subset\R^n$, the {\bf $\beta$-horn neighbourhood of amplitude $\eta$ of $X$} is the germ, at 0, of the following set: 
	$$
	\mathcal{H}_{\beta,\eta}(X):=\bigcup_{x\in X}\{z\in \R^n\mid \|x-z\|<\eta \|x\|^\beta\}.
	$$
\end{definition}

\begin{remark}\label{Rem: NE HT condition}
	It was proved in \cite{birbrair1999local} that the exponent $\beta$ of either an H\"older triangle $T$ or a horn $H$ is an inner bi-Lipschitz invariant. Moreover, if $\phi: \R^n\to\R^n$ is an outer bi-Lipschitz map, then the image of a $\beta$-horn neighbourhood of $X\subset \R^n$ is contained in a $\beta$-horn neighbourhood of $\phi(X)$. 
\end{remark}

\subsection{Ambient Lipschitz equivalence and tangent cones}\label{subsec:2.5}

In this subsection, we establish the necessary preliminaries for the study of Lipschitz geometry of zeros of mixed polynomials.

\begin{definition}\label{V}
	Let $f$ and $g$ be mixed polynomials. We say that $f$ is \textbf{bi-Lipschitz (topologically) $V$-equivalent to} $g$ if there exists an outer bi-Lipschitz map (homeomorphism) germ $\phi:(V(f),0) \to (V(g),0)$. If $\phi$ extends to an ambient bi-Lipschitz map (ambient homeomorphism)  $\tilde{\phi}:(\C^2,0) \to (\C^2,0)$, we say that $f$ is \textbf{ambient bi-Lipschitz (ambient topologically) $V$-equivalent to} $g$. 
\end{definition}

\begin{definition} Let $K \subseteq \R$ be a connected set containing the origin. Let $f_{\varepsilon} : \C^2 \to \C$, $\varepsilon \in I$ be a smooth family of mixed polynomials. We say that the family \(\{f_\varepsilon\}_{\varepsilon \in \R}\) is \textbf{bi-Lipschitz (topologically) $V$-trivial} (or \textbf{ambient bi-Lipschitz (ambient topologically) $V$-trivial}) along $I$ if for any $\varepsilon \in I$,   $f_{\varepsilon}$ is bi-Lipschitz (topologically) $V$-equivalent (or ambient bi-Lipschitz (ambient topologically) $V$-equivalent) to $f_{0}$. When the set $I$ is understood, we say that the family is \textbf{bi-Lipschitz (topologically) $V$-trivial} (or \textbf{ambient bi-Lipschitz (ambient topologically) $V$-trivial}).
\end{definition}

One of the main objects in ambient Lipschitz geometry of set germs is its tangent cone. One of its many interesting properties is that an ambient bi-Lipschitz map between germs induces an ambient bi-Lipschitz map between their tangent cones.

\begin{definition}
Let $A\subset \R^n$ be a subanalytic set and $x_0 \in \overline{A}$, then the  \textbf{tangent cone} of $A$ at $x_0$ is 
\begin{equation*} \label{tangcone}
\begin{aligned}
C(A, x_0) := \{v; &\ \exists \alpha:[0, \epsilon) \to \R^n \mid \alpha(0) = x_0, \alpha((0, \epsilon)) \subset A, \ \&
\  \alpha(t) - x_0 = tv + o(t)\}.
\end{aligned}
\end{equation*}
\end{definition}

\begin{theorem}\label{cone-sampas}
    Let $X,Y \subseteq \R^n$ and $x_0 \in \overline{X}$, $y_0 \in \overline{Y}$. Suppose that the germs $(X,x_0)$ and $(Y,y_0)$ are (ambient) bi-Lipschitz equivalent. Then, the tangent cones $(C(X,x_0),x_0)$ and $(C(Y,y_0),y_0)$ are also (ambient) bi-Lipschitz equivalent.
\end{theorem}

\begin{proof}
    See \cite[Theorem~3.2]{Sampaio:2016}.
\end{proof}

The next lemma allows us to extend radially a bi-Lipschitz map defined on the sphere.

\begin{lemma}\label{lemma-extension}
    Let $\phi: \mathbb{S}^n \to \mathbb{S}^n$ be a bi-Lipschitz map. Consider the map  $\Phi: \R^{n+1} \to \R^{n+1}$ given as 
    $$\Phi(\bm{z})=\begin{cases} \|z\| \phi \left(\dfrac{z}{\|z\|}\right),\ z \neq 0, \\ 0 ,\ z=0.\end{cases}. $$
    Then, $\Phi$ is a bi-Lipschitz map.
\end{lemma}

\begin{proof}
    Notice first that $\Phi$ is a well-defined, continuous map whose inverse $\Phi^{-1}: \R^{n+1} \to \R^{n+1}$ is also well-defined, continuous and is given as
    $$\Phi^{-1}(\bm{z})=\begin{cases} \|z\| \phi^{-1} \left(\dfrac{z}{\|z\|}\right),\ z \neq 0, \\ 0 ,\ z=0.\end{cases}. $$
    Therefore, $\Phi$ is a homeomorphism. We will show that if $\phi$ is $C$-bi-Lipschitz, then $\Phi$ is $(C+1)$-bi-Lipschitz. Let $p,q \in \R^{n+1}$ and suppose, without loss of generality, that $\|p\| \ge \|q\|$. If $q=0$, then $\|\Phi(p)-\Phi(q)\|=\|p\| \le(C+1)\|p-q\|$. If $q \ne 0$ and $\|p\|=\|q\|$, then $$\|\Phi(p)-\Phi(q)\|=\|p\|.\left\|\phi\left(\frac{p}{\|p\|}\right)-\phi\left(\frac{q}{\|q\|}\right)\right\| \le C\|p\|.\left\|\frac{p}{\|p\|}-\frac{q}{\|q\|}\right\|=C.\|p-q\|\le (C+1)\|p-q\|.$$
    If $\|p\|>\|q\|>0$, then let $\tilde p=\frac{\|q\|}{\|p\|}.p$. Notice that $\|\tilde p\|=\|q\|$ and $\angle p \tilde pq \ge 90^{\circ}$. Then, by the triangular inequality and by the previous case, we have
    $$\|\Phi(p)-\Phi(q)\|\le\|\Phi(p)-\Phi(\tilde p)\|+\|\Phi(\tilde p)-\Phi(q)\|\le \|p-\tilde p\|+C\|\tilde p-q\|\le(C+1).\max\{\|p-\tilde p\|,\|\tilde p-q\|\}.$$
    On the other hand, since $\angle p\tilde pq \ge 90^{\circ}$, we have $\max\{\|p-p'\|,\|p'-q\|\} \le \|p-q\|$, and hence $$\|\Phi(p)-\Phi(q)\|\le(C+1).\max\{\|p-\tilde p\|,\|\tilde p-q\|\} \le (C+1)\|p-q\|.$$ Analogously, we have $\|\Phi^{-1}(p)-\Phi^{-1}(q)\| \le (C+1)\|p-q\|$, and the lemma follows.
\end{proof}

\begin{remark}
    Lemma \ref{lemma-extension} also holds for any Hilbert space, since $\angle p\tilde pq \ge 90^{\circ}$ is equivalent to $\langle p-\tilde p,q-\tilde p \rangle \le 0$.
\end{remark}

In order to construct suitable ambient bi-Lipschitz maps between zeros of mixed polynomials in families, we need the concept of ambient bi-Lipschitz isotopy (see \cite[Definition 6.1]{medeiros2023amb}).

\begin{definition}\label{amb-isotopy-equiv}
Let $K\ge1$ and let $(X,0), (X_0,0), (X_1,0) \subseteq (\mathbb{R}^n,0)$ such that $(X_1,0), (X_2,0) \subseteq (X,0)$. We say that $(X_1,0), (X_2,0)$ are \textbf{ambient ($K$)-bi-Lipschitz isotopic in $(X,0)$} if there is a continuous map $\Phi : (X,0)\times [ 0,1] \to (X,0)$ such that, if we denote $\Phi_{\tau}(p) = \Phi(p,\tau)$, then:
\begin{enumerate}
\item $\Phi_{\tau}: (X,0) \to (X,0)$ is an $K$-bi-Lipschitz map, for all $0 \le \tau \le 1$;
\item $\Phi_0 = id_{(X,0)}$;
\item $\Phi_1((X_1,0))=(X_2,0)$.

\end{enumerate}

The map $\Phi$ is an \textbf{ambient ($K$)-bi-Lipschitz isotopy in $(X,0)$, taking $(X_1,0)$ into $(X_2,0)$}. We also say that the isotopy $\Phi$ is \textbf{invariant on $(Y,0) \subset (X,0)$} if $\Phi_{\tau}|_{(Y,0)}=id_{(Y,0)}$, for all $0 \le \tau \le 1$.
\end{definition}

\begin{remark}
    For each $t_0,t\in [0,1]$, the transition map $\psi_{t_0\to t}:=\Phi_{t}\circ \Phi_{t_0}^{-1}$ is a $K^2$-bi-Lipschitz map.
\end{remark}

The next result is useful to unite two inner bi-Lipschitz maps into a single inner bi-Lipschitz map. For LNE sets, such a map is also outer bi-Lipschitz.

\begin{lemma}[Lipschitz Gluing Lemma] \label{Gluing Lipschitz Lemma}
	Let $X_1,X_2 \subseteq \mathbb{R}^n$, $Y_1,Y_2 \subseteq \mathbb{R}^m$ be closed arcwise connected sets such that $X_1\cap X_2, Y_1 \cap Y_2 \ne \emptyset$. Suppose $\phi_1 : X_1 \to Y_1$ and $\phi_2 : X_2 \to Y_2$ are inner bi-Lipschitz maps satisfying $\phi_1(p)=\phi_2(p)$, for every $p\in X_1 \cap X_2$, and $\phi_1^{-1}(q)=\phi_2^{-1}(q)$, for every $q\in Y_1 \cap Y_2$. Then, if $X=X_1\cup X_2$ and $Y=Y_1 \cup Y_2$, the map $\phi: X \to Y$ given by $\phi(p)=\phi_i(p)$, if $p\in X_i$ ($i=1,2$), is an inner bi-Lipschitz map.
\end{lemma}

\begin{proof}
    See \cite[Proposition 2.19]{medeiros2023amb}.
\end{proof}

By \cite[Theorem 10.6]{medeiros2023amb}, we know that every two horn germs in $(\R^3,0)$ with the same exponent are ambient bi-Lipschitz equivalent. By \cite[Proposition 8.5 and 4.7]{medeiros2023amb}, we can decompose the ambient space of a family of such horns into a family of suitable regions delimited by synchronised triangles (see \cite[Definitions 4.1 and 4.7]{medeiros2023amb}). Then, applying \cite[Theorem 6.3]{medeiros2023amb} to such a family of regions, we obtain the following result.
\begin{theorem}\label{teo-gordurinha}
    Let $\beta \in \mathbb{Q}_{\ge 1}$ and let $\{(H_{\tau}(\varepsilon),0)\}_{\tau, \varepsilon \in [0,1]}$ be a family of LNE $\beta$-horns in $(\R^3,0)$ that vary continuously on $s$ and $\varepsilon$, with bounded derivatives with respect to $s$ and $\varepsilon$. Then, there is a family of ambient bi-Lipschitz isotopies $\{\Phi_{s}\}_{s \in [0,1]}$, with $\Phi_{\tau}: (\R^3,0)\times[0,1]\to (\R^3,0)$ being an ambient bi-Lipschitz isotopy taking $H_{\tau}(0)$ into $H_{\tau}(1)$ an such that $\|\Phi_s(p,\varepsilon)\|=\|p\|$, for all $s, \varepsilon \in [0,1]$ and $p\in (\R^3,0)$. Moreover, if we denote $\Phi_{\tau,\varepsilon}(p)$ as $\Phi_{\tau}(p,\varepsilon)$, then the maps $\Phi_{\tau,\varepsilon}$ vary continuously on $s$ and $\varepsilon$, with bounded derivatives with respect to $s$ and $\varepsilon$.
\end{theorem}

\section{Some preparation results}
\label{somepreparationresults}
\subsection{Obstruction locus in mixed polynomials}
\label{obstructionlocus}
Let \(f, \theta:\C^2 \to \C\)  be mixed polynomials, with $f$ being a radial mixed polynomial of radial-type \((P;d)=(p_1,p_2;d)\), \(k_f:=\frac{p_1}{p_2}\) and $d(P;\theta)\ge d$. Kerner and Mendes proved in \cite{Kerner-Mendes} that, under the condition \(\Sigma_{L}(f)=\{0\}\), the family \(\{f+\varepsilon \theta\}_{\varepsilon \in [0,1]}\) is ambient bi-Lipschitz $V$-trivial. Here, \(\Sigma_{L}(f)\) denotes the obstruction locus of \(f\) (see \cite[Section 3.2, Equation (6)]{Kerner-Mendes}) and for the case of radial mixed polynomial, or equivalently, weighted homogeneous real map $f$ of weight-type $(p_1,p_1,p_2,p_2;d,d)$, it follows directly from the formula that
 \begin{equation}\label{obst-locus}
    \Sigma_{L}(f)=
\begin{cases} 
 Sing(V(f)) \cup Sing(V(f) \cap \{u=0\}),  &\text{ if } k_f<1, \\ 
Sing(V(f)) \cup Sing(V(f) \cap \{v=0\}),  &\text{ if } k_f >1,  \\
Sing(V(f)),  &\text{ if } k_f =1. 
   \end{cases}
\end{equation} 

We have the following result as an immediate consequence of \cite[Corollary 3.7]{Kerner-Mendes}, applied in the case of radial mixed polynomials in two variables.  
\begin{theorem} \label{Lips-triviality}  Let \(\{f+\varepsilon \theta\}_{\varepsilon \in I}\) be a family of mixed polynomials such that $f$ is radial mixed polynomial of radial-type $(P;d)$, \(d(P;\theta) \geq  d\), and $\Sigma_{L}(f)=\{0\}$. If the inequality is strict, then the family is ambient bi-Lipschitz $V$-trivial for $I=\R$. If we have equality, then the family is ambient bi-Lipschitz $V$-trivial for $I$ being a sufficiently small neighbourhood of $0$.
\end{theorem}

\begin{remark}
    Throughout this paper, we will consider the following notation: let \(f\) be an IND and $\Gamma_{\mathrm{inn}}$-nice mixed polynomial, and let \(\theta\) be a mixed polynomial such that \(d(P_i; \theta) \geq d(P_i;f)\) holds for each $P_i \in \mathcal{P}_{\mathrm{inn}}(f)$. The parameter interval $I$ in the family $\{f+\varepsilon\theta\}_{\varepsilon \in I}$ is always a sufficiently small neighbourhood of $0$, if \(d(P_i; \theta) = d(P_i;f)\), for some $P_i \in \mathcal{P}_{\mathrm{inn}}(f)$, and the interval $I$ is $\R$, if \(d(P_i; \theta) = d(P_i;f)\), for every $P_i \in \mathcal{P}_{\mathrm{inn}}(f)$.
\end{remark}

 We will see that $\Sigma_{L}(f)$ is closely related to $k_f$ and the inner non-degeneracy condition.

Bode and Sanchez Quiceno proved in \cite{BodeSanchez2023} that a radial mixed polynomial satisfies \(Sing(V(f))=\{0\}\) if and only if $f$ is IND. Thus, this non-degeneracy condition is necessary for a radial mixed polynomial to satisfy $\Sigma_{L}(f)=\{0\}$. On the other side, we characterise the condition $\Sigma_{L}(f)=\{0\}$ from the point of view of the Newton boundary by the following lemma.
 \begin{lemma}\label{radial-obst}
 Let $f$ be a radial IND mixed polynomial. Then, $\Sigma_{L}(f)=\{0\}$ if, and only if, $f$ satisfies one of the following conditions
 \begin{enumerate}
 \item[(i)] \(k_f=1\),
 \item[(ii)]  \(k_f>1\) and \(f\) is \(u\)-convenient,
 \item[(iii)]  \(k_f<1\) and \(f\) is \(v\)-convenient.
 \end{enumerate}
 \end{lemma} 
 \begin{proof}
($\mathbf{\Leftarrow}$) If \(k_f=1\), then  by Equation~\eqref{obst-locus} follows that \(\Sigma_{L}(f)=Sing(V(f))\). Thus, as $f$ is IND \(\{0\}= Sing(V(f))=  \Sigma_{L}(f)\).  

If $k_f>1$ and $f$ is $u$-convenient, then the Newton boundary of $f$ contains a lattice point $(s,0)$ such that $s p_{1}= d$. Moreover, the Newton boundary does not contain a lattice point $(n,1)\in \mathbb{Z}_{\geq 0}^2$ with $n<s$. Indeed, if $np_{1}+p_{2}=d$, it would imply $np_{1}+p_{2}=sp_{1}$, leading to $k_f=\frac{1}{s-n}\leq1$, which contradicts our hypothesis that $k_f>1$. Thus, the mixed polynomial $f$ can be expressed as $f(u,v)=f_{\Delta_u}(u)+h(u,v)$, where $|v|^2$ divides $h$, and $\Delta_u=(s,0)$ is the $u$-extreme vertex of $f$. 

At a point $(u,0)$, we have: 
\begin{align*}
f(u,0)&=f_{\Delta_u}(u),\\
f_{u}(u,0)&=(f_{\Delta_u})_{u}(u),\\
f_{\bar{u}}(u,0)&=(f_{\Delta_u})_{\bar{u}}(u),\\
f_{v}(u,0)&=0,\\
f_{\bar{v}}(u,0)&=0.
\end{align*}

If $(u,0)$ is a regular point  of $V(f)$ with $u\neq 0$, then  the system
$$\begin{cases}
f_{\Delta_u}(u)=0\\
|(f_{\Delta_u})_u(u)|^2-|(f_{\Delta_u})_{\bar{u}}(u)|^2=0
\end{cases}$$
has no solution of the form $(u, 0)$ with $u\neq 0$. This is equivalent to $f_{\Delta_u}$ being non-degenerate.  In \cite{AraujoBodeSanchez}, it was shown that the non-degeneracy of  \(f_{\Delta_u}\) is equivalent to \(V(f_{\Delta_u})= \{u=0\}\).  Therefore, 
\begin{equation}\label{nondeguvertex}
V(f)\cap \{v=0\} = V(f_{\Delta_u})\cap  \{v=0\} = \{0\}.
\end{equation}
 From Equations \eqref{obst-locus} and \eqref{nondeguvertex}, as \(k_f>1\), it follows that \(\Sigma_{L}(f)= Sing(V(f)).\) Since $f$ is IND, we conclude that \(\Sigma_{L}(f)=\{0\}\). 

If $k_f<1$ and $f$ is $v$-convenient, then we follow the same arguments as above and prove that
\begin{equation}\label{nondegvvertex}
V(f)\cap \{u=0\} = V(f_{\Delta_v})\cap  \{u=0\} = \{0\}.
\end{equation}
 From Equations 
 \eqref{obst-locus} and \eqref{nondegvvertex}, as \(k_f<1\) and $f$ is IND, it follows that \(\Sigma_{L}(f)=\{0\}\).

($\mathbf{\Rightarrow}$)
Let $f$ be a radial mixed polynomial satisfying \(\Sigma_{L}(f)=\{0\}\). Suppose that $f$ does not satisfy one of the conditions (i)-(iii) in this Lemma. Thus, $f$ satisfies one of the following cases: 
\begin{itemize}
\item[(a)] $k_f>1$ and  $f$ is not $u$-convenient (here, $f$ could be $v$-convenient or not); 
\item[(b)] $k_f<1$ and  $f$  is not $v$-convenient (here, $f$ could be $u$-convenient or not). 
\end{itemize} 
Suppose $k_f>1$ and $f$ is not $u$-convenient, the other case is studied analogously.  Since $k_f>1$ and $f$ is IND, Equation \eqref{obst-locus} implies    \[\Sigma_{L}(f)=Sing(V(f)) \cup Sing(V(f) \cap \{v=0\})=\{0\}\cup Sing(V(f) \cap \{v=0\})=Sing(V(f) \cap \{v=0\}).\]  Since $f$ is not $u$-convenient, it follows that $\{v=0\} \subset V(f)$. Thus,
$$\Sigma_{L}(f)=Sing(V(f) \cap \{v=0\})=\{v=0\}\neq \{0\},$$ 
which is a contradiction.   
 \end{proof}

\subsection{Separating horn neighbourhoods}
\label{separatinghorns}
To deal with the case $\Sigma_{L}(f)\ne \{0\}$, we need first to ``prepare the ambient'' around the links of the zeros of their respective face functions, by locating their obstruction locus (Lemma \ref{radial-obst-Newton}), then finding suitable horn neighbourhoods isolating such singularities (Lemma \ref{generalhorn}) and constructing ambient bi-Lipschitz maps for each component (Lemmas \ref{lemma:lipeomorphismlocal} and \ref{lemma:overdeformation}). Finally, we ``unify'' such ambient bi-Lipschitz maps using Lemma \ref{Lemma:Union} and prove triviality even when the obstruction locus of $f$ is non-trivial.
\begin{lemma}\label{radial-obst-Newton}
	Let $f$ be an IND mixed polynomial that is $\Gamma_{\mathrm{inn}}$-nice. Take $\mathcal{P}_{\mathrm{inn}}(f)=\{P_1,\dots,P
	_N\}$. Then, 
	\begin{equation*}
		\begin{aligned}
			\text{For } i = 1: \quad & \Sigma_{L}(f_{P_1}) \subseteq \{ v = 0 \}, \\
			\text{For } 2 \leq i \leq N - 1: \quad & \Sigma_{L}(f_{P_i}) \subseteq \{ uv = 0 \}, \\
			\text{For } i = N: \quad & \Sigma_{L}(f_{P_N}) \subseteq \{ u = 0 \}.
		\end{aligned}
	\end{equation*}
\end{lemma} 
\begin{proof}
	Let \(i=1\). It follows that 
		\begin{align}\label{fp1general}
		Sing(V(f_{P_1})) \subseteq \{v=0\} \text{ and } V(f_{P_1}) \cap \{v=0\} \subseteq \{v=0\}. 
	\end{align} 
	Then, substituting Equation \eqref{fp1general} in Equation \eqref{obst-locus} for $f_{P_1}$ and \(k_1 \geq 1\), we get \[\Sigma_{L}(f_{P_1}) \subseteq \{v=0\}.\] If $k_1 <1$, then \(f_{\Gamma_{\mathrm{inn}}}\) is $v$-convenient and by Equation \eqref{nondegvvertex} 
	\begin{equation}\label{fp1generalu0}
		V(f_{P_1}) \cap \{u=0\}= \{0\}. 
	\end{equation}
	Thus, substituting Equation \eqref{fp1generalu0} in  Equation \eqref{obst-locus} for  $f_{P_1}$ and \(k_1 <1\), we get \[\Sigma_{L}(f_{P_1})\subseteq \{v=0\}.\]
	
	\noindent Let \(2 \leq i \leq N-1\). It follows from the non-degeneracy condition that \[Sing(V(f_{P_i})) \subseteq \{uv=0\}.\] Thus, by Equation \eqref{obst-locus} for $f_{P_i}$ and \(k_i>0\), we obtain \(\Sigma_{L}(f_{P_i})\subseteq \{uv=0\}\). 
	\vspace{0.2cm}
	
	\noindent The case \(i=N\) follows analogously to the case \(i=1\), making appropriate changes, especially considering Equations \eqref{obst-locus} and \eqref{nondeguvertex}.
\end{proof}
\begin{lemma}\label{generalhorn}
 Let $f$ be an IND mixed polynomial that is $\Gamma_{\mathrm{inn}}$-nice. Then there exist a small enough $\rho>0$ and horn neighbourhoods $\mathcal{H}_{i}$ of $V_{P_i},\ i=1,\dots, N$, with $P_i \in \mathcal{P}_{\mathrm{inn}}(f)$, satisfying 
 $$ \mathcal{H}_{i} \cap \mathcal{H}_{j} \cap B_{\rho}^4(0) =\{0\}, \text{ if } i \neq j.$$
 Here, $B_{\rho}^4(0)$ denotes the closed ball in $\R^4$ with centre $0$ and radius $\rho$.
\end{lemma}
\begin{proof}
For \(i=1,2,\dots, N-1\) we define  
\begin{equation*}
 \begin{aligned}
 m_i:=\min \{ |u_*| \mid  (u_*,\rme^{\rmi t_*}) \in L_i \}, \,
 M_i:=\max \{ |u_*| \mid  (u_*,\rme^{\rmi t_*}) \in L_i \},
 \end{aligned} 
 \end{equation*}
 where $L_i$ is defined as in Subsection~\ref{nested}. Notice that since $L_i$ is compact, $m_i$ and $M_i$ are well defined for each of these $i$. By the radial homogeneity of \(f_{P_i}\), we have, for each $i=1,\dots, N-1$ and $(u_*,\rme^{\rmi t_*}) \in L_i$, $(r^{k_i}u_*,r\rme^{\rmi t_*}) \subset V({f_{P_i}})$. For \(i=1,2,\dots, N-1\), $\epsilon_i >0$ small enough, define   
\begin{multline*}
A_{i}(\epsilon_i):= \{(u,v) \in \C^2 \mid \exists \lambda \in \R_{\geq 0}, (m_i-\epsilon_i)\lambda^{k_i}\leq  |u|\leq  (M_i + \epsilon_i)\lambda^{k_i} \, \& \, \ (1-\epsilon_i)\lambda \leq |v|\leq  (1+\epsilon_i)\lambda\},
\end{multline*}
\begin{equation*}
C_{i}(\epsilon_i)= \{(\lambda^{p_{i,1}} s_1, \lambda^{p_{i,2}} s_2 ) \mid  (s_1,s_2) \in A_{i}(\epsilon_i) \cap \mathbb{S}^3  \ \& \  \lambda \in [0,1] \}.
\end{equation*}
For \(i=2,3,\dots, N\), we define, similarly to what was done previously,
\begin{equation*}  
 \begin{aligned}
 m_{\underline{i}}:=\min \{ |v_*| \mid (\rme^{\rmi \varphi}, v_*) \in L_{\underline{i}} \} \, , \, 
 M_{\underline{i}}:=\max \{ |v_*| \mid (\rme^{\rmi \varphi}, v_*) \in L_{\underline{i}} \}.
 \end{aligned} 
 \end{equation*}
and for $\epsilon_{\underline{i}}>0$ small enough, we also define
\begin{multline*}\label{Aunder}
A_{\underline{i}}(\epsilon_{\underline{i}}):= \{(u,v) \in \C^2 \mid  \exists \lambda \in \R_{\geq 0}, (1-\epsilon_{\underline{i}})\lambda \leq  |u|\leq  (1+\epsilon_{\underline{i}})\lambda \, \& \, (m_{\underline{i}}-\epsilon_{\underline{i}})\lambda^{\frac{1}{k_i}}\leq  |v|\leq  (M_{\underline{i}} + \epsilon_{\underline{i}})\lambda^{\frac{1}{k_i}} \},
\end{multline*}
\begin{equation*}\label{Cunder}
C_{\underline{i}}(\epsilon_{\underline{i}})= \{(\lambda^{p_{i,1}} s_1, \lambda^{p_{i,2}} s_2 ) \mid  (s_1,s_2) \in A_{\underline{i}} \cap \mathbb{S}^3 \ \& \  \lambda \in [0,1]\}.
\end{equation*}

\begin{afirm}
    \(A_{i} \cap B_1^4(0) = C_{i}\) and  \(V_{P_i}\cap B_1^4(0) \subset  C_{i}\), for each \(i=1,\dots, N-1\). Moreover,  \(A_{\underline{N}} \cap B_1^4(0) = C_{\underline{i}}\) and \(V_{P_i}\cap B_1^4(0) \subset  C_{\underline{i}}\) for \( i=2,\dots, N\).
\end{afirm}

\begin{proof}
    Let $(s_1,s_2) \in A_{i} \cap \mathbb{S}^3$. Then there exists $\lambda_0\in [0,1]$ such that  
\begin{align*}
(m_i-\epsilon_i)\lambda_0^{k_i}\leq |s_1|\leq  (M_i + \epsilon_i)\lambda_0^{k_i},\\ \ (1-\epsilon_i)\lambda_0 \leq |s_2|\leq  (1+\epsilon_i)\lambda_0.
\end{align*}
Let $\alpha \in (0,1)$. Multiplying by $\alpha^{p_{i,1}}$ and $\alpha^{p_{i,2}}$ the first and second inequalities, respectively, we get 
\begin{align*}
(m_i-\epsilon_i)\lambda_0^{k_i}\alpha^{p_{i,1}}\leq  |\alpha^{p_{i,1}}s_1|\leq  (M_i + \epsilon_i)\lambda_0^{k_i}\alpha^{p_{i,1}},\\ \ (1-\epsilon_i)\lambda_0\alpha^{p_{i,2}} \leq |\alpha^{p_{i,2}}s_2|\leq (1+\epsilon_i)\lambda_0\alpha^{p_{i,2}}.
\end{align*}
Since \((\lambda_0\alpha^{p_{i,2}})^{k_i}=\lambda_0^{k_i}\alpha^{p_{i,1}}\) and \(\lambda_0\alpha^{p_{i,2}}<1\), for all $\alpha<1$, then $C_{i} \subset A_{i} \cap B_1^4(0) $. By performing the reciprocal calculations, we obtain $A_{i} \cap B_1^4(0) \subset C_{i}$, and hence $C_{i} = A_{i} \cap B_1^4(0)$.
To prove that $V_{P_i}\cap B_1^4(0) \subset  C_{i}$ notice that 
$(\lambda^{k_i} u_*, \lambda \rme^{\rmi t_*}) \in V_{P_i}, \lambda >0$ where $(u_*, \rme^{\rmi t_*}) \in L_i$, then   
\begin{align*}
(m_i-\epsilon_i)\lambda^{k_i}\leq |\lambda^{k_i} u_*|\leq (M_i + \epsilon_i)\lambda^{k_i},\\ \ (1-\epsilon_i)\lambda \leq |\lambda \rme^{\rmi t_*}|\leq (1+\epsilon_i)\lambda.
\end{align*}
Thus, $V_{P_i} \subset A_{i}$, which implies 
$V_{P_i} \cap B_1^4(0) \subset C_{i}$. Analogously, the claim follows for the cases involving the sets $A_{\underline{i}}$ and $C_{\underline{i}}$. 

Take $\ell$ as the minimal $i$ such that $k_i<1$. If $k_i>1,$ for all $i$, then $\ell=N+1$ (see Definition \ref{DEF: decompostion links}). Define:
\begin{equation*}\label{hornneig1}
\begin{aligned}
&\mathcal{H}_1:= C_{1}(\epsilon_1), \\
      & \ \ \ \  \vdots \\
  & \mathcal{H}_{\ell-1}= C_{\ell-1}(\epsilon_{\ell-1}),\\
   &  \mathcal{H}_{\ell}= C_{\underline{\ell}}(\epsilon_{\underline{\ell}}),\\
   & \ \ \ \ \  \vdots  \\
&\mathcal{H}_N:= C_{\underline{N}}(\epsilon_{\underline{N}}).\\
\end{aligned}
 \end{equation*}
Here, $\epsilon_i$ and $\epsilon_{\underline{i}}$ are chosen sufficiently small.  
\end{proof}
   
  We claim that there is $\rho>0$ such that $ \mathcal{H}_{i} \cap \mathcal{H}_{j} \cap B^4_{\rho}(0) =\{0\}, \text{ if } i \neq j.$ Since \(\mathcal{H}_i\) and \(\mathcal{H}_j\) are clearly horn neighbourhoods, it is sufficient to prove that  in a neighbourhood of the origin, any two arcs $\gamma_{s_i,i}\subset \mathcal{H}_i$ and  $\gamma_{s_j,j}\subset \mathcal{H}_j$ only intersect at the origin. Suppose $i<j$. We have that $\gamma_{s_i,i}(\lambda)=\gamma_{s_j,j}(\tilde{\lambda})$ if and only if $(s_{i,1}\lambda^{k_i},s_{i,2}\lambda)=(s_{j,1}\tilde{\lambda}^{k_{j}},s_{j,2}\tilde{\lambda})$.
Since $f$ is nice and IND with respect to  $\Gamma_{\mathrm{inn}}(f)$, it follows that at least one of the numbers $s_{i,1}$ or $s_{j,1}$ is non-zero, and likewise, at least one of $s_{i,2}$ or $s_{j,2}$ is non-zero. Consequently, the system 
\begin{equation*}
\begin{aligned}
 s_{i,2}\lambda&= s_{j,2}\tilde{\lambda},  \\
s_{i,1}\lambda^{k_i}&=s_{j,1}\tilde{\lambda}^{k_{j}}, 
\end{aligned}
\end{equation*}
has an isolated solution at the origin in the case $k_i\neq k_j$, and the lemma follows.
\end{proof}
\begin{obs}\label{linkofTypeIII}
	 The link of the zero set of a semi-radial mixed polynomial $f$ of radial-type $(P;d)$, such that $k_f>1$ and $f$ is not $u$-convenient (or $k_f<1$ and $f$ is not $v$-convenient) can be seen more simple in terms of $L_u$ and $L_v$ (see Remark \ref{equivalenceoflinks}). Let us see how this is done in the case $k_f>1$ and $f_{P}$ is not $u$-convenient, as the other case is analogous.
     
     The set \(\mathcal{P}_{\mathrm{inn}}(f)\) coincides with \(\mathcal{P}_{\mathrm{inn}}(f_{P_1})= \{P_1,P_2\}\), where $P_1:=P$ and \(P_2:=(1,1)\). Notice that  \(\Gamma_{\mathrm{inn}}(f_{P_1})=\Gamma_{\mathrm{inn}}(f)\). Applying Theorem~\ref{thm:link-triviality1} separately for \(f_{P_1}\) and \(f\), we get isotopic links. However, their nested representations are different, as the face function $(f_{P_1})_{P_2}$ is not necessarily equal to $f_{P_2}$. Thus, $L_{f_{P_1}} \simeq \mathbf{L}([L_1],L_v)$ and $L_{f} \simeq \mathbf{L}([L_1],[L_{\underline{2}}]')$, where $L_{\underline{2}}$ is associated with \(f_{P_2}\). The variety $V(f_{P_1})$ is ambient homeomorphic to $V(f)$ by a germ of homeomorphisms \(\phi: (\C^2,0) \to (\C^2,0)\). The set \(V(L_1;f_{P_1})\) denotes the component of $V(f_{P_1})$ whose intersection with $\C \times \mathbb{S}^1$ is the link $L_1$, and 
		\(V(L_1;f):=\phi(V(L_1;f_{P_1}))
		\). Analogously, we define \(V(L_v;f_{P_1})\)  and \(V(L_v;f)\).
\end{obs}

\begin{proposition}\label{lemma:lipeomorphismlocal}
Let $f$ be an IND mixed polynomial that is $\Gamma_{\mathrm{inn}}$-nice. Suppose that $ \mathcal{P}_{\mathrm{inn}}(f)=\{P_1, \dots, P_N\}$, with $N \ge 2$. Then, for any $1 \le i \le N$ there is $\beta_i \in \mathbb{Q}_{\ge 1}$, $\eta_i>0$ small enough, and a subanalytic bi-Lipschitz map $\Phi_i: (\C^2 \setminus \mathcal{H}_{\beta_i,\eta_i}(\Sigma_L(f_{P_i})),0) \to (\C^2 \setminus \mathcal{H}_{\beta_i,\eta_i}(\Sigma_L(f_{P_i})),0) $ such that  \(\Phi_i (V(L_{P_i};f_{P_i}),0)=(V(L_{P_i};f),0)\). Moreover, we can choose $\beta_1=1$ if $k_1>1$ and $f_{P_1}$ is not $u$-convenient, and $\beta_N=1$ if $k_N<1$ and $f_{P_N}$ is not $v$-convenient. 
\end{proposition}
\begin{proof}
	The \(\Gamma_{\mathrm{inn}}\)-niceness and Lemma \ref{radial-obst-Newton} imply:
	\begin{equation}\label{eq:localzeros}
		\begin{aligned}
			\text{For } i = 1: \quad & V(L_{P_i};f_{P_i}) \subset \mathbb{C}^2 \setminus \{v=0\}, \, \Sigma_L(f_{P_i}) \subseteq \{v=0\};\\
			\text{For } 2 \leq i \leq N - 1: \quad & V(L_{P_i};f_{P_i}) \subset \mathbb{C}^2 \setminus \{uv=0\}, \, \Sigma_L(f_{P_i}) \subseteq \{uv=0\}; \\
			\text{For } i = N: \quad & V(L_{P_i};f_{P_i}) \subset \mathbb{C}^2 \setminus \{u=0\}, \, \Sigma_L(f_{P_i}) \subseteq \{u=0\}.
		\end{aligned}
	\end{equation}
		We know that Theorem~\ref{Lips-triviality} holds only when $\Sigma_L(f)=0$, which is equivalent to $N=1$. However, Kerner and Mendes method \cite[Subsection 3.3]{Kerner-Mendes} allows us to obtain ambient bi-Lipschitz triviality outside a horn neighbourhood $\mathcal{H}_{\beta_i,\eta_i}(\Sigma_L(f_{P_i}))$ of the obstruction locus, for some $\beta_i \in \mathbb{Q}_{\ge 1}$ and $\eta_i>0$ small enough. Thus, considering the family $\{f_{P_i}+ \varepsilon (f-f_{P_i})\}_{\varepsilon \in \R}$, since Equation \ref{eq:localzeros} satisfies $d(P_i;f-f_{P_i})>d(P_i;f_{P_i})$, we find the desired subanalytic bi-Lipschitz maps, and those maps also vary smoothly with $\varepsilon$. 
        
        Finally, if $k_1>1$ and $f_{P_1}$ is not $u$-convenient, then $\operatorname{Cont}(V(L_{P_1};f_{P_1}),\{v=0\})=1$. Therefore, there is $\eta_1'>0$ small enough such that $V(L_{P_1};f_{P_1}) \subset \C^2 \setminus \mathcal{H}_{1,\eta_1'}(\Sigma_L(f_{P_1}))$. Since $(\mathcal{H}_{\beta_1, \eta_1}(\Sigma_L(f_{P_1}),0) \subset (\mathcal{H}_{1, \eta_1'}(\Sigma_L(f_{P_1}),0)$, the restriction of $\Phi_1$ to $(\mathcal{H}_{1, \eta_1'}(\Sigma_L(f_{P_1}),0)$ is the desired map. A similar conclusion holds for $k_N<1$, and $f_{P_N}$ is not $v$-convenient.
\end{proof}

\begin{corolario}\label{biLipslocal} Let $f$ be an IND mixed polynomial that is $\Gamma_{\mathrm{inn}}$-nice. Then, there is a set of triples $\{(\Phi_i, \mathcal{H}_i, \tilde{\mathcal{H}}_i)\}_{i\in I_f }$, satisfying the following conditions:
	\begin{itemize}
            \item[(i)] for each $i\in I_f$, $\Phi_i$ is the bi-Lipschitz map as in Proposition~ \ref{lemma:lipeomorphismlocal};
            \item[(ii)] $\mathcal{H}_i$ is a horn neighbourhood of $V_{P_i}$ and $\tilde{\mathcal{H}}_i=\Phi_i(\mathcal{H}_i)$ is a horn neighbourhood of $\Phi_i(V_{P_i})$;   
		\item[(iii)] for all sufficiently small $\rho>0$ and \(i \neq j\) \[ \mathcal{H}_{i} \cap \mathcal{H}_{j} \cap B^4_{\rho}(0) ={\tilde{\mathcal{H}}}_{i} \cap \tilde{\mathcal{H}}_{j} \cap B^4_{\rho}(0)=\{0\}.\]
	\end{itemize}
\end{corolario}

\begin{proof}
    Conditions $(i)$ and $(ii)$ follow directly from Lemma \ref{generalhorn} and Proposition~\ref{lemma:lipeomorphismlocal}. Notice that, by shrinking $\epsilon_i, \epsilon_{\underline{i}}$ in the proof of Lemma \ref{generalhorn}, these conditions still hold if we change $\mathcal{H}_i$ for any horn neighbourhood $\mathcal{H}_i' \subset\mathcal{H}_i$ of each $V_{P_i}$. The horn neighbourhoods $\mathcal{H}'_i$ can be taken small enough such that its deformations $\Phi(\mathcal{H}'_i)$ are in $\mathcal{H}_i$. Since $(\mathcal{H}_i,0) \cap (\mathcal{H}_j,0)=\{0\}$ for $i\ne j$, if we change each $\mathcal{H}_i$ by its corresponding $\mathcal{H}_i'$, then condition $(iii)$ is also satisfied for some $\rho>0$ small enough. 
\end{proof}
Throughout this subsection, notice that our strategy for investigating the Lipschitz geometry of the zeros of mixed polynomials involves splitting such a zero set into several pieces. We then prove that each of these pieces can be well approximated, in the bi-Lipschitz sense, by a specific set of weighted homogeneous zero sets coming from the face functions $f_P$ with $P \in \mathcal{P}_{\mathrm{inn}}(f)$. When extending this approach to families of polynomials $\{f+\varepsilon\theta\}_{\varepsilon \in \R}$, we also need to investigate how deformations affect these approximations. Specifically, we want to ensure that small perturbations of $f_P$ do not alter their bi-Lipschitz class. This means that for a given weight vector $P$ within $\mathcal{P}_{\mathrm{inn}}(f)$, a perturbation of the form $f_P + \varepsilon \theta_P$ (where $\theta_P$ is not identically zero) should maintain the bi-Lipschitz equivalence for a sufficiently small $\varepsilon$, even though $\theta_P$ has the same radial degree as $f_P$. 

The following lemma addresses a particular case of the previous situation. It shows that we can deform certain radial mixed polynomials on the same radial degree and still preserve bi-Lipschitz equivalence, where the deformation parameter $\varepsilon$ is not necessarily required to be sufficiently small.
\begin{lemma}\label{lemma:overdeformation} 
	Let $M(u,v):=A(u)v+B(u)\bar{v}$ be a radial mixed polynomial of radial-type $(1,1;d)$ such that $V(M)\cap \C^{\{1,2\}}=\emptyset$. If $M^*(u,v)$ is a radial mixed polynomial of radial-type $(1,1;d)$, which does not depend on $v$ ($M^*=M^*(u)$), then
	\begin{itemize}
		\item[(i)] the mixed polynomial $M+\delta M^*, \ \delta \in \mathbb{R}$ is a radial mixed polynomial of radial-type $(1,1;d)$, and  \[\Sigma_{L}(M+\delta M^*)=Sing(V(M+\delta M^*))\subseteq \{u=0\},\]
		\item[(ii)] there is an ambient bi-Lipschitz map $\Phi_{M}$ satisfying  \[\Phi_{M} (V(L_v;M))=V(L_v;M+M^*).\]
	\end{itemize}
\end{lemma}
\begin{proof}
	\noindent \textbf{(i):} Let $P=(1,1)$. We have $d=d(P;M)= d(P;M^*)$, and Proposition~\ref{lemma:lipeomorphismlocal} guarantees only a trivialisation of $\{M+ \varepsilon M^*\}$ for sufficiently  small $\varepsilon$. The condition $V(M)\cap \C^{\{1,2\}}=\emptyset$ implies $|A(u)|\neq |B(u)|, \text{ for all } u\in \mathbb{C}^*$ because, otherwise, if $(u,r\rme^{\rmi t})\in V(M)$, with $r \ne 0$, then
	\begin{equation*}\label{monomial}
		0=r\rme^{-\rmi t}(A(u)\rme^{2\rmi t}+B(u)) \Rightarrow A(u)=B(u)=0 \text{ or } \rme^{2\rmi t}=-\frac{B(u)}{A(u)}.
	\end{equation*}
    The first condition implies $u=0$, and the second condition is impossible since it must hold for all $t$. Therefore, $(u,0)$, $u\neq 0$, is a  regular point of $M$. 
	It is clear  that  $M+\delta M^*$ is radial of radial-type $(P;d)$ for any $\delta$. By  \cite[Corollary 4.11]{AraujoBodeSanchez}, the regularity of the points $(u,0)$, $u\ne 0$, and the independence of $M^*$ from $v$ implies that \(Sing(V(M+\delta M^*))\subset \{u=0\}.\) Therefore, the result follows from Equation~\eqref{obst-locus}, as $M+\delta M^*$  satisfies $k_{M+\delta M^*}=1$.
	
	\vspace{0.2cm}
	
	\noindent \textbf{(ii):} Consider the family $\{M+\delta M^*\}_{\delta \in [0,1]}$. The members of this family are radial of radial-type $(P;d)$. By (i), $Sing(V(M+\delta M^*)) \subset \{u=0\}$, and thus, applying the implicit function theorem and the extension isotopy theorem in $\mathbb{S}^3$, there is a smooth ambient isotopy 
    $\phi_\delta: \mathbb{S}^3 \to \mathbb{S}^3$ taking $L_v$ to the link $V(L_v;M+\delta M^*) \cap S^3$, for each $\delta \in [0,1]$. Since $\mathbb{S}^3$ is compact and $\phi_\delta$ is smooth, we conclude that $\phi_\delta$ is a bi-Lipschitz map. Now consider the map germ $\Phi_{\delta}: (\R^4,0)\to(\R^4,0)$ given as 
    $$\Phi_\delta(\bm{z})=\begin{cases} |z| \phi_\delta \left(\dfrac{z}{|z|}\right),\ z \neq 0, \\ 0 ,\ z=0.\end{cases} $$
    By Lemma \ref{lemma-extension}, the maps $\Phi_{\delta}$ are ambient bi-Lipschitz. Therefore, the map $\Phi_M:=\Phi_1$ is the desired ambient bi-Lipschitz map.
\end{proof}
\begin{obs}	Lemma~\ref{lemma:overdeformation} also works for radial mixed polynomials of the form  $N(u,v)=A(v)u+B(v)\bar{u}$ and deformation $N^*(u,v)$ satisfying the analogue properties. 
\end{obs}

\subsection{A Lipschitz contact criterion}
\label{Lipcontactcriterion}
The zero set $V(f)$ of an IND mixed polynomial $f$ that is $\Gamma_{\mathrm{inn}}$-nice can be approximated by the weighted homogeneous zero sets $V_{P_i}, \ P_i\in \mathcal{P}_{\mathrm{inn}}(f)$. Thus, if we want to study the contact at the origin of two components $V_1$ and $V_2$ of $V(f)$ that are in the same horn neighbourhood $\mathcal{H}_{i}$, then we will consider the contact at the origin of their corresponding weighted homogeneous approximations (see Corollary~\ref{biLipslocal}). However, if these components belong to different horn neighbourhoods, then their contact behaviour can be more complicated, since such zeros and their respective homogeneous parts can have different contact orders at the origin. In the following, we provide a broad criterion for understanding the contact at the origin of components of zeros of mixed polynomials when they are in different horn neighbourhoods. Motivated to preserve the property that the contact order at the origin of such zeros and their respective homogeneous parts are the same, we consider the following two conditions:

\begin{equation}
\tag{A}
\label{Eq:propertyA}
\begin{aligned}
	&\text{For } i,j \in I_f^{(k > 1)}  \text{ with }  i\neq j, \text{ and components } I \subset L_{i,f}  \text{ and } J \subset L_{j,f}: \\
	& \mathrm{proj}_2(I) \cap \mathrm{proj}_2(J) \neq \emptyset \Rightarrow \ 
	(\operatorname{Int}( \mathrm{proj}_2(I)) \cap  \mathrm{proj}_2(J)\neq \emptyset \text{ or }  \mathrm{proj}_2(I) \cap \operatorname{Int} (\mathrm{proj}_2(J))\neq \emptyset).
\end{aligned}
\end{equation}

\begin{equation}
\tag{B}
\label{Eq:propertyB}
\begin{aligned}
	&\text{For } i,j \in I_f^{(k < 1)} \text{ with } i \neq j, \text{ and components } I \subset L_{\underline{i},f}  \text{ and } J \subset L_{\underline{j},f}: \\ 
	&\mathrm{proj}_1(I) \cap \mathrm{proj}_1(J) \neq \emptyset \Rightarrow 
	(\operatorname{Int}( \mathrm{proj}_1(I)) \cap \mathrm{proj}_1(J) \neq \emptyset \text{ or } \mathrm{proj}_1(I) \cap \operatorname{Int} (\mathrm{proj}_1(J)) \neq \emptyset ).
\end{aligned}
\end{equation}
Here, ${\rm proj}_1: \mathbb{S}^1\times \C \to \mathbb{S}^1$ and ${\rm proj}_2: \C\times \mathbb{S}^1 \to \mathbb{S}^1$ are projections. 
\begin{proposition}\label{nontangencycriteria}Let $f$ be an IND mixed polynomial that is $\Gamma_{\mathrm{inn}}$-nice. Then:
\begin{enumerate}
	\item[(i)] If $I,\ J$ and $K$ are components of $L_{i,f}$, $L_{\underline{j},f}$ and $L_{l,f}$, respectively, where  $\ i \in I_{f}^{(k>1)}$, $ j \in I_{f}^{(k<1)}$ and  $l \in I_{f}^{(k=1)}$, then  
	\begin{align*}
		\operatorname{Cont}(V(I;f), V(J;f))=\operatorname{Cont}(V(I;f), V(K;f))=\operatorname{Cont}(V(J;f),& V(K;f))=1.
	\end{align*}
	\item[(ii)]   If $I$ and $J$ are components of $L_{l,f},$ where $\ l \in I_{f}^{(k=1)}$, then $$\operatorname{Cont}(V(I;f), V(J;f))= 1.$$    
	\item[(iii)]  If $I$ and $J$ are components of $L_{i,f},$ and  $L_{j,f}$, respectively, where $i, j \in I_{f}^{(k>1)}$, \(i\leq j\), then 
	\[1\leq \operatorname{Cont}(V(I;f), V(J;f))\leq k_j. \] 
	Moreover,	
	\begin{equation*}
		\operatorname{Cont}(V(I;f), V(J;f))= \begin{cases} 
			1,  &\text{iff $\mathrm{proj}_2(I) \cap \mathrm{proj}_2 (J)= \emptyset$,} \\ 
			k_j,  &\text{if $i=j$ or $f$ satisfies Condition \eqref{Eq:propertyA}.} 
		\end{cases}
	\end{equation*} 
	\item[(iv)]    If $I$ and $J$ are components of $L_{\underline{i},f},$ and  $L_{\underline{j},f}$, respectively, where $i, j \in I_{f}^{(k<1)}$, \(i\leq j\), then 
	\[1\leq \operatorname{Cont}(V(I;f), V(J;f))\leq \frac{1}{k_i}. \] 
	Moreover,	
	\begin{equation*}
		\operatorname{Cont}(V(I;f), V(J;f))= \begin{cases} 
			1,  &\text{iff $\mathrm{proj}_1(I) \cap \mathrm{proj}_1 (J)= \emptyset$,} \\ 
			\frac{1}{k_i},  &\text{if $i=j$ or $f$ satisfies Condition \eqref{Eq:propertyB}.} 
		\end{cases}
	\end{equation*} 
\end{enumerate}
\end{proposition}
\begin{proof} The main ingredient to prove this proposition is to consider real half-branches and compare their tangency orders. Such half-branches are known as test arcs, and they appeared for the first time in \cite{Fernandes2003} as a tool for discovering relations between Puiseux pairs and the Lipschitz geometry of complex plane curves.\\
\noindent\textbf{(i):} Consider two real half-branches, $\tilde{\gamma}_i \in V(I;f)$ and $\tilde{\gamma}_j \in V(J;f)$. By \cite[Lemma~3.3]{Kerner-Mendes}, these correspond to  $\gamma_i \in V(I;f_{P_i})$ and $\gamma_j \in V(J;f_{P_j})$, respectively. Furthermore, we have that 
\begin{equation}\label{eq:3contactlema}
	\operatorname{\operatorname{tord}} (\gamma_ i,\tilde{\gamma}_i)>1 \text{ and } \operatorname{tord}(\gamma_ j,\tilde{\gamma}_j)>1.
\end{equation}
By Remark~\ref{geralzao}, the real half-branches $\gamma_i$ and $\gamma_j$ can be parameterized as follows: 
\begin{equation}\label{Eq:1contactlema}
	\gamma_i(\lambda)= (\lambda^{k_i} u_I(\tau_i(\lambda)), \lambda\rme^{\rmi t_I(\tau_i(\lambda))}) 
\end{equation}
and
\begin{equation*}\label{Eq:2contactlema}
	\gamma_j(\lambda)= (\lambda \rme^{\varphi_J(\psi_j(\lambda))}, \lambda^{\frac{1}{k_j}}v_J(\psi_j(\lambda))),
\end{equation*}
where $\tau_i, \psi_j: \R \to \R$, and $(u_I, \rme^{\rmi t_I})$ and  $(\rme^{\rmi \varphi_{J}},v_J)$ are parametrizations of the knots $I$ and $J$ in $\C \times \mathbb{S}^1$ and $\mathbb{S}^1\times \C$, respectively. Given that $k_i-1,\ \frac{1}{k_j}-1>0$, and that $(u_I(\tau_i(0)),\rme^{\rmi t_I(\tau_i(0))})$ and $(\rme^{\varphi_J(\psi_j(0))},v_J(\psi_j(0)))$ are points in $I$ and $J$, respectively, the following limit exists and satisfies  
\[\lim_{\lambda \to 0}\frac{ ||\gamma_i(\lambda) - \gamma_j(\lambda)||}{\lambda}>0.\]
Consequently, this leads to \(\operatorname{tord} (\gamma_ i,\gamma_j)=1\), which, by applying the isosceles property (Remark~\ref{Rem: non-archimedean property}) with the inequalities of the tangent orders in Equation~\eqref{eq:3contactlema}, implies that \(\operatorname{tord} (\tilde{\gamma}_ i,\tilde{\gamma}_j)=1.\) Since $\tilde{\gamma}_i$ and $\tilde{\gamma}_j$ were chosen arbitrarily, we conclude that    
\[\operatorname{Cont}(V(I;f), V(J;f))=1.\]
Analogously, we can show that
\[\operatorname{Cont}(V(I;f), V(K;f))=\operatorname{Cont}(V(J;f),V(K;f))=1.\]

\noindent\textbf{(ii):} 
Follows analogously to (i) of this Lemma, doing $k_i=k_j=1$.

\noindent\textbf{(iii):} Consider two real half-branches, $$\gamma_i \in V(I;f_{P_i}) \text{ and } \gamma_j \in V(J;f_{P_j}),\  i,j \in I_f^{(k>1)},\ i\leq j.$$ These real half-branches can be parametrized as Equation~\eqref{Eq:1contactlema}, where they satisfy the condition that 
$(u_I(\tau_i(0)),\rme^{t_I(\tau_i(0))})$ and  $(u_J(\tau_j(0)),\rme^{t_J(\tau_j(0))})$ are points belonging to the knots $I$ and $J$ in $\C \times \mathbb{S}^1$, respectively. Here, $k_i\geq k_j$, and $u_J(\tau_j(0))\neq 0$ if $i<j$.
\vspace{0.2cm}

If $\mathrm{proj}_2(I) \cap \mathrm{proj}_2 (J)= \emptyset$, then the following limit exists and satisfies 
\begin{align*}
	\lim_{\lambda \to 0}\frac{ ||\gamma_i(\lambda)- \gamma_j(\lambda)||}{\lambda} &\geq || \rme^{\rmi t_I(\tau_i(0))}- \rme^{\rmi t_J(\tau_j(0))}||>0.
\end{align*}
This leads to \(\operatorname{tord} (\gamma_ i,\gamma_j)=1,\), which, by applying the isosceles property, implies that the corresponding real half-branches $\tilde{\gamma}_i$ and $\tilde{\gamma}_j$ satisfy  \(\operatorname{tord} (\tilde{\gamma}_ i,\tilde{\gamma}_j)=1.\) 
Therefore,  \(\operatorname{Cont}(V(I;f), V(J;f))=1.\)
\vspace{0.2cm}

If $\mathrm{proj}_2(I) \cap \mathrm{proj}_2 (J)\neq \emptyset$, we can find two real half-branches:
$$\gamma_i(\lambda)= (\lambda^{k_i} u_I(t_*), \lambda\rme^{\rmi t_*}) \text{ and }  \gamma_j(\lambda)= (\lambda^{k_j} u_J(t_*), \lambda\rme^{\rmi t_*})$$ 
in $V(I;f_{P_i})$ and $V(J;f_{P_j})$, respectively. Given that $k_i \geq k_j$, we can ensure that the following limit exists and satisfies 
\begin{align*}
	\lim_{\lambda \to 0}\frac{ ||\gamma_i(\lambda)- \gamma_j(\lambda)||}{\lambda^{k_j}} &= ||\lambda^{k_i-k_j}u_I(t_*)-u_J(t_*)||>0,
\end{align*}
This leads to \(\operatorname{tord} (\gamma_ i,\gamma_j)=k_j,\), which, by applying the isosceles property, implies 	
\(\operatorname{tord} (\tilde{\gamma}_ i,\tilde{\gamma}_j)>1.\) Thus, we conclude that \(\operatorname{Cont}(V(I;f), V(J;f))>1.\) Therefore, we find that   
\[\operatorname{Cont}(V(I;f),V(J;f))=1 \text{ if and only if } \mathrm{proj}_2(I) \cap \mathrm{proj}_2 (J)=\emptyset.
\]
To show that \(\operatorname{Cont}(V(I;f),V(J;f))\leq k_j\), consider  two real half-branches:
\[\tilde{\gamma}_i(\lambda)= (\lambda^{k_i} \tilde{u}_I(\tau_i(\lambda)) \lambda\rme^{\rmi \tilde{t}_I(\tau_i(\lambda))}) \text{ and } 	\tilde{\gamma}_j(\lambda)= (\lambda^{k_j} \tilde{u}_J(\tau_j(\lambda)) \lambda\rme^{\rmi \tilde{t}_J(\tau_j(\lambda))}), \]
in $V(I;f)$ and $V(J;f)$, respectively. Then, for a rational number $\kappa>0$ we obtain:
\begin{align*}
	\frac{||\tilde{\gamma}_i(\lambda)- \tilde{\gamma}_j(\lambda)||}{\lambda^{k_j+\kappa}} &\geq ||\lambda^{k_i-k_j-\kappa} \tilde{u}_I(\tau_i(\lambda))- \lambda^{-\kappa} \tilde{u}_J(\tau_j(\lambda))||.
\end{align*}
Since  $\tilde{u}_J(\tau_j(\lambda))$ converges to a non-zero complex number, we find 
\begin{align*}
	\lim_{\lambda \to 0}\frac{ ||\tilde{\gamma}_i(\lambda)- \tilde{\gamma}_j(\lambda)||}{\lambda^{k_j+\kappa}}= \infty,
\end{align*}
proving that \[1\leq \operatorname{Cont}(V(I;f), V(J;f))\leq k_j. \] 

Still, in the case where $\mathrm{proj}_2(I) \cap \mathrm{proj}_2 (J)\neq \emptyset$, if $f$ and $g$ satisfy Condition \eqref{Eq:propertyA}, then we can find two real half-branches:
\[\tilde{\gamma}_i(\lambda)= (\lambda^{k_i} \tilde{u}_I(\tau(\lambda)), \lambda\rme^{\rmi \tilde{t}(\tau(\lambda))}) \text{ and } 	\tilde{\gamma}_j(\lambda)= (\lambda^{k_j} \tilde{u}_J(\tau(\lambda)), \lambda\rme^{\rmi \tilde{t}(\tau(\lambda))}) \]
in $V(I;f)$ and $V(J;f)$, respectively.
Then, we have 
\begin{align}\label{Eq:4contaclemma}
	\frac{||\tilde{\gamma}_i(\lambda)- \tilde{\gamma}_j(\lambda)||}{\lambda^{k_j}} &=	
	||\lambda^{k_i-k_j}\tilde{u}_I(\tau(\lambda))- \tilde{u}_J(\tau(\lambda)||.
\end{align}
Since the right side of Equation \eqref{Eq:4contaclemma} converges to a non-zero real number, we obtain  
\(\operatorname{tord} (\tilde{\gamma}_ i,\tilde{\gamma}_j)=k_j.\) 
Therefore, \(\operatorname{Cont}(V(I;f), V(J;f))=k_j.\)

\noindent\textbf{(iv):} Follows analogously to (iii) of this proposition, using appropriate changes. 
\end{proof}

\subsection{A local bi-Lipschitz extension to the identity} \label{unionlema} 

When we have a bi-Lipschitz isotopy $\Phi$ from a set $X \subset \R^n$ to $\tilde{X} \subset \R^n$, it usually cannot be extended to the entire ambient space in a direct manner. For example, Kerner and Mendes' method in \cite[Subsection 3.3]{Kerner-Mendes} fails to extend bi-Lipschitz isotopy on a horn neighbourhood of the singular locus. However, when the set has some specific structure, such isotopy can be extended to the ambient by other means (one such structure is the carrousel decomposition of $(\C^2,0)$ with respect to a complex plane curve introduced in \cite{neumann-pichon}). To do a similar decomposition of $(\C^2,0)$ with respect to a link component of the zeros of a $\Gamma_{\rm inn}$-nice  IND mixed polynomial, we first locate where the tangent cone of such a component is (Proposition \ref{aaaa}). Then, we deform the bi-Lipschitz isotopy obtained from Kerner and Mendes' method to the identity, around a suitable horn  neighbourhood of such a tangent cone, culminating in Lemma \ref{Lemma:Union}.

\begin{proposition}\label{aaaa}
    Let $f$ be an IND mixed polynomial that is $\Gamma_{\mathrm{inn}}$-nice, with $\mathcal{P}_{\mathrm{inn}}(f)=\{P_1,\dots,P_N\}$, $N \geq 2$. For each $i \in \{1,\dots,N\}$, we have
    \begin{itemize}
        \item[(i)] $C(V(L_{P_i};f),0) \subset \{u=0\}$, if $k_i>1$;
        \item[(ii)] $C(V(L_{P_i};f),0) \subset \{v=0\}$, if $k_i<1$;
        \item[(iii)] $C(V(L_{P_i};f),0) =V(L_{P_i};f_{P_i})$, if $k_i=1$;
    \end{itemize}
\end{proposition}

\begin{proof}
   For $k_i>1$, each connected component of $V(L_{P_i};f)\setminus\{0\}$ can be parametrized as $(r^{k_i}u_{i}(r,\tau),r\rme^{\rmi t_{i}(r,\tau)})$ (see Remark \ref{geralzao}). Since $\|(r^{k_i}u_{i}(r,\tau),r\rme^{\rmi t_{i}(r,\tau)})\|=r+o(r)$ and $\frac{r^{k_i}u_{i}(r,\tau)}{r} \to 0$ when $r \to 0^+$, the tangent cone of such a component is in $\{u=0\}$, and (i) is proved. The proof of (ii) is analogous.

   To prove (iii), just notice that
   $$\lim_{r \to 0^+}\frac{(r^{k_i}u_{i}(r,\tau),r\rme^{\rmi t_{i}(r,\tau)})}{r}=(u_{i}(0,\tau),\rme^{\rmi t_{i}(0,\tau)}) \in L_{i}$$
   Since $k_i=1$, the set $V(L_{P_i};f_{P_i})$ is a cone with basis $L_i$, and thus the result follows.
   
\end{proof}

\begin{remark}
   By the previous proposition, if $k_i \ne 1$, then $C(V(L_{P_i};f))$ is contained in a cone whose link is a curve diffeomorphic to $\mathbb{S}^1$ (namely $\{u=0\} \cap \mathbb{S}^3$ or $\{v=0\} \cap \mathbb{S}^3$). On the other hand, if $k_i=1$, then $C(V(L_{P_i};f))$ is a cone over several curves $L_{P_i} \subset \mathbb{S}^3$ that are diffeomorphic to $\mathbb{S}^1$
\end{remark}

\begin{lemma}\label{Lemma:Union}
    Let $\mathcal{F}=\{f+\varepsilon\theta\}_{\varepsilon \in [0,1]}$ be a family of mixed polynomials such that $f$ is IND and $\Gamma_{\mathrm{inn}}$-nice, $d(P_i; \theta)> d(P_i;f)$, and $\mathcal{P}_{\mathrm{inn}}(f)=\{P_1,\dots,P_N\}$, $N \geq 2$. Fix $i\in \{1,\dots,N\}$, and for any $\varepsilon \in [0,1]$, let $C_{\varepsilon}=C(V(L_{P_i};f+\varepsilon\theta),0)$. Denote $\mathcal{C}$ as $\{u=0\}$ if $k_i>1$, $\{v=0\}$ if $k_i<1$, and $\cup_{\varepsilon\in [0,1]}C_\varepsilon$ if $k_i=1$. Then, for each $\eta>0$ small enough, there is a ambient bi-Lipschitz isotopy $\Phi:(\C^2,0)\times[0,1]\to (\C^2,0)$ (see Definition \ref{amb-isotopy-equiv}) such that
    \begin{itemize}
        \item $\Phi_\varepsilon(p)=p$, for every $\varepsilon \in [0,1]$ and $p \in \C^2\setminus{\mathcal{H}_{1,\eta}(\mathcal{C})}$;
        \item $\Phi_\varepsilon(V(L_{P_i};f),0)=(V(L_{P_i};f+\varepsilon\theta),0)$, for every $\varepsilon \in [0,1]$.
    \end{itemize}
\end{lemma}

\begin{proof}
    
     By Proposition \ref{lemma:lipeomorphismlocal} and Corollary 2.2 of \cite{Valette} (see also Theorem 4.4.8 in \cite{Valette2}), there are $\beta_i \in \mathbb{Q}_{\ge 1}$ and $\eta_i>0$ small enough such that, for each $\varepsilon \in [0,1]$, there is a subanalytic bi-Lipschitz map $\Phi_{i,\varepsilon}:(\C^2\setminus \mathcal{H}_{\beta_i,\eta_i}(\Sigma_L(f_{P_i})),0) \to (\C^2\setminus \mathcal{H}_{\beta_i,\eta_i}(\Sigma_L(f_{P_i})),0)$ such that $\Phi_{i,\varepsilon}(V(L_{P_i};f_{P_i}),0)=(V(L_{P_i};f+\varepsilon\theta),0)$ and $\|\Phi_{i,\varepsilon}(p)\|=\|p\|$ for all $p$. Moreover, such maps depend smoothly on $\varepsilon$. Hence, for each $\varepsilon_0 \in [0,1]$ and for 
     each  small enough $\delta>0$, there is a uniform constant $K_{\varepsilon_0,\delta}>0$ such that, for each $\varepsilon \in [0,1]\cap(\varepsilon_0-\delta,\varepsilon_0+\delta)$, the maps $\tilde{\psi}_{\varepsilon_0 \to \varepsilon}:= \Phi_{i,\varepsilon} \circ \Phi_{i,\varepsilon_0}^{-1}$ are subanalytic, $K_{\varepsilon_0,\delta}-$bi-Lipschitz and vary smoothly on $\varepsilon$. Note that $\tilde{\psi}_{\varepsilon_0 \to \varepsilon_0}$ is the identity map and thus $K_{\varepsilon_0,\delta} \to 1$ when $\delta \to 0^+$.
     
     For a fixed $\varepsilon_0 \in [0,1]$, define $C$ as either $\{u=0\}$ if $k_i>1$, $\{v=0\}$ if $k_i<1$, and the union of $\{0\}$ and an arbitrary connected component of $C_{\varepsilon_0}\setminus\{0\}$ if $k_i=1$. Notice that $C$ is a cone over $0$ such that the curve $C\cap\mathbb{S}^3$ is diffeomorphic to $\mathbb{S}^1$, and has length $\ell$. Fix $p\in C\cap\mathbb{S}^3$ and suppose that $\gamma:[0,1] \to C\cap\mathbb{S}^3$ is a smooth parametrisation such that $\gamma(0)=\gamma(1)=p$ and $\gamma'(s)=\ell$, for every $s \in [0,1]$. Therefore, for every $t>0$, the curve $C\cap\mathbb{S}_t^3$ has length $\ell t$ and is parametrised by the curve $\gamma_t: [0,1] \to C\cap\mathbb{S}_t^3$ given as $\gamma_t(s)=t\gamma(s)$, for all $s\in[0,1]$. Moreover, $C \subseteq \mathcal{C}$, and for every $\eta>0$, there is $\delta>0$ small enough such that $K_{\varepsilon_0,\delta}<2$ and thus $\tilde{\psi}_{\varepsilon_0\to\varepsilon}(\mathcal{H}_{1,\eta}(C)) \subset \mathcal{H}_{1,2\eta}(C)$ for every $\varepsilon \in [0,1]\cap(\varepsilon_0-\delta,\varepsilon_0+\delta)$.   
     
     Now take $\eta=3\eta'>0$ small enough such that $\mathcal{H}_{1,3\eta'}(C\cap\mathbb{S}^3)$ is a tubular neighbourhood of $C\cap\mathbb{S}^3$, $(\mathcal{H}_{1,3\eta'}(\mathcal{C})\cap\mathcal{H}_{\beta_i,\eta_i}(\Sigma_L(f_{P_i})),0)=\{0\}$ and take $\delta$ as above. For each $r \in [0,3\eta']$, $s\in[0,1]$ and $t>0$ small enough, define $P_{s}(t)$ as the 3-dimensional hyperplane containing $\gamma_t(s)$ and orthogonal to $\gamma_t'(s)$, define $H_{r,s}(t)$ as $\mathcal{H}_{1,r}(C)\cap \mathbb{S}_t^3\cap P_{s}(t)$, and let $H_{r,s}=\{0\}\cup(\cup_{t>0}H_{r,s}(t))$, whose boundary is a 1-horn. Since $C\cap \mathbb{S}^3$ is smooth, $\partial(\mathcal{H}_{1,r}(C))$ is LNE and the maps $\tilde{\psi}_{\varepsilon_0\to\varepsilon}$ are $K_{\varepsilon_0,\delta}$-bi-Lipschitz, there is a uniform constant $K_1$ such that $\tilde{\psi}_{\varepsilon_0\to\varepsilon}(\partial(\mathcal{H}_{1,r}(C)))$ is $K_1$-LNE. Hence, the boundary of $\tilde{H}_{r,s}(t)= \tilde{\psi}_{\varepsilon_0\to\varepsilon}(\mathcal{H}_{1,r}(C))\cap \mathbb{S}_t^3 \cap P_{s}(t)$ is $K_1$-LNE for each such $r,s,t$. 
     
     Defining $\tilde{H}_{r,s}=\{0\}\cup(\cup_{t>0}\tilde{H}_{r,s}(t))$, by Theorem \ref{Edson-Rodrigo}, there is a uniform constant $K_2$ such that the boundary of $\tilde{H}_{r,s}$ is a 1-horn which is $K_2$-LNE. Then, by Theorem \ref{teo-gordurinha}, there is an ambient bi-Lipschitz map $\varphi_{s,\varepsilon}: (\tilde{H}_{2\eta',s},0) \to (H_{2\eta',s},0)$ such that $\varphi_{s,\varepsilon}(p)=p$, for all $p \in \tilde{H}_{\eta',s}$, $\varphi_{s,\varepsilon}(p) \in \partial(H_{2\eta',s})$, for all $p \in \partial(\tilde{H}_{2\eta',s})$ and $\|\varphi_{s,\varepsilon}(p)\|=\|p\|$, for all $p \in \tilde{H}_{2\eta,s}$. Moreover, since $\tilde{H}_{2\eta',s}$ depends smoothly on $s$ and $\varepsilon$, the maps $\varphi_{s,\varepsilon}$ depend continuously on $s$ and $\varepsilon$, with bounded derivatives. Therefore, the map $\varphi_{\varepsilon}: H_{2\eta',1}(C) \to H_{2\eta',1}(C)$ given as $\varphi_{\varepsilon}(p)=\varphi_{s,\varepsilon}(p)$, if $p \in H_{2\eta',s}(t)$, is a bi-Lipschitz map such that $\|\varphi_{\varepsilon}(p)\|=\|p\|$ for all $p \in \mathcal{H}_{1,2\eta'}(C)$, $\varphi_\varepsilon(V(L_{P_i};f+\varepsilon_0\theta))=V(L_{P_i};f+\varepsilon\theta)$ and $\varphi_{\varepsilon}(\partial(\mathcal{H}_{1,2\eta'}(C))) = \partial(\mathcal{H}_{1,2\eta'}(C))$. Notice also that $\varphi_{\varepsilon_0}$ is the identity map and the maps $\varphi_{\varepsilon}$ depend continuously on $\varepsilon$, with bounded derivatives with respect to $\varepsilon$.

     For each $p\in \partial(\mathcal{H}_{1,2\eta'}(C))$, $p \ne 0$, there are unique $s\in [0,1]$ and $t>0$ such that $p \in \partial(H_{2\eta', s}(t))$. Then, for each $x \in [2\eta',3\eta']$, there is a unique point $q=p_x$ in $\mathbb{S}_t^3$ such that $q$ is the intersection of the geodesic ray in $\mathbb{S}_t^3$, starting from $P_s(t) \cap C \cap \mathbb{S}_t^3$ and passing through $p$, with $\partial(H_{x,s}(t))$. Finally, let $\tau_\varepsilon(x)=\frac{(3\eta'-x)\varepsilon+(x-2\eta')\varepsilon_0}{\eta'}$ and define $\psi_{\varepsilon_0 \to \varepsilon}: (\C^2,0) \to (\C^2,0)$ as
     $$\psi_{\varepsilon_0 \to \varepsilon}(q)=\begin{cases} 
			q,  & \text{if } q \notin \mathcal{H}_{1,3\eta'}(C); \\ 
			(\varphi_{\tau_\varepsilon(x)}\circ\tilde{\psi}_{\varepsilon_0 \to \tau_{\varepsilon}(x)}(p))_x ,  &\text{if } q=p_x\in \partial(\mathcal{H}_{1,x}(C)), \, 2\eta' \le x \le 3\eta'; \\
            \varphi_{\varepsilon}\circ\tilde{\psi}_{\varepsilon_0 \to \varepsilon}(q), & \text{if } q \in \mathcal{H}_{1,2\eta'}(C).
		\end{cases}$$
		
	\begin{figure}[H]
		\centering
		\includegraphics[width=\textwidth]{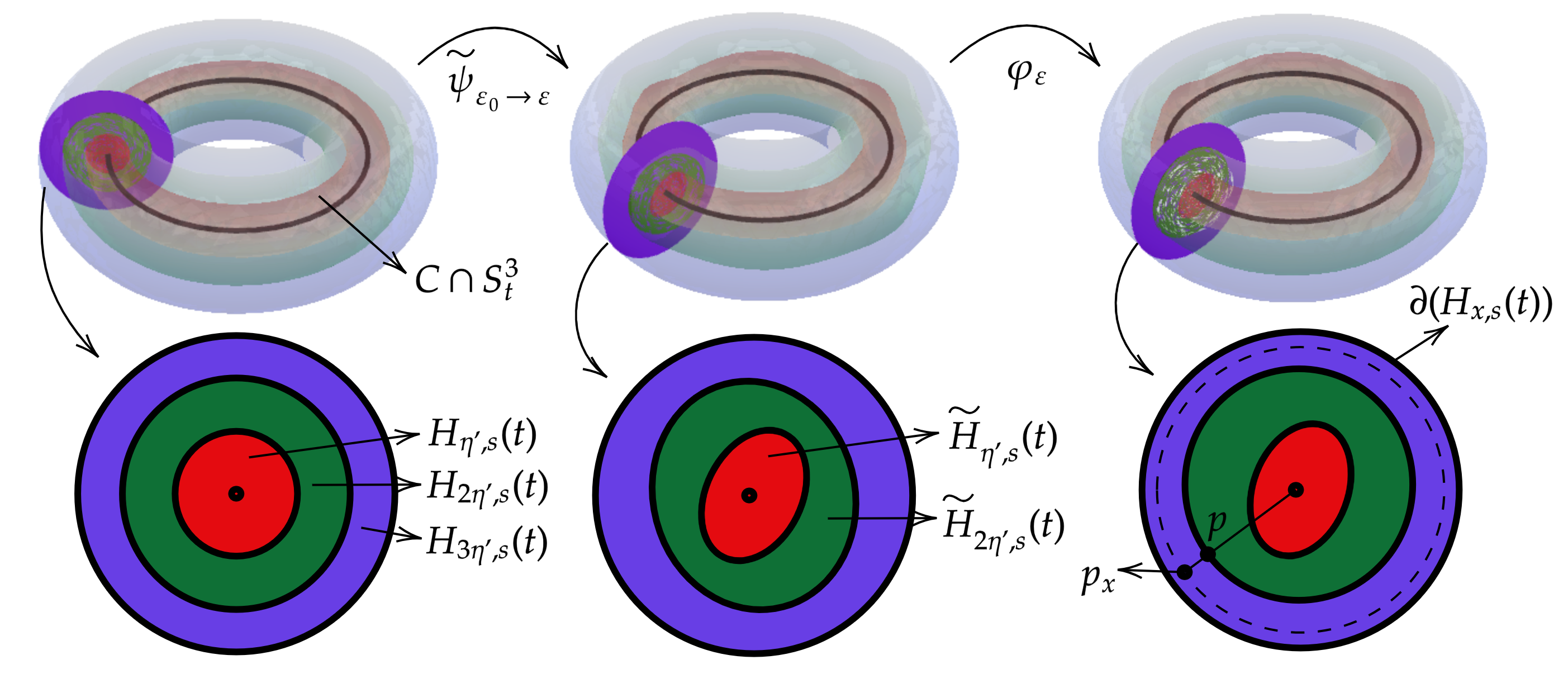}
		\caption{Proof of Lemma \ref{Lemma:Union}}
		\label{fig:gordurinha}
	\end{figure}
		
    The map $\psi_{\varepsilon_0 \to \varepsilon}$ is well defined, since $\tau_\varepsilon(3\eta')=\varepsilon_0$, $\tau_\varepsilon(2\eta')=\varepsilon$, and $(\varphi_{\varepsilon_0}\circ\tilde{\psi}_{\varepsilon_0\to\varepsilon_0}(p))_{3\eta'}=p_{3\eta'}=q$, $(\varphi_{\varepsilon}\circ \tilde{\psi}_{\varepsilon_0\to\varepsilon}(p))_{2\eta'}=\varphi_{\varepsilon}\circ \tilde{\psi}_{\varepsilon_0\to\varepsilon}(q)$. Notice also that, by construction, 
    $$\psi_{\varepsilon_0 \to \varepsilon}|_{\mathcal{H}_{1,2\eta'}(C)}: \mathcal{H}_{1,2\eta'}(C) \to \mathcal{H}_{1,2\eta'}(C),$$ $$\psi_{\varepsilon_0 \to \varepsilon}|_{\partial(\mathcal{H}_{1,x}(C))}: \partial(\mathcal{H}_{1,x}(C)) \to \partial(\mathcal{H}_{1,x}(C)) \, \, (2\eta'\le x \le 3\eta'), $$
    $$\psi_{\varepsilon_0 \to \varepsilon}|_{\C^2 \setminus\mathcal{H}_{1,3\eta'}(C)}: \C^2 \setminus\mathcal{H}_{1,3\eta'}(C) \to \C^2 \setminus\mathcal{H}_{1,3\eta'}(C)$$ are homeomorphisms, by the Homeomorphism Gluing Lemma. Thus $\psi_{\varepsilon_0 \to\varepsilon}$ is also a homeomorphism. As the map $\tau_\varepsilon(x)$ is linear in both $\varepsilon$ and $x$, the maps $\tilde{\psi}_{\varepsilon_0 \to \tau_{\varepsilon(x)}}$ and  $\varphi_{\tau_\varepsilon(x)}$ depend continuously on $\varepsilon$ and $x$, with bounded derivatives, and hence there is a uniform constant $K_3$ such that the map $\psi_{\varepsilon_0 \to \varepsilon}$ is $K_3$-bi-Lipschitz for each $\varepsilon \in [0,1]\cap(\varepsilon_0-\delta,\varepsilon_0+\delta)$.

    To finish the proof, notice that the union of the intervals $(\varepsilon_0-\delta,\varepsilon_0+\delta)$ is an open cover of the compact set $[0,1]$ and thus it can be covered by a finite number of them, say $I_n=(\varepsilon_n-\delta_n,\varepsilon_n+\delta_n)$, for $n=0,1,\dots,M$ and $0\le \varepsilon_0<\varepsilon_1<\dots<\varepsilon_M\le 1$. Suppose also that such intervals are chosen such that $I_n\cap I_{n+1} \ne \emptyset$ and let $s_n \in I_n\cap I_{n+1}$, for $n=1,\dots,M-1$. Finally, for each $\varepsilon\in [0,1]$, let $m$ be the least positive integer such that $\varepsilon \in I_{m}$ and let $$\Phi_\varepsilon:=\psi_{\varepsilon_{m}\to \varepsilon}\circ\psi_{\varepsilon_{m}\to s_{m}}^{-1}\circ\psi_{\varepsilon_{m-1}\to s_{m}}\circ\psi_{\varepsilon_{m-1}\to s_{m-1}}^{-1}\circ\dots\circ\psi_{\varepsilon_1\to s_2}\circ\psi_{\varepsilon_1\to s_1}^{-1}\circ\psi_{\varepsilon_0\to s_1}\circ\psi_{\varepsilon_0\to 0}^{-1}.$$
    By construction, such an ambient isotopy $\Phi$ satisfies all the conditions of Lemma when $k_i \ne 1$. When $k_i=1$, we construct ambient bi-Lipschitz mappings for each connected component of $\mathcal{C}\setminus\{0\}$ and then we glue those maps using the Lipschitz Gluing Lemma (Lemma \ref{Gluing Lipschitz Lemma}), since in the intersection such maps are equal to the identity.

\end{proof}

\section{Ambient Lipschitz geometry of semiradial mixed polynomials}\label{Sect:Bi-LipVtriviality}

\subsection{Bi-Lipschitz V-triviality}\label{sub:ambientbilipchitztriviality}

In this subsection, we aim to improve Theorem \ref{Lips-triviality} even in the case of $\Sigma_{L}(f)\ne\{0\}$. Since $\Sigma_{L}(f)$ is closely related to $k_f$ and the inner non-degeneracy condition, we first investigate the isolated case $\Sigma_{L}(f)=\{0\}$, and after this, the non-isolated case $\Sigma_{L}(f)\neq \{0\}$. To achieve such a task, we need the following definition.

 \begin{definition}\label{deftypes}
A semi-radial mixed polynomial $f$ of radial-type $(P;d)$ is categorised into three types based on the following properties:
\begin{itemize}
\item[] \textbf{Type-I}:  If \(k_f>1\) and \(f_P\) is \(u\)-convenient, or \(k_f<1\) and \(f_P\) is \(v\)-convenient.
\item[] \textbf{Type-II}: When the coefficient $k_f$ equals 1.
\item[] \textbf{Type-III}: If \(k_f>1\) and \(f_P\) is not \(u\)-convenient, or \(k_f<1\) and \(f_P\) is not \(v\)-convenient. 
\end{itemize}
\end{definition}
These three types collectively describe all semi-radial mixed polynomials.

Lemma~\ref{radial-obst} outlines the specific class of mixed polynomials for which Theorem~\ref{Lips-triviality} can be applied. Specifically, it identifies families of mixed polynomials $\{f+\varepsilon \theta\}_{\varepsilon \in I},$ for which Theorem~\ref{Lips-triviality} implies the ambient 
bi-Lipschitz $V$-triviality.
Our aim is to study a class of families of mixed polynomials that encompass those satisfying Theorem~\ref{Lips-triviality} and for which we could
 obtain bi-Lipschitz \( V \)-triviality. 

 \begin{proposition}\label{radial-trivial}
 	 Let \(\{f+\varepsilon \theta\}_{\varepsilon \in I}\) be a family such that $f$ is a Type-I or Type-II semi-radial mixed polynomial of radial-type $(P;d)$ and \(d(P;\theta) \geq  d\). Then, the family is ambient bi-Lipschitz $V$-trivial.

 \end{proposition} 
 
 \begin{proof}
We have that  $f=f_P+\tilde{f}$, where $f_P$ is an IND radial mixed polynomial of radial-type $(P;d)$ and $d(P;\tilde{f})>d$. By Lemma~\ref{radial-obst}, both Type-I and Type-II satisfy $\Sigma_{L}(f_P)= \{0\}$. 

Suppose first $d(P;\theta)=d$. 
We can see that Theorem \ref{Lips-triviality} applies to the following families:
\begin{itemize}
	\item[(a)] \(\{f_P + \varepsilon \tilde{f}\}_{\varepsilon \in I}\),
	\item[(b)] \(\{f_P + \varepsilon \theta_P\}_{\varepsilon \in I}\), where \(\theta_P\) is the face function of \(\theta\) with respect to \(P\),
	\item[(c)] \(\{f_P + \varepsilon \theta_P + \delta[\varepsilon (\theta - \theta_P) + \tilde{f}]\}_{\delta \in I}\), where \(\varepsilon\) is fixed and sufficiently small.
\end{itemize}
Thus, for any $\varepsilon \in [0, \varepsilon_0), \ \varepsilon_0>0$ small enough, there exist ambient bi-Lipschitz maps $\Psi$, $\Phi_\varepsilon$, and $\phi$ making the following diagram commutative:
\begin{equation*}\label{comutativecases}
	\begin{tikzcd}[row sep=large, column sep=large]
		(\C^2,V(f)) 
		\arrow[r, "\Psi^{-1} \circ \Phi_\varepsilon \circ \phi"] 
		& (\C^2,V(f+\varepsilon \theta)) \\
		(\C^2,V(f_P)) 
		\arrow[r, "\Phi_\varepsilon", swap] 
		\arrow[u, "\Psi", swap] 
		& (\C^2,V(f_P+\varepsilon \theta_P)). 
		\arrow[u, "\phi"]
	\end{tikzcd}
\end{equation*}
Therefore, the family of bi-Lipschitz maps \(\{(\Psi^{-1} \circ \Phi_\varepsilon \circ \phi)^{-1}\}_{\varepsilon \in [0. \varepsilon_0)} \) trivialises the zeros of the family $\{f+ \varepsilon \theta\}_{\varepsilon \in [0. \varepsilon_0)} $, that is,  $(\Psi^{-1} \circ \Phi_\varepsilon \circ \phi)^{-1} (V(f+\varepsilon \theta))= V(f),$ and the conclusion follows in this case.  

For the case $d(P;\theta)>d$, by Theorem \ref{Lips-triviality}, the ambient bi-Lipschitz map $\Phi_\varepsilon$ exists for any $\varepsilon\in I$, which implies the desired conclusion.
\end{proof}

\begin{example}\label{firstsemiradiaexamples1}
Consider the holomorphic polynomial $f(u,v)=uv+u^4+v^4$. It is semi-radial of radial-type $(P;d)=(1,1;1)$. By Proposition~\ref{radial-trivial}, the family $\{f+\varepsilon \theta\}_{\varepsilon \in I}$ is ambient bi-Lipschitz $V$-trivial if $d(P;\theta) > 1$.  Notice that the family $\{uv+ u^4+ v^4 +\varepsilon(u^3+v^3)\}_{\varepsilon \in I}$ is ambient bi-Lipschitz $V$-trivial. From this family, we deduce that the Newton boundary is not invariant under ambient bi-Lipschitz $V$-equivalence, as $uv+ u^4+ v^4$ and $uv+ u^3+ v^3$ are ambient bi-Lipschitz $V$-equivalent but have different Newton boundaries.
\end{example}


\begin{proposition}\label{Prop:biliptrivialitytypeIII}
	  Let \(\{f+\varepsilon \theta\}_{\varepsilon \in I}\) be a family such that $f$ is a Type-III semi-radial mixed polynomial of radial-type $(P;d)$, \(d(P;\theta) \geq  d\). Then, the family is ambient bi-Lipschitz $V$-trivial.
 
\end{proposition}
\begin{proof}

    We have that $f=f_P+\tilde f$, where $f_P$ is an IND radial mixed polynomial of radial-type $(P;d)$ and $d(P;\tilde f)>d$. Therefore, in the same way as the proof of Proposition \ref{radial-trivial}, we can suppose (changing $f$ by $f_P$) that $f$ is a Type-III and radial mixed polynomial of radial-type $(P;d)=(p_{1}, p_{2};d)$ with $k_f=\frac{p_{1}}{p_{2}}$.  We assume that \(k_f>1\) and that \(f\) is not \(u\)-convenient. The case where $k_f<1$ and \(f\) is not \(v\)-convenient follows similarly by swapping the roles of \(u\) and \(v\).
	
To find the required bi-Lipschitz map, we split $V(f)$ into pieces associated with \(\Gamma_{\mathrm{inn}}(f)\). We construct triples for each piece using Proposition~\ref{lemma:lipeomorphismlocal} and Lemma \ref{lemma:overdeformation}, then gluing the corresponding bi-Lipschitz map. Since $f$ is Type-III, the diagram \(\Gamma_{\mathrm{inn}}(f)\) satisfies \(\mathcal{P}_{\mathrm{inn}}(f)=\{P_1=P,P_2=(1,1)\}\). Hence the links $L_{f}$ and $L_{f+\varepsilon\theta}$, $\varepsilon \in I$,  are isotopic to the nested link $\mathbf{L}([L_1],L_v)$.
	
	Suppose first that $d(P;\theta)=d$. By Lemma \ref{Lemma:Union}, there are two ambient bi-Lipschitz maps \(\Phi_{L_1,\varepsilon}, \Phi_{L_v,\varepsilon}: (\C^2,0) \to (\C^2,0) \) such that $\Phi_{L_1,\varepsilon}(V(L_1;f))=V(L_1;f+\varepsilon\theta)$ and $\Phi_{L_v,\varepsilon}(V(L_v;f))=V(L_v;f+\varepsilon\theta)$. The map $\Phi_{L_1,\varepsilon}$ is obtained directly from this lemma, and the map $\Phi_{L_v,\varepsilon}$ is also obtained from this lemma, but not so directly, as we may consider the maps $\tilde{\psi}_{\varepsilon_0 \to \varepsilon}$ in the proof of Lemma \ref{Lemma:Union} as $\Phi_{f_{P_2},f+\varepsilon\theta}\circ\Phi_{f_{P_2},f+\varepsilon_0\theta }^{-1}$ instead of $\Phi_{i,\varepsilon} \circ \Phi_{i,\varepsilon_0}^{-1}$, where $\Phi_{f_{P_2},f+\varepsilon\theta}$ and $\Phi_{f_{P_2},f+\varepsilon_0\theta}$ are the bi-Lipschitz maps in Proposition~\ref{lemma:lipeomorphismlocal} such that $\Phi_{f_{P_2},f+\varepsilon_0\theta}(V(L_v;f_{P_2}))=V(L_v;f+\varepsilon_0\theta)$ and $\Phi_{f_{P_2},f+\varepsilon\theta}(V(L_v;f_{P_2}))=V(L_v;f+\varepsilon\theta)$ (this is because $L_{\underline{2}}=L_v$). By Proposition~\ref{lemma:lipeomorphismlocal}, such maps $\tilde{\psi}$ have the same properties as the previous $\tilde{\psi}$, and thus we can obtain the desired ambient bi-Lipschitz map $\Phi_{L_v,\varepsilon}$.
    
\begin{equation*}\label{deltafamilies}
	\begin{tikzcd}[row sep=2.5em, column sep=4em]
		& (V({L_v}; f_{P_2}),0) \arrow[ld, "\Phi_{f_{P_2}, f}"'] \arrow[rd, "\Phi_{f_{P_2},f+\varepsilon\theta}"] & \\
		(V({L_v}; f),0) \arrow[rr, "\Phi_{L_v,\varepsilon}"] & & (V({L_v}; f+\varepsilon\theta),0). 
	\end{tikzcd}
\end{equation*}    
     For every $\eta>0$ small enough, we have $$\Phi_{L_1,\varepsilon}|_{\C^2 \setminus \mathcal{H}_{1,\eta}(V(L_1;f))}=id_{\C^2 \setminus \mathcal{H}_{1,\eta}(V(L_1;f))} \text{ and } \Phi_{L_v,\varepsilon}|_{\C^2 \setminus \mathcal{H}_{1,\eta}(V(L_v;f))}=id_{\C^2 \setminus \mathcal{H}_{1,\eta}(V(L_v;f))}.$$
    In particular, $\Phi_{L_1,\varepsilon}(p)=p=\Phi_{L_v,\varepsilon}(p)$, for all $p \in \C^2\setminus(\mathcal{H}_{1,\eta}(V(L_1;f)\cup V(L_v;f)))$.  We also have by Proposition~\ref{nontangencycriteria} that, for any component $J$ of $L_1$,
	\[\operatorname{Cont}( V(J;f),V({L_v};f))=\operatorname{Cont}( V(J;f+\varepsilon\theta),V(L_v;f+\varepsilon\theta))=1.\]
	Therefore, we can choose $\eta$ small enough such that $\mathcal{H}_{1,\eta}(V(J;f)) \cap \mathcal{H}_{1,\eta}(V({L_v};f))=\{0\}$, for all such components $J$. This implies that the map $\Phi: (\C^2,0) \to (\C^2,0)$ given as $\Phi(p)=\Phi_{L_1,\varepsilon}(p)$, if $p \notin \mathcal{H}_{1,\eta}(V(L_v;f))$ and $\Phi(p)=\Phi_{L_v,\varepsilon}(p)$, if $p \notin \mathcal{H}_{1,\eta}(V(L_1;f))$, is a bi-Lipschitz map by the Lipschitz Gluing Lemma (Lemma \ref{Gluing Lipschitz Lemma}). The result follows in this case.
	
	Suppose now that $d(P;\theta)>d$. In this case, by Lemma~\ref{Lemma:Union}, the ambient bi-Lipschitz map $\Phi_{L_1,\varepsilon}$ can be defined for all $\varepsilon \in \R$. However, it is not direct that the same happens with $\Phi_{L_v,\varepsilon}$. The bi-Lipschitz map $\Phi_{f_{P_2},f}$ can be defined in the same way as the previous case, so it remains to prove that the bi-Lipschitz map $\Phi_{f_{P_2},f+\varepsilon \theta}$ can be defined for all $\varepsilon \in \R$, to assure that $\Phi_{L_v,\varepsilon}$ is defined for all $\varepsilon \in \R$.
	
 	The bi-Lipschitz map $\Phi_{f_{P_2},f+\varepsilon \theta}$ is defined as follows: if $(f+\varepsilon \theta)_{P_2}=f_{P_2}$, then Proposition~\ref{lemma:lipeomorphismlocal} provides it directly, as $d(P_2;f+\varepsilon \theta-f_{P_2})>d(P_2;f)$; if $(f+\varepsilon \theta)_{P_2}-f_{P_2}\not\equiv 0$, then by Lemma~ \ref{lemma:overdeformation} we construct \(\Phi_{f_{P_2},(f+\varepsilon \theta)_{P_2}}\) and by Proposition~\ref{lemma:lipeomorphismlocal} 
	we obtain \(\Phi_{(f+\varepsilon \theta)_{P_2},f+\varepsilon \theta}\), and these bi-Lipschitz maps exist for all $\varepsilon \in \R$.

\begin{equation*}\label{deltafamilies1}
		\begin{tikzcd}[row sep=2.5em, column sep=4em]
			& (V({L_v}; f_{P_2}),0) \arrow[ld, "\Phi_{f_{P_2}, (f+\varepsilon \theta)_{P_2}}"'] \arrow[rd, "\Phi_{f_{P_2},f+\varepsilon \theta}"] & \\
			(V({L_v}; (f+\varepsilon \theta)_{P_2}),0) \arrow[rr, "\Phi_{(f+{\varepsilon}\theta)_{P_2},f+\varepsilon \theta}"] & & (V({L_v}; f+\varepsilon \theta),0).
		\end{tikzcd}
	\end{equation*}
    
	Thus, we conclude as previously, using the family of bi-Lipschitz maps \(\{\Phi_{L_1,\varepsilon}, \Phi_{L_v,\varepsilon}\}_{\varepsilon \in I}\) and using the Lipschitz Gluing Lemma.
\end{proof}

As a direct consequence of Propositions \ref{radial-trivial} and \ref{Prop:biliptrivialitytypeIII}, we obtain the following theorem.

\begin{theorem}\label{MainThm:1}
		Let \(\{f + \varepsilon \theta\}_{\varepsilon \in I}\) be a family of mixed polynomials satisfying:
		\begin{enumerate}
			\item[(i)] \(f\) is semi-radial of radial-type \((P; d)\),
			\item[(ii)] \(d(P; \theta) \geq d\).
		\end{enumerate}
		Then, the family is ambient bi-Lipschitz $V$-trivial. 
	\end{theorem}

\begin{corolario}\label{cor:bilipVequivsemi-radialprincipalpart}
		Let $f$ be a semi-radial mixed polynomial of radial-type $(P;d)$. Then, $f$ is ambient bi-Lipschitz $V$-equivalent to $f_P$. 
\end{corolario}
\begin{proof}
Since \(f\) is semi-radial of radial type \((P; d)\), it can be included in the family \(\{f_P + \varepsilon \tilde{f}\}_{\varepsilon \in I}\) with \(\varepsilon = 1\). By definition, \(d(P; \tilde{f}) > d\). By Theorem \ref{MainThm:1}, this family is ambient bi-Lipschitz \(V\)-trivial, and the trivialisation holds for all \(\varepsilon \in [0,1]\). Consequently, \(f\) is ambient bi-Lipschitz \(V\)-equivalent to \(f_P\).
\end{proof}

\begin{corolario}\label{Lipschitz-homogeneo}
    Let $f$ and $g$ be semi-radial mixed polynomials of radial-type $(P;d(P;f))$ and $(Q;d(Q;g))$, respectively. If $f$ and $g$ are bi-Lipschitz $V$-equivalent via a bi-Lipschitz map $
\tilde{F}:(V(f),0)\to (V(g),0)$, then there is a commutative diagram of bi-Lipschitz map as follows:
\begin{equation*}\label{comutativediagram}
	\begin{tikzcd}
		(V(f),0) \arrow{r}{\tilde{F}} \arrow[swap]{d}{} & (V(g),0) \arrow{d}{} \\
		(V(f_{P}),0) \arrow{r}{F} & (V(g_{Q}),0).
	\end{tikzcd}
\end{equation*}
\end{corolario}

\subsection{Tangent cones of mixed polynomials}\label{Subsect:necessarybi_lipVequivsemi-radial}

We now relate the tangent cone of the semi-radial mixed polynomials to the zero set of their radial parts. Such cones are described up to ambient bi-Lipschitz equivalence using the previous theorem and Sampaio's Theorem of ambient equivalence between tangent cones in \cite{Sampaio:2016}. 

\begin{lemma}\label{tangconetypeII}
 Let $f$ be a Type-II semi-radial mixed polynomial of radial-type $(P;d)$. Then $C(V(f),0)$ is ambient bi-Lipschitz equivalent to $V(f_P)$. 
\end{lemma}
\begin{proof}
Since $f$ is a Type-II semi-radial mixed polynomial, then $f=f_{P_1}+\tilde{f}$ with $\mathcal{P}_{\mathrm{inn}}(f)=\{P_1=(1,1)\}$ and $f_{P_1}$ being a radial mixed polynomial of radial-type $(1,1;d).$
\vspace{0.1cm}

Let $(u_*,v_*) \in V({f_{P_1}})$. Therefore, 
$(\lambda u_*, \lambda v_*)\in V(f_{P_1})$ is a real half-branch with tangent vector at zero $(u_*,v_*)$, thus  $V(f_{P_1})\subset  C(V(f_{P_1}),0)$.  Reciprocally, every real half-branch $\gamma(\lambda)$ of $V(f_{P_1})$, centered at the origin, can be parametrizated by \[(\lambda u(t(\lambda)), \lambda \rme^{\rmi t(\lambda)}) \text{ or } (\lambda \rme^{\rmi t(\lambda)}, \lambda v(t(\lambda)).\] Suppose that it is parametrizated by $\gamma(\lambda)=(\lambda u(t(\lambda)), \lambda \rme^{\rmi t(\lambda)}).$ Taking $t_*=\lim_{\lambda \to 0} t(\lambda)$. We get \[ v = \lim_{\lambda\to 0} \frac{\gamma(\lambda)}{\lambda}=\frac{(\lambda u(t(\lambda)), \lambda \rme^{\rmi t(\lambda)})}{\lambda}= (u(t_*),\rme^{\rmi t_*}).\]
In this case, $v \in V((f_{P_i})_i)$ and $C(V(f_{P_1}),0) \subset V(f_{P_1})$. Therefore, $C(V(f_{P_1}),0)=V(f_{P_1})$. Moreover, $V(f_{P_1})$ is the cone over $L_{f_{P_1}}$. By Proposition~\ref{radial-trivial}, $f$ is bi-Lipschitz $V$-equivalent to $f_{P_1}$; thus, by Theorem \ref{cone-sampas}, $C(V(f),0)$ is bi-Lipschitz equivalent to $C(V(f_{P_1}),0)=V(f_{P_1})$.    
\end{proof}
Consider the maps $T_1$ and $T_2$ defined by: 
\[T_1: \mathbb{C} \times \mathbb{C} \to \mathbb{C} \times \{0\},\ (z_1, z_2) \mapsto (z_1, 0) \text{ and } T_2: \mathbb{C} \times \mathbb{C} \to \{0\} \times \mathbb{C}, \ (z_1, z_2) \mapsto (0, z_2).\]

\begin{lemma}\label{tangconetypeI}
 Let $f$ be a Type-I semi-radial mixed polynomial of radial-type $(P;d)
 $.
\begin{enumerate}
\item[(i)] If $k_f>1$ and $f_P$ is $u$-convenient, then $C(V(f),0)$ is ambient bi-Lipschitz equivalent to a subset of $\{u=0\}$. In particular, if $f_P$ is $u$- (or $\bar{u}$-) semiholomorphic or it is not $v$-convenient, then $C(V(f),0)$ is ambient bi-Lipschitz equivalent to $\{u=0\}$.    
\item[(ii)]  If $k_f<1$ and $f_P$ is $v$-convenient, then $C(V(f),0)$ is ambient bi-Lipschitz equivalent to a subset of $\{v=0\}$. In particular, if $f_P$ is $v$- (or $\bar{v}$-) semiholomorphic or it is not $u$-convenient, then $C(V(f),0)$ is ambient bi-Lipschitz equivalent to $\{v=0\}$. 
\end{enumerate}
\end{lemma}
\begin{proof}
\noindent \textbf{(i):} In this case, $P=P_1 \in \mathcal{P}_{\mathrm{inn}}(f)$ and $k_f=k_1$. Consider a real half-branch $\gamma(\lambda)=(\lambda^{k_1} u_*, \lambda \rme^{\rmi t_*}) \in V(f_{P_1})$, centred at the origin. Since $k_1>1$ we have that $(0,\rme^{\rmi t_*}) \in C(V(f_{P_1}),0)$. Therefore, the set $$T_2(V(f_{P_1}))=\{(0,\lambda \rme^{\rmi t}) \mid  (\lambda^{k_1}u(t),\lambda \rme^{\rmi t}) \in V(f_{P_1}), \ \lambda>0, \ t\in [0,2\pi],\ u(t) \in \C\},$$
is contained in $C(V(f_{P_1}),0)$. Moreover, by Equation~\eqref{nondeguvertex} we have $C(V(f_{P_1}),0)\cap \{v=0\}=\{0\}$ and thus the real half-branches of the form $(\lambda u_*, 0)$ are not contained in $V(f_{P_1})$. Reciprocally, 
every real half-branch of $V(f_{P_1})$ centred at the origin can be parametrized by $\gamma(\lambda)=(\lambda^{k_1}u(t(\lambda)), \lambda \rme^{\rmi t(\lambda)})$. Taking $t_*=\lim_{\lambda \to 0} t(\lambda)$. We get \[ v = \lim_{\lambda\to 0} \frac{\gamma(\lambda)}{\lambda}=\frac{(\lambda^{k_1}u(t(\lambda)), \lambda \rme^{\rmi t(\lambda)})}{\lambda}= (0,\rme^{\rmi t_*}).\]
	Thus, $C(V(f_{P_1}),0)=T_2(V(f_{P_1}))$. Therefore, by Corollary~\ref{cor:bilipVequivsemi-radialprincipalpart} and Theorem \ref{cone-sampas}, $C(V(f),0)$ is bi-Lipschitz equivalent to $C(V(f_{P_1}))$. Notice that $T_2(V(f_{P_1}))$ can be calculated from 
$V((f_{P_1})_1)$ as 
\[(\lambda^{k_1}u(t),\lambda \rme^{\rmi t}) \in V(f_{P_1}) \text{ if and only if } (u(t), \rme^{\rmi t}) \in V((f_{P_1})_1).\] In particular, if $f_{P_1}$ is \(u\)- (or \(\bar{u}\)-) semiholomorphic, then for any $t \in [0,2\pi]$ there exists $u(t)\in \C$ such that $(u(t),\rme^{\rmi t})\in V((f_{P_1})_1)$. Thus, $C(V(f_{P_1}))=\{u=0\}$. If $f_{P_1}$ is not $v$-convenient, then $(0,\rme^{\rmi t}) \in V(f_{P_1})$ for all $t\in [0,2\pi]$; in this case, we also have $C(V(f_{P_1}))=\{u=0\}$.
\vspace{0.3cm}

\noindent \textbf{(ii):} In this case, \(P=P_1 \in \mathcal{P}_{\mathrm{inn}}(f)\), $k_f=k_1$, and 
$$T_1(V(f_{P_1}))=\{(\lambda \rme^{\rmi t},0)\mid (\lambda \rme^{\rmi t}, \lambda^{\frac{1}{k_1}} v(t)) \in V(f_{P_1}), \lambda>0, \ t\in [0,2\pi], v(t) \in \C \}.$$
 The set $T_1(V(f_{P_1}))$ is calculated from $V((f_{P_1})_{\underline{1}})$. The result follows the same arguments as in the previous case.
\end{proof}
\begin{lemma}\label{tangconetypeIII}
 Let $f$ be a Type-III semi-radial mixed polynomial of radial-type $(P;d)$. Then,
\begin{enumerate}
\item[(i)] If $k_f>1$ and $f_P$ is not $u$-convenient, then $C(V(f),0)$ is ambient bi-Lipschitz equivalent to a union of $\{v=0\}$ and a subset of $\{u=0\}$. In particular, if $f_P$ is \(u\)- (or \(\bar{u}\)-) semiholomorphic or it is not $v$-convenient, then $C(V(f),0)$ is ambient bi-Lipschitz equivalent to $\{uv=0\}$.    
\item[(ii)]  If $k_f<1$ and $f_P$ is not $v$-convenient, then $C(V(f),0)$ is ambient bi-Lipschitz equivalent to a union of $\{u=0\}$ and a subset of $\{v=0\}$. In particular, if $f_P$ is \(v\)- (or \(\bar{v}\)-) semiholomorphic or if it is not $u$-convenient, then $C(V(f),0)$ is ambient bi-Lipschitz equivalent to $\{uv=0\}$. 
\end{enumerate}
\end{lemma}
\begin{proof}
\noindent \textbf{(i):} The proof follows in the same way as Lemma~\ref{tangconetypeI}, the unique difference is that $f_P$ is not $u$-convenient, thus $\{v=0\} \subset V(f_P)$, which implies $\{v=0\} \subset  C(V(f_P),0)$. On the other hand, by Remark~\ref{linkofTypeIII}, $T_2(V(f_P)) \subset \{u=0\}$.

  Therefore, $C(V(f_P),0)$ is bi-Lipschitz equivalent to \[\{v=0\} \cup T_2(V(f_P)).\] By Corollary~\ref{cor:bilipVequivsemi-radialprincipalpart} and Theorem \ref{cone-sampas}, we get the result. 
\vspace{0.2cm}

\noindent \textbf{(ii):} Follows analogously to Lemma~\ref{tangconetypeIII}~(i)  by using Lemma~\ref{tangconetypeI} (ii).
\end{proof}
    
\section{
Ambient Lipschitz geometry of 1-braid closures and non-tangent Hopf-links}\label{Sect:1-braids}

In this section, we aim to obtain ambient bi-Lipschitz triviality beyond semi-radial mixed polynomials. We begin our investigation with $\Gamma_{\mathrm{inn}}$-nice and IND mixed polynomials, whose link is "quite simple" in both metric and topological senses. This motivates the definitions of metric 1-braid closure and non-tangent Hopf-link. Crucially, the characteristics of such links are sufficient for their classification and rigidity in the ambient Lipschitz category (see Theorem~\ref{prop:1-braid-Lipschitz-trivial}), and this classification extends to several mixed polynomials whose link is isotopic to either a trivial knot (Proposition~\ref{trivial-knot}) or a topological Hopf-link (Proposition~\ref{hopf-link}).  
  
\begin{definition}\label{1-braidclosure}
	Let $f$ be a mixed polynomial such that \( V(f) \subset \C^2 \) is the germ of an isolated surface singularity. The link \(L_f\) associated with \( V(f) \) is called a \textbf{metric 1-braid closure} if:
	\begin{enumerate}
		\item[(i)] \( L_f \) is the closure of a single strand braid with braid axis \( L_u \) or \( L_v \) (see Remark~\ref{braidclosure}).
		\item[(ii)] The tangent cone of \( V(f) \) at the origin satisfies 
			\begin{equation*}
			C(V(f),0)= \begin{cases} 
				\{u = 0\},  &\text{if $L_v$ is the braid axis of $L_f$,} \\ 
					\{v = 0\},  &\text{if $L_u$ is the braid axis of $L_f$.} 
			\end{cases}
		\end{equation*} 
	\end{enumerate}
\end{definition}
By definition, a metric 1-braid closure with braid axis $L_u$ is isotopic to those with braid axis $L_v$ and their surfaces have tangent cones at the origin that are ambient bi-Lipschitz equivalent.
\begin{definition}\label{Def:Lip-transverse-Hopf-link}
	Let $f$ be a mixed polynomial such that \( V(f) \subset \C^2 \) is the germ of an isolated surface singularity. The link \(L_f\) associated with \( V(f) \) is called a \textbf{non-tangent Hopf-link} if:
\begin{enumerate}
	\item[(i)]  $V(f)$ admits a decomposition $V_1 \cup V_2$ satisfying that their links \(L_i,\ i=1,2, \) are metric 1-braid closures; with $L_u$ and $L_v$ as their braid axis, respectively;
	\item[(ii)] \(\operatorname{Cont}(V_1,V_2)=1\).  
\end{enumerate}
\end{definition}
\begin{example}
The link of $f(u,v)=uv+u^3+v^3$ is a non-tangent Hopf-link.  The Newton boundary is formed by two compact 1-faces associated with the weight vectors $P_1=(2,1)$ and $P_2=(1,2)$. By \cite[Theorem 1.2]{AraujoBodeSanchez}, $L_f$ is isotopic to $\mathbf{L}([L_1], [L_{\underline{2}}]')$, where $L_1$ is the solution in $u$ of $(f_{P_1})_1$ and  $L_{\underline{2}}$ is the solution in $v$ of $(f_{P_2})_{\underline{2}}$. On the other hand, we find that \(\Gamma_{\mathrm{inn}}(f) = \{Q_1 = (1,1)\}\). Consequently, by Theorem~\ref{thm:link-triviality1}, \(L_f\) is isotopic to \(L_{f_{Q_1}}\), which corresponds to the link of the holomorphic polynomial \(uv\).
\end{example}

The next proposition relates the metric 1-braid closure with its tangent cone up to ambient bi-Lipschitz equivalence. This will be useful to obtain the main result of this section (Theorem~\ref{prop:1-braid-Lipschitz-trivial}).

\begin{proposition}\label{proposition:1-braid-LNE}
	Let $f$ be an IND mixed polynomial that is $\Gamma_{\mathrm{inn}}$-nice such that $V(f)$ has a component $V_i$ that is a metric 1-braid closure. Then, $(V_i,0)$ is LNE and is ambient bi-Lipschitz equivalent to its tangent cone $C(V_i,0)$.
\end{proposition}
\begin{proof}
	Since $V_i$ is a 1-braid closure, by Remark \ref{geralzao} there is a smooth family $\{u_r\}_{r>0}$ of $2\pi$-periodic functions $u_r: \R \to \C$ and a smooth, $2\pi$-periodic function $u: \R\to \C$ satisfying $\lim_{r\to 0^+}u_r(t)=u(t)$, such that $V_i\setminus\{0\}$ is parametrized by $(r^{k_i}u_r(t),r\rme^{\rmi t})$, for some $k_i\ge 1$ and each $t \in \R$, when the tangent cone of $V_i$ at $0$ is $\{u=0\}$; or by $(R\rme^{\rmi \varphi},R^{1/k_i}v_r(\varphi))$, for some $k_i\le 1$ and each $\varphi \in \R$, when the tangent cone of $V_i$ at $0$ is $\{v=0\}$. In this proof, we will consider the first parametrization, as the proof for the second parametrization is analogous.
    
	The link of $V_i$ is homeomorphic to $\mathbb{S}^1$, and since its tangent cone at $0$ is $\{u=0\}$, it has real dimension $2$. Then, by Remark \ref{Rem: NE HT condition}, $V_i$ is inner bi-Lipschitz equivalent to the standard $1$-horn. Let $T>0$ be large enough and $r_0>0$ small enough such that $T>{\rm max}\{|u'(t)| \, : \, t\in[0,2\pi]\}$ and $|u_r(t)|<T$, for every $t \in [0,2\pi]$ and $r<r_0$. Notice that, since $\{|u_r(t)|\}_{r<r_0; \, t \in \R}$ is bounded and the link of $V_i$ is a metric 1-braid closure, the projection of $(V_i,0)$ in its tangent cone $\{u=0\}$ is a bi-Lipschitz map. More precisely, $V_i$ can be seen as the graph of a bi-Lipschitz map $\phi: (\C,0)\to (\C,0)$ given as $\phi(0)=0$, $\phi(re^{it})= r^{k_i}u_r(t)$, for $r>0$, $t\in \R$, and $(V_i,0)$ is the germ, at 0, of $\{(\phi(\bm{z}),z) \, ; \, x\in \C\}$. 
    
    Now, for each $\varepsilon \in [0,1]$, define $\Phi_{i,\varepsilon}: (\C^2,0) \to (\C^2,0)$ by $\Phi_{i,\varepsilon}(z_1,z_2)=(z_1-\varepsilon\phi(z_2),z_2)$, for all $(z_1,z_2) \in (\C^2,0)$. It is easy to see that each $\Phi_{i,\varepsilon}$ is a bijection. 

    {\bf Claim.} There is a uniform constant $K_0>1$ such that each $\Phi_{i,\varepsilon}$ is $K_0$-bi-Lipschitz. Moreover, such maps are subanalytic and vary smoothly with $\varepsilon$.
    
    Since $\phi$ is subanalytic (see Remark \ref{geralzao}) and $\varepsilon$ is a linear parameter on $\Phi_{i,\varepsilon}$, it remains to prove the bi-Lipschitz property. To prove this, notice that if $K>1$ is such that $\phi$ is $K$-bi-Lipschitz, then for each $p=(u_1,v_1),q=(u_2,v_2)\in (\C^2,0)$, if $u=u_1-u_2$, $v=v_1-v_2$, and $w=\phi(v_1)-\phi(v_2)$, then $\frac{1}{K}|v|\le |w|\le K|v|$. We also have $\|\Phi_{i,\varepsilon}(p)-\Phi_{i,\varepsilon}(q)\|=\sqrt{|u-\varepsilon w|^2+|v|^2}$. Therefore, since $|u-\varepsilon w|^2 \le (|u|+|v|)^2=|u|^2+|v|^2+2|u||v|$ and $2|u||v| \le |u|^2+|v|^2$, we obtain
    \begin{align*}
    |\Phi_{i,\varepsilon}(p)-\Phi_{i,\varepsilon}(q)\| \le \sqrt{2|u|^2+2|w|^2+|v|^2}\le\sqrt{(1+2K^2)(|u|^2+|v|^2)}=\sqrt{(1+2K^2)}\|p-q\|.
    \end{align*}

    If $|u-\varepsilon w|\ge\frac{1}{K}|u|$, then
    \begin{align*}
    \|\Phi_{i,\varepsilon}(p)-\Phi_{i,\varepsilon}(q)\|=\sqrt{|u-\varepsilon w|^2+|v|^2}\ge\sqrt{\frac{1}{K^2}(|u|^2+|v|^2)}=\frac{1}{K}\|p-q\|.
    \end{align*}
    
    Now, if $|u-\varepsilon w|<\frac{1}{K}|u|$, then $|u|-|\varepsilon w|\le |u- \varepsilon w|$ implies $(1-\frac{1}{K})|u|<\varepsilon|w|\le K\varepsilon|v| \le K|v|$ and thus $|v|>\frac{K-1}{K^2}|u|$. Let $K'=1+\frac{K^4}{(K-1)^2}$. Then, $\frac{1}{\sqrt{K'-1}}=\frac{K-1}{K^2}$ and since $|v|^2 \ge\frac{1}{K'}(|u|^2+|v|^2)$ if, and only if, $|v|\ge\frac{1}{\sqrt{K'-1}}|u|$, which is true, we obtain
    \begin{align*}
    \|\Phi_{i,\varepsilon}(p)-\Phi_{i,\varepsilon}(q)\|=\sqrt{|u-\varepsilon w|^2+|v|^2}\ge\sqrt{|v|^2}\ge\frac{1}{\sqrt{K'}}\sqrt{|u|^2+|v|^2} \ge\frac{1}{\sqrt{K'}}\|p-q\|.
    \end{align*}

    Therefore, $\Phi_{i,1}$ is an ambient bi-Lipschitz map such that $\Phi_{i,1}(V_i)=\{u=0\}$, and the result follows. In particular, $(V_i,0)$ is LNE because $\{u=0\}$ is obviously LNE.
\end{proof}

\begin{theorem}\label{prop:1-braid-Lipschitz-trivial}
Let \( f \) and \( g \) be IND mixed polynomials that are $\Gamma_{\mathrm{inn}}$-nice and such that their links \( L_f \) and \( L_g \) satisfy one of the following conditions:
\begin{enumerate}
		\item[(i)] Both \( L_f \) and \( L_g \) are empty.
	\item[(ii)] Both \( L_f \) and \( L_g \) are metric 1-braid closures.
	\item[(iii)] Both \( L_f \) and \( L_g \) are non-tangent Hopf-links.
    \end{enumerate}
Then, \( f \) and \( g \) are ambient bi-Lipschitz \( V \)-equivalent, and $V(f)$ and $V(g)$ are LNE.
\end{theorem}
\begin{proof}
    If condition $(i)$ holds, the result follows by vacuity. If condition $(ii)$ holds, the result follows immediately from Proposition \ref{proposition:1-braid-LNE}. Finally, if condition $(iii)$ holds, let $V(f)=V_1\cup V_2$ and $V(g)=\tilde V_1\cup \tilde V_2$, with $\operatorname{Cont}(V_1,V_2)=\operatorname{Cont}(\tilde{V_1},\tilde{V}_2)=1$. The contact condition implies that we can assume, without loss of generality, that $C(V_1,0)=C(\tilde{V}_1,0)=\{u=0\}$ and $C(V_2,0)=C(\tilde{V}_2,0)=\{v=0\}$. For $i=1,2$, Proposition \ref{proposition:1-braid-LNE} implies that $V_1$ and $\{u=0\}$ are ambient bi-Lipschitz equivalent (as well as $V_2$ and $\{v=0\}$), with the family $\{\Phi_{i,\varepsilon}\}_{\varepsilon\in [0,1]}$ in the Claim of such Proposition being a family of ambient bi-Lipschitz maps such that $\Phi_{i,0}$ is the identity and $\Phi_{1,1}(V_1)=\{u=0\}$, $\Phi_{2,1}(V_2)=\{v=0\}$. Replacing the maps $\Phi_{i,\varepsilon}$ by the maps $\Phi_{i,\varepsilon}$ in this proof, by Lemma \ref{Lemma:Union} we obtain an ambient bi-Lipschitz isotopy $\Phi_{\varepsilon}$ such that $\Phi_1(V_1)=\{u=0\}$ and $\Phi_1(V_2)=\{v=0\}$. Analogously, there is an ambient bi-Lipschitz isotopy $\tilde{\Phi}_{\varepsilon}$ such that $\tilde{\Phi}_1(\tilde{V}_1)=\{u=0\}$ and $\tilde{\Phi}_1(\tilde{V}_2)=\{v=0\}$. Therefore, $V_1\cup V_2$ and $\tilde{V}_1\cup \tilde{V}_2$ are ambient bi-Lipschitz equivalent to $\{u=0\} \cup \{v=0\}$, which is LNE, and the result follows.
\end{proof}

\begin{corollary}\label{MainThm:1a}
		Let \(\{f + \varepsilon \theta\}_{\varepsilon \in I}\) be a family of mixed polynomials such that \(f\) is IND and $\Gamma_{\mathrm{inn}}$-nice and \(d(P_i; \theta) \geq d(P_i;f)\), $P_i \in \mathcal{P}_{\mathrm{inn}}(f)$. If the link \( L_f \) satisfies one of the following: 
		\begin{enumerate}
		\item[(i)] \( L_f \) is empty.
	\item[(ii)] \( L_f \) is metric 1-braid closure.
	\item[(iii)] \( L_f \) is non-tangent Hopf-link.
    \end{enumerate}
    Then, the family is ambient bi-Lipschitz $V$-trivial.
	\end{corollary}
\begin{example}\label{1braidclosures_type_I_II}
Let \( \omega \in \C \text{ with } |\im(\omega)|>1\).  Consider the mixed polynomials 
\begin{align*}
f(u,v)&=(u+v^2\bar{v})(u\bar{u}+ (v\bar{v})^2v^2+\omega (v\bar{v})^3), \\
 g(u,v)&=u(u\bar{u}+ v^2+\omega (v\bar{v})),\\
h(u,v)&=v(u\bar{u}+ (v\bar{v})^2v^2+\omega (v\bar{v})^3).
\end{align*}
The surfaces $V(f)$, $V(g)$ and $V(h)$ have an isolated singularities at the origin. Moreover, $f$, $g$ and  $h$ are radial mixed polynomials with $k_f=k_h=3$ and $k_g=1$. Thus, we can easily verify that $g$ is Type-II, $f$ is Type-I, and $h$ is Type-III. The links associated with these three mixed polynomials are metric 1-braid closures as $|\im (\omega)|>1$ implies that 
\[f_i(u,\rme^{\rmi t})= (u+\rme^{\rmi t})(u \bar{u}+\rme^{2\rmi t}+\omega )\]
has only one solution of the form $(-\rme^{\rmi t},\rme^{\rmi t})$
and $g$ and $h$ has only one solution of the form $(u=0)$ and $(v=0)$, respectively. Therefore, by Theorem~\ref{prop:1-braid-Lipschitz-trivial} the mixed polynomials $f$, $g$ and $h$ are ambient bi-Lipschitz $V$-equivalent. 
\end{example}

The next two propositions show that metric 1-braid closures and non-tangent Hopf-links are ambient bi-Lipschitz equivalent to the zero set of Type II semi-radial mixed polynomials whose link is isotopic to them. This shows that the topological type of such germs also determines their ambient Lipschitz geometry.

\begin{proposition}\label{trivial-knot}
    Let $f$ be an IND, $\Gamma_{\rm inn}$-nice mixed polynomial such that the link $L_f$ is a metric 1-braid closure, and let $g$ be a semi-radial, Type II mixed polynomial such that the link of $V(g)$ is isotopic to the trivial knot. Then, $f$ is ambient bi-Lipschitz $V$-equivalent to $g$.
\end{proposition}

\begin{proof}
    Suppose that $C(V(f),0)=\{u=0\}$, as the other case is analogous. Let $Q \in \Gamma_{\rm inn}$ such that $g$ is ambient bi-Lipschitz $V$-equivalent to $g_Q$ (see Corollary \ref{cor:bilipVequivsemi-radialprincipalpart}). Notice that $V(g_Q)$ is a cone, at $0$, of a smooth curve $L \subset \mathbb{S}^3$, and since $L$ is isotopic to the trivial knot $\{0\}\times\mathbb{S}^1$, whose tangent cone is $\{u=0\}$. By Lemma \ref{lemma-extension}, $V(g_Q)$ and $\{u=0\}=C(V(f),0)$ are ambient bi-Lipschitz equivalent, and the result follows from Proposition \ref{proposition:1-braid-LNE}.  
\end{proof}

\begin{proposition}\label{hopf-link}
    Let $f$ be an IND mixed polynomial that is $\Gamma_{\rm inn}$-nice such that the link $L_f$ is a non-tangent Hopf-link, and let $g$ be a semi-radial, Type II mixed polynomial such that the link of $V(g)$ is isotopic to the Hopf-link. Then, $f$ is ambient bi-Lipschitz $V$-equivalent to $g$.
\end{proposition}

\begin{proof}
By Theorem \ref{prop:1-braid-Lipschitz-trivial}, $V(f)$ is ambient bi-Lipschitz equivalent to $\{u=0\}\cup \{v=0\}$, whose link is the union of $\{0\}\times \mathbb{S}^1$ and $\mathbb{S}^1\times \{0\}$. Since $g$ is Type II of radial-type $(Q,d_{g_Q})$, and the link of $V(g)$ is isotopic to the Hopf-link, there is a smooth map $\phi: \mathbb{S}^3 \to \mathbb{S}^3$ such that $\phi(L_{g_Q})=(\{0\}\times \mathbb{S}^1)\cup (\mathbb{S}^1\times \{0\})$. Hence, by Lemma \ref{lemma-extension}, $V(g_Q)$ and $\{u=0\}\cup \{v=0\}$ are ambient bi-Lipschitz equivalent. The result then follows, since $V(g_Q)$ is ambient bi-Lipschitz equivalent to $V(g)$, by Corollary \ref{cor:bilipVequivsemi-radialprincipalpart}.
\end{proof}
\begin{example}\label{ex:typeiinotmetrics}
\begin{enumerate}
    \item[(i)] Consider the mixed polynomial $f(u,v)=u^3+v^2\bar{v}$. It is Type-II semiradial of radial-type $(P_1=(1,1);3)$, and $L_f$ is isotopic to $\mathbf{L}([L_1])$, where $L_1$ is the zeros of $f_1(u,\rme^{\rmi t})=u^3+\rme^{\rmi t}$. The link $L_1$ is the closure of the braid in 3-strands $B_1=\sigma_1\sigma_2$, where $\sigma_i$, $i=1,2$ denotes the Artin's generators of the braid group $\mathbb{B}_3$. Thus, the link $L_f$ is the trivial knot. Since $L_f$ is the closure of a 3-strand braid, it is not a metric 1-braid closure. Nevertheless,  Proposition~\ref{trivial-knot} guarantees that $f$ is ambient bi-Lipschitz $V$-equivalent to any $\Gamma_{\mathrm{inn}}$-nice and IND mixed polynomial whose link is a metric 1-braid closure.  
    \item[(ii)] Consider the mixed polynomial $f(u,v)=u^2-v^2$. It is Type-II semiradial of radial-type $(P_1=(1,1);2)$. The link $L_1$ is a torus link that is the closure of the 2-strand braid $B_1=\sigma_1^2$, where $\sigma_1$ is the Artin's generator of the braid group $\mathbb{B}^2$. The link $L_f$ is a Hopf-link, but it is not a non-tangent Hopf-link. Indeed, $L_f$ is isotopic to the links $L^1$ and $L^2$, which are parametrized by $(\rme^{\rmi t},\rme^{\rmi t})$ and $(-\rme^{\rmi t},\rme^{\rmi t})$, respectively. The tangent cone at the origin of these components are  $C(V(L^1;f),0)=V(L^1;f)$ and $C(V(L^2;f),0)=V(L^2;f)$, that are not equal neither $\{v=0\}$ nor $\{u=0\}$. Nevertheless, Proposition~\ref{hopf-link} guarantees that $f$ is ambient bi-Lipschitz $V$-equivalent to any IND and $\Gamma_{\mathrm{inn}}$-nice mixed polynomial whose link is a non-tangent Hopf-link.   
\end{enumerate}
\end{example}

\section{Necessary conditions for bi-Lipschitz equivalence}\label{Sect:necessarybi_lipVequiv}

In this section, we study IND mixed polynomials whose link is neither a 1-braid closure nor a non-tangent Hopf-link. We now focus on these cases because the associated  surfaces allow us to obtain a meaningful comparison of the tangent orders of real half-branches and their contact orders at the origin, using the so-called test arc method in \cite{Fernandes2003}. With this, we obtain several necessary conditions for outer bi-Lipschitz equivalence based on the Newton polygon associated with the mixed polynomials.

\subsection{The type of semi-radial mixed polynomial is a Lipschitz invariant}\label{subsec6.1}

\begin{proposition}\label{proptypeI}
Let $f$ and $g$ be semi-radial mixed polynomials. Assume that the links associated with both $f$ and  $g$ are neither empty nor metric 1-braid closures nor non-tangent Hopf links, and that $f$ is bi-Lipschitz $V$-equivalent to $g$. Then, 
\begin{enumerate}
\item[(i)] If $f$ is Type-II, then $g$ must also be Type-II.
\item[(ii)] If $f$ is Type-III, then $g$ must also be Type-III.  
\end{enumerate}
\end{proposition}
\begin{proof}
Let \(P_i \in \mathcal{P}_{\mathrm{inn}}(f)\) and \(Q_j \in \mathcal{P}_{\mathrm{inn}}(g)\) be the weight vectors associated with the radial-type of the semi-radial mixed polynomials $f$ and $g$, respectively.
	 \vspace{0.2cm}

\noindent \textbf{(i):} Since $f$ is Type-II by hypothesis, Lemma~\ref{tangconetypeII} ensures that $C(V(f),0)$ is bi-Lipschitz equivalent to $V(f_{P_i})$. In this case, $C(V(f),0)$ consists of the cone over the components of $L_{f_{P_i}}$, meaning it comprises cones that share a common vertex at the origin. Given that 
$f$ is bi-Lipschitz $V$-equivalent to $g$, Theorem \ref{cone-sampas} implies that the tangent cone $C(V(f),0)$ is bi-Lipschitz equivalent to $C(V(g),0)$ via a bi-Lipschitz map $\phi$. We divide the proof into two claims:
\smallskip

\noindent \textbf{Claim I.} \(g\) is not Type-I.
\smallskip

\noindent We first assume for contradiction that $g$ is Type-I. Without loss of generality, we can suppose that the radial mixed polynomial  $g_{Q_j}$ is $u$-convenient and satisfies $k_g=\frac{q_{j,1}}{q_{j,2}}>1$; an analogous argument applies for the other case. By Lemma~\ref{tangconetypeI}, the tangent cone $C(V(g),0)$ is bi-Lipschitz equivalent to a subset of $\{u=0\}$. Since $\phi$ is a bi-Lipschitz map satisfying $\phi(0)=0$, we have that  $C(V(f),0) \setminus \{0\}$ is homeomorphic to  $C(V(g
),0) \setminus \{0\}$. Therefore, by the topology of $C(V(f),0)$ and the fact that $f$ is Type II, the link
 $L_f$ must be a knot and hence $C(V(g_{Q_j}))=\{u=0\}$. Moreover, as $f$ is bi-Lipschitz $V$-equivalent to $g$, the link $L_g$ also has to be a knot.
   
Now, we apply the method of real half-branches used  in \cite{Fernandes2003}: Since, by hypothesis, $L_{g_{Q_j}}$ is not a 
metric 1-braid closure and $C(V(g_{Q_j}))=\{u=0\}$, there exist four distinct roots  $(u_l, \rme^{\rmi t_l}),\ l=1,\dots,4, $ of $(g_{Q_j})_j$ such that $t_1<t_2<t_3<t_4$, $\frac{t_3-t_1}{2\pi} \in \mathbb{Z}$, and $\frac{t_2-t_1}{2\pi},\frac{t_4-t_1}{2\pi} \notin \mathbb{Z}$. Define the real arcs: 
\begin{equation}\label{fourarcs}
\gamma_l(\lambda)=(\lambda^{k_g} u_l, \lambda \rme^{\rmi t_l}), \ l=1,\dots,4,
\end{equation} 
which are contained in $V(g_{Q_j})$. Given that $t_3=t_1+2n\pi$, $\gamma_1$, and $\gamma_3$ are tangent arcs where $\operatorname{tord} (\gamma_1,\gamma_3)=k_g>1$. 
By Proposition~\ref{radial-trivial}, since $f$ is ambient bi-Lipschitz $V$-equivalent to $f_{P_1}$ and $g$ is ambient bi-Lipschitz $V$-equivalent to $g_{Q_j}$, we can consider a bi-Lipschitz map  $F: (V(f_{P_1}),0)\to (V(g_{Q_j}),0)$. Define: 
\[\chi_l(\lambda)= (\lambda u(\tau_l(\lambda)), \lambda\rme^{\rmi t(\tau_l(\lambda))}) \in  F^{-1}(\gamma_i),\ l=1,\dots,4.\]  
Here, $(u(\tau), \rme^{\rmi t(\tau)}),\ \tau \in [0,2\pi M]$, is a parametrization of the entire $V((f_{P_i})_i)$, arising from the concatenation of the parametrizations of the curves (not necessarily closed) in Equation~\eqref{linkparametrization}.  
Since $\operatorname{tord}(\gamma_1,\gamma_3)=k_g$ and $F$ is bi-Lipschitz, it follows that: 
\begin{equation}\label{contactk}
 \operatorname{tord}(\chi_1,\chi_3)=k_g. 
\end{equation}
In this case, we have that 
\begin{equation}\label{eq1typeIinv}
\lim_{\lambda\to 0} t(\tau_1(\lambda))- t(\tau_3(\lambda))=0,
\end{equation}
 otherwise, 
\[ ||\chi_1(\lambda)- \chi_3(\lambda)||\geq ||\lambda \rme^{\rmi t(\tau_1(\lambda))}-\lambda \rme^{\rmi t(\tau_3(\lambda))}||= \lambda || \rme^{\rmi( t(\tau_1(\lambda))- t(\tau_3(\lambda)))}-1||,\]
 which would imply that $\operatorname{tord}(\chi_1,\chi_3)=1$, contradicting Equation~\eqref{contactk}. 
Therefore, Equation~\eqref{eq1typeIinv} holds, and it occurs in the following two cases: 
\begin{enumerate}
\item[(a)] \( \lim_{\lambda\to 0} (\tau_1(\lambda)- \tau_3(\lambda))=2\pi k, \ k=0 \text{ or }M,\)
\item[(b)]  \(\lim_{\lambda\to 0} \tau_1(\lambda)=\theta_1\neq \theta_3=\lim_{\lambda\to 0} \tau_3(\lambda), \text{ with }  t(\theta_1)=t(\theta_3), \ \theta_1,\theta_3\not\in \{ 0, 2\pi M\}. \)
\end{enumerate}
In the second case, we have  $\lim_{\lambda \to 0} (u( \tau_1(\lambda))-u(\tau_3(\lambda))) \neq 0.$ Thus, 
\[ ||\chi_1(\lambda)- \chi_3(\lambda)||\geq \lambda||u( \tau_1(\lambda))-u( \tau_3(\lambda))||\]
and thus there is $L>0$ such that  \[\lim_{\lambda\to 0} \frac{ ||\chi_1(\lambda)- \chi_3(\lambda)||}{\lambda}\geq L,\] implying $\operatorname{tord}(\chi_1,\chi_3)=1$, contradicting \eqref{contactk}. 

Suppose $k=0$, that is, \(\lim_{\lambda\to 0} (\tau_1(\lambda)- \tau_3(\lambda))=0;\) the case $k=M$ follows analogously.  In this scenario, one of the following alternatives occurs: 

\[\lim_{\lambda\to 0} (\tau_1(\lambda)- \tau_2(\lambda))=0\text{ or } \lim_{\lambda\to 0} (\tau_1(\lambda)- \tau_4(\lambda))=0. \]

Suppose  \(\lim_{\lambda\to 0} (\tau_1(\lambda)- \tau_2(\lambda))=0.\)
In particular,  
\[ ||\chi_1(\lambda) - \chi_2(\lambda)||\ll \lambda \therefore tord(\chi_1,\chi_2)>1.\]
Now, since $\gamma_l=F(\chi_l)$ for $l=1,2$, we have
\[||\gamma_1(\lambda) - \gamma_2(\lambda)||\gtrsim \lambda|\rme^{\rmi t_1} - \rme^{\rmi t_2} | \approx \lambda \therefore tord(\gamma_1,\gamma_2)=1,\]
which leads to a contradiction. 

\noindent \textbf{Claim II.} \(g\) is not Type-III. 
\smallskip 

\noindent Suppose by contradiction that $g$ is Type-III. Without loss of generality, we can suppose that 
$g_{Q_j}$ is not $u$-convenient and $k_g>1$. By Lemma~\ref{tangconetypeIII}, the tangent cone $C(V(g),0)$ is homeomorphic to the union of $\{v=0\}$ with a subset of $\{u=0\}$. Therefore, since $f$ is bi-Lipschitz $V$-equivalent to $g$, we get that $C(V(f),0)$ is homeomorphic to $C(V(g),0)$, and consequently by the topology of $C(V(f),0)$ and the fact that $f$ is Type II, $V(g)$ has two connected components, $V_1$ and $V_2$, each one of then with link homeomorphic to $\mathbb{S}^1$ and satisfying  $$C(V_1)=\{v=0\} \text{ and } C(V_2)=\{u=0\}.$$ 
By assumption, $g$ has a link that is not a non-tangent Hopf-link, thus the link of $V_2$ is not a metric 1-braid closure and it does not intersect $\{v=0\}$. Therefore, applying to $V_2$ the same reasoning of the four real half-branches done in Claim I of this Proposition, we obtain a contradiction.   
\vspace{0.2cm}

Therefore, since $g$ is neither Type-I nor Type-III, then $g$ is Type-II, as these three types describe the class of semi-radial mixed polynomials.  
\vspace{0.3cm}

\noindent \textbf{(ii):} Without loss of generality, we can suppose that $f_{P_i}$ is not $u$-convenient and  $k_f=\frac{p_{i,1}}{p_{i,2}}>1$.  By (i) of this proposition, we know that $g$ is not Type-II. We suppose by contradiction that $g$ is Type-I. By Remark \ref{linkofTypeIII}, we know that the link $L_f$ is isotopic to a nested braid $\mathbf{L}([L_1],L_v)$. Notice that, since $L_f$ is not a metric 1-braid closure, $L_1$ is not empty. Thus, $L_{f}$ has more than one component. 

Let $V_1:=\{v=0\}$ and $V_2$ be the components of $V(f_{P_i})$. By Lemma~\ref{tangconetypeIII},  $C(V_1,0)=\{v=0\}$ and $C(V_2,0) \subset \{u=0\}$. Hence \(\operatorname{Cont} (V_1,V_2)=1\), which implies $\operatorname{Cont}(F(V_1), F(V_2))=1$ as $F$ is a bi-Lipschitz map.  Since $g$ is Type-I, Lemma~\ref{tangconetypeI} implies
\begin{align*}
&C(F(V_1),0)\subset \{u=0\}, \\ &C(F(V_2),0)  \subset \{u=0\}, \\ &C(F(V_1),0)\cap C(F(V_2),0) =\{0\},  
\end{align*}
if $k_f>1$ and $f_P$ is $u$-convenient, or
\begin{align*}
&C(F(V_1),0)\subset \{v=0\}, \\ &C(F(V_2),0)  \subset \{v=0\}, \\ &C(F(V_1),0)\cap C(F(V_2),0) =\{0\},  
\end{align*}
if $k_f<1$ and $f_P$ is $v$-convenient.
 Thus, $C(F(V_1),0) \subsetneq  \{u=0\}$ or  $C(F(V_1),0) \subsetneq  \{v=0\}$. Since  $F$ is a bi-Lipschitz map, by Theorem \ref{cone-sampas}, one of these cones has to be homeomorphic to $C(V_1,0)=\{v=0\}$, which is a contradiction.
\end{proof}
\begin{example}\label{1braidclosures_type_I_IV}
Consider the mixed polynomials  \[f(u,v)=uv(u\bar{u}+ (v\bar{v})^2v^2+\omega (v\bar{v})^3),\text{ and }\]  \[
 g(u,v)=uv(u\bar{u}+ v^2+\omega (v\bar{v})),\ \omega \in \C \text{ with } |\im(\omega)|>1.\]
The mixed polynomials $f$ and $g$ satisfy $Sing(V(f))=Sing(V(g))=\{0\}$. Moreover, $f$ and $g$  are  radial mixed polynomials with $k_f=3$ and $k_g=1$. A direct verification  shows that $g$ is Type-II, and $f$ is Type-III. Furthermore, the links $L_f$ and $L_g$ are non-tangent Hopf-links arising from $\{uv=0\}$.

Since $\mathrm{Id}(u,v)=(u,v)$ is a bi-Lipschitz map that satisfies $\mathrm{Id}(V(f))=V(g)$, it follows that  $f$ and $g$ are bi-Lipschitz $V$-equivalent. This confirms that the condition in Proposition~\ref{proptypeI}, which excludes the links \(L_f\) and \(L_g\) from being non-tangent Hopf-links, is indeed necessary.  
\end{example}

\begin{proposition}\label{kinvarianceI}
Let \( f \) and \( g \) be semi-radial mixed polynomials such that  their links are neither empty nor metric 1-braid closures. If \( f \) is bi-Lipschitz \( V \)-equivalent to \( g \),  and $f$ is Type-I, then $g$ must be Type-I and \( k_f = k_g \text{ or } k_f = k_g^{-1}\).
\end{proposition}
\begin{proof}
	Let \(P_i \in \mathcal{P}_{\mathrm{inn}}(f)\) and \(Q_j \in \mathcal{P}_{\mathrm{inn}}(g)\) be the weight vectors associated with the radial-types of the semi-radial mixed polynomials $f$ and $g$, respectively. First, notice that it is sufficient to prove the result for $f_{P_i}$ being $u$-convenient and $k_f>1$, as those $v$-convenient with $k_f<1$ are bi-Lipschitz $V$-equivalent (via change of variables $u \Leftrightarrow v$) to a $u$-convenient satisfying $k_f>1$. 
    Notice that $L_f$ cannot be a non-tangent Hopf-Link because its tangent cone is a subset of $\{u=0\}$ by Lemma \ref{tangconetypeI}, and thus only $L_v$ can be the braid axis of the components of $L_f$. Therefore, by Proposition~\ref{proptypeI} and the fact that $f$ is Type-I, it follows that $g$ is neither Type-II nor Type-III.
  Applying the same change of variables ($u \Leftrightarrow v$) if necessary, we assume $k_g > 1$ and that $g_{Q_j}$ is $u$-convenient. Consequently, both $f$ and $g$ are Type-I semi-radial mixed polynomials, and we aim to show that $k_f = 
  k_g$. 
  
  Since $f$ and $g$ are bi-Lipschitz $V$-equivalent, by Corollary \ref{Lipschitz-homogeneo}, there is a bi-Lipschitz map $F: (V(f_{P_1}),0)\to (V(g_{Q_1}),0)$. We now divide the proof into two cases:
  \smallskip
  
	\noindent \textbf{Case I}: $L_{f_{P_1}}$ has a component $J$ that is not a metric 1-braid closure.
	 \smallskip
     
 \noindent If $C(V(J,f_{P_1}),0)$ is a line, then $C(F(V(J,f_{P_1})),0)$ is also a line, by Theorem \ref{cone-sampas}. By radial homogeneity, $V(J,f_{P_1})$ and $F(V(J,f_{P_1}))$ can be parametrized as $(r^{k_f}u(\tau),r\rme^{\rmi t(\tau)})$ and $(r^{k_g}\tilde{u}(\tau),r\rme^{\rmi \tilde{t}(\tau)})$, respectively, with $t(\tau)$, $\tilde{t}(\tau)$ being constants because (see the proof of Lemma \ref{tangconetypeI})
 $$C(V(J;f_{P_1}),0)=T_2(V(J;f_{P_1})) \text{ and } C(F(V(J;f_{P_1})),0)=T_2(F(V(J;f_{P_1}))).$$ 
 Hence, $V(J;f_{P_1})$ is a $k_f$-horn and $F(V(J;f_{P_1}))$ is a $k_g$-horn, and therefore $k_f=k_g$. 
 
 \noindent
 If $C(V(J,f_{P_1}))$ is not a line, then $C(F(V(J,f_{P_1})),0)$ is not a line, by Theorem \ref{cone-sampas}. By radial homogeneity, $V(J,f_{P_1})$ and $F(V(J,f_{P_1}))$ can be parametrized as $(r^{k_f}u(\tau),r\rme^{\rmi t(\tau)})$ and $(r^{k_g}\tilde{u}(\tau),r\rme^{\rmi \tilde{t}(\tau)})$, respectively. As $C(V(J;f_{P_1}),0)=T_2(V(J;f_{P_1}))$ and $V(J;f_{P_1})$ is not a 1-braid closure, there are $\tau_1<\tau_2<\tau_3<\tau_4$ such that $0<t(\tau_1)<t(\tau_2)=t(\tau_4)<t(\tau_3)<2\pi$ and $u(\tau_2)\ne u(\tau_4)$. Therefore, for $i=1,2,3,4$, the arcs $\gamma_i(r)=(r^{k_f}u(\tau_i),r\rme^{\rmi t(\tau_i)})$ satisfy $tord(\gamma_i,\gamma_j)=1$ for every $1\le i<j\le 4$, except for $i=2,j=4$, where $tord(\gamma_i,\gamma_j)=k_f$. Thus, if $k_f>k_g$, then there is a $k_f$-H\"older triangle $T$ in $F(V(J;f_{P_1}))$ whose boundary arcs are $F(\gamma_2)$ and $F(\gamma_4)$. Since either $F(\gamma_1)$ or $F(\gamma_3)$ are in $T$, we have $tord(F(\gamma_1),F(\gamma_2))>1$ or $tord(F(\gamma_3),F(\gamma_2))>1$, a contradiction. By reversing the roles of $f$ and $g$, we also obtain a contradiction from $k_g>k_f$. Therefore, $k_f=k_g$.
 \smallskip
 	
 \noindent \textbf{Case II}: $L_{f_{P_1}}$ and $L_{g_{Q_1}}$ are unions of metric 1-braid closures with braid axis $L_v$.
 	 	 \smallskip
         
 \noindent Let $V(I_1;f_{P_1})$ and $V(I_2;f_{P_1})$ be components of $V(f_{P_1})$, where $I_1$ and $I_2$ are components of $L_{f_{P_1}}$. By Proposition \ref{nontangencycriteria}, the contact order of these two components is given by
\begin{equation*}\label{contactf}
	 	\operatorname{Cont}(V(I_1;f_{P_i}), V(I_2;f_{P_i}))=k_f.  
	 \end{equation*}
	 Similarly, for  two components of $L_{g_{Q_1}}$, say, $J_1$ and $J_2$,
	 \begin{equation*}\label{contactg}
	 	\operatorname{Cont}(V(J_1;g_{Q_1}) ,V(J_2;g_{Q_1})) = 
	 	k_g.
	 \end{equation*}
	 	Therefore, since $f$ is bi-Lipschitz $V$-equivalent to $g$ and the contact order is an invariant of bi-Lipschitz maps, we obtain $k_f=k_g$. 
\end{proof}
Following the same idea as above, we obtain:
\begin{proposition}\label{kinvarianceIII}
Let \( f \) and \( g \) be semi-radial mixed polynomials such that their links are neither empty nor metric 1-braid closure nor non-tangent Hopf-links. If \( f \) is bi-Lipschitz \( V \)-equivalent to \( g \),  and $f$ is Type-III, then $g$ must be Type-III and \( k_f = k_g \text{ or } k_f = k_g^{-1}\).
\end{proposition}
\begin{proof}
	By Proposition~\ref{proptypeI}, both $f$ and $g$ are Type-III. Without loss of generality, we can suppose that $\mathcal{P}_{\mathrm{inn}}(f)=\{P_1,P_2\}$ and $\mathcal{P}_{\mathrm{inn}}(g)=\{Q_1,Q_2\}$, where $f$ and $g$ are semi-radial of radial-type $(P_1;d(P_1;f))$ and $(Q_1;d(Q_1;f))$, respectively, and we can also suppose $k_f>1$ and $k_g>1$. Since the links of \(f\) and \(g\) are neither metric 1-braid closure nor non-tangent Hopf-links, then the links of $V(L_{1,f};f)$ and $V(L_{1,g};g)$ are not metric 1-braid closures. Applying the same ideas in the proof of Proposition~\ref{kinvarianceI} to these components of $V(f)$ and $V(g)$, we obtain $k_f=k_g$.  
\end{proof}
Next, we observe that the ideas underlying the invariance (up multiplicative inverse) of the number $k_f$ in the semi-radial case can be extended to IND mixed polynomials that are $\Gamma_{\mathrm{inn}}$-nice. 
For the purpose of classification under the bi-Lipschitz $V$-equivalence, we may assume that if $\mathcal{P}_{\mathrm{inn}}(f)=\{P_1\}$, then $L_{P_1,f}=V_{P_1,f} \cap \mathbb{S}^3 \subset \C \times \C^{*}$ is compact. This assumption is justified because $\Gamma_{\mathrm{inn}}(f)$ having only one 1-face forces $f$ to be semi-radial. Consequently, Theorem~\ref{MainThm:1} guarantees a deformation of $f$ into an ambient bi-Lipschitz $V$-equivalent mixed polynomial $f_{\varepsilon_0}$, which is also semi-radial and whose set $L_{P_1,f_{\varepsilon_0}}$ is compact. 

\begin{corolario}\label{kinvarianceGammaniceIND}
	Let $f$ and $g$ be bi-Lipschitz $V$-equivalent IND mixed polynomials that are $\Gamma_{\mathrm{inn}}$-nice. If there exists a non-empty connected component of $L_{P_i,f},\ i \in I_f^{(k\neq 1)}$ (see Equation \eqref{Def:Ifk}), that is not a metric 1-braid closure, then there exists ${\rm j}(i) \in I_g^{(k\neq 1)}$ such that $k_{i,f}=k_{{\rm j}(i),g}$ or $k_{i,f}=k_{{\rm j}(i),g}^{-1}$.  
\end{corolario}

\begin{proof}
Consider a bi-Lipschitz map $\tilde{F}: V(f) \to V(g)$. Let $J$ be a non-empty knot of the link $L_{P_i,f},\ i \in I_f^{(k\neq 1)}$. We choose ${\rm j}(i) \in I_{g}$ such that $\tilde{F}(V(J;f))=V(\tilde{J};g)$ and $\tilde{J}$ is a sublink of $L_{Q_{{\rm j}(i),g}}$.  By Lemma~\ref{lemma:lipeomorphismlocal}, there exists a bi-Lipschitz map 
$F: V(J;f_{P_i}) \to V(\tilde{J};g_{Q_{{\rm j}(i)}})$.  Applying the same arguments of the proof of Proposition~\ref{kinvarianceIII}  to the bi-Lipschitz map $F$ and link $J$ (that is not metric 1-braid closure), we obtain $k_{i,f}=k_{{\rm j}(i),g}$ or $k_{i,f}=k_{{\rm j}(i),g}^{-1}$. Moreover,  since $i \in I_f^{(k\neq 1)}$, ${\rm j}(i) \in I_g^{(k\neq 1)}$.
\end{proof}

\subsection{The contact data of a mixed polynomial is a Lipschitz invariant}\label{subsec6.2}

Let $f$  be a mixed polynomial, let $i_{\mathrm{m}}(f):=\min (I_f)$, $i_{\mathrm{M}}(f):=\max(I_f)$ (see Equation \ref{Def:If}) and $\mathcal{P}_{\mathrm{inn}}(f)=\{P_1,\dots,P_{N_f}\}$ (see Definition \ref{def:gammainn}). We define the contact data set $\mathcal{C}(f)$ as follows:

	\begin{align}\label{contactdata}
		\mathcal{C}(f) := \bigg \{ (\kappa, m(\kappa)) \in \mathbb{Q}_{>0} \times \{1, 2\} \ | \ \kappa \in \{k_i \mid i \in I_f^{(k \geq 1)} \} \cup \{ k_i^{-1} \mid i \in I_f^{(k < 1)} \} \bigg \},
	\end{align} 

where
$$m(\kappa) = \begin{cases} 
			2 & \text{if } \kappa, \kappa^{-1} \in \{k_i \mid i \in I_f^{(k \neq 1)}\}, \\
			1 & \text{otherwise}.
		\end{cases}.$$

We also define $\mathcal{NC}(f):=(N_f,\mathcal{C}(f))$ and define the following two conditions:
\begin{equation}
	\tag{C}
	\label{Eq:propertyC}
	\begin{aligned}
	 I_f^{(k>1)}\neq \emptyset \Rightarrow 
		\begin{cases}
				 \text{There is a sublink } I \subset L_{i},\ i\in I_f^{(k > 1)}, \quad \text{satisfying} \quad  \mathrm{proj}_2(I)=\mathbb{S}^1, \text{ and }  \\
				L_{P_{i_{\mathrm{m}}(f)}}\ \text{ is not a metric 1-braid closure.}
		\end{cases}	
	\end{aligned}
\end{equation}

\begin{equation}
	\tag{D}
	\label{Eq:propertyD}
	\begin{aligned}
		I_f^{(k<1)}\neq \emptyset \Rightarrow 
		\begin{cases}
			 \text{There is a sublink } I \subset L_{\underline{i}},\ i\in I_f^{(k < 1)},\quad  \text{satisfying} \quad  \mathrm{proj}_1(I)=\mathbb{S}^1, \text{ and }   \\
			L_{P_{i_{\mathrm{M}}(f)}}\ \text{ is not a metric 1-braid closure.}
		\end{cases}	
	\end{aligned}
\end{equation}
Finally, let \( f \) and \( g \) be IND mixed polynomials that are bi-Lipschitz $V$-equivalent via a bi-Lipschitz map \( F \). Suppose that $f$ and $g$ are $\Gamma_{\mathrm{inn}}$-nice and satisfy \eqref{Eq:propertyC} and \eqref{Eq:propertyD}. Here, $\mathcal{P}_{\mathrm{inn}}(g)=\{Q_1,\dots,Q_{N_g}\}$.
\begin{lemma}\label{lemma:1MainThm:1.3}
	Let $f$ and $g$ be mixed polynomials as above. Then,
	\begin{align}
	&F(\cup_{i \in I_f^{(k=1)}} V(L_{P_i,f};f))=\cup_{i \in I_g^{(k=1)}} V(L_{Q_i,g};g), \nonumber\\
	&F(\cup_{i \in I_f^{(k>1)}} V(L_{P_i,f};f))=\cup_{i \in I_g^{(k>1)}} V(L_{Q_i,g};g),\label{Equation1lemma43}\\
	&F(\cup_{i \in I_f^{(k<1)}} V(L_{P_i,f};f))=\cup_{i \in I_g^{(k<1)}} V(L_{Q_i,g};g), \nonumber \\
	&\ \ \ \ \ \ \ \ \ \ \ \ \ \ \ \ \ \ \text{ or }\nonumber  \\
		&F(\cup_{i \in I_f^{(k=1)}} V(L_{P_i,f};f))=\cup_{i \in I_g^{(k=1)}} V(L_{Q_i,g};g),\nonumber \\
 &	F(\cup_{i \in I_f^{(k>1)}} V(L_{P_i,f};f))=\cup_{i \in I_g^{(k<1)}} V(L_{Q_i,g};g),\label{Equation2lemma43}\\
 	 &F(\cup_{i \in I_f^{(k<1)}} V(L_{P_i,f};f))=\cup_{i \in I_g^{(k>1)}} V(L_{Q_i,g};g).\nonumber 
 	\end{align}
\end{lemma}
\begin{proof} 
	Assume $I_f^{(k>1)}\neq \emptyset$. By \eqref{Eq:propertyC}, the link $L_{P_{i_{\mathrm{m}}(f)},f}$ is not a metric 1-braid closure. Consequently, there exists a sublink $I$ of $L_{P_{i_{\mathrm{m}}(f)},f}$ such that either  
\begin{equation}\label{eq:1lemmazones}
	F(V(I;f))  \subset \cup_{i \in I_g^{(k>1)}} V(L_{Q_i,g};g) 
	\end{equation}
or 
\begin{equation}\label{eq:2lemmazones}
 	F(V(I;f))  \subset \cup_{i \in I_g^{(k<1)}} V(L_{Q_i,g};g).
 \end{equation} 
 To see this, notice that if $L_{P_{i_{\mathrm{m}}(f)},f}$ contains a component that is not a metric 1-braid closure, Corollary~\ref{kinvarianceGammaniceIND} immediately provides Equation~\eqref{eq:1lemmazones} or \eqref{eq:2lemmazones} for that component. Otherwise, if  $L_{P_{i_{\mathrm{m}}(f)},f}$ consists entirely of metric 1-braid closures (it must to have at least two), Proposition~\ref{nontangencycriteria} implies that their respective components in $V(L_{P_{i_{\mathrm{m}}(f)},f};f)$ have pairwise contact order greater than one. Applying Proposition~\ref{nontangencycriteria} again, we obtain Equation~\eqref{eq:1lemmazones} or \eqref{eq:2lemmazones} for $I$ being the whole $L_{P_{i_{\mathrm{m}}(f)},f}$. 
  
 Now, suppose that Equation \eqref{eq:2lemmazones} holds. By Condition \eqref{Eq:propertyC} and Proposition~\ref{nontangencycriteria}, there exists a knot $J \subset L_{P_i,f}$ for some  $i \in I_f^{(k>1)}$, such that $V(J;f)$ has contact order greater than one to any component of $\cup_{i \in I_f^{(k>1)}} V(L_{P_i,f};f)$. Thus, 
 \begin{equation}\label{eq:3lemmazones}
 	 \operatorname{Cont} (F(V(I;f)),F(V(J;f)))>1.
\end{equation}
 By Proposition~\ref{nontangencycriteria}, and using Equations \eqref{eq:2lemmazones} and \eqref{eq:3lemmazones}, we obtain
\begin{equation}\label{eq:4lemmazones}
		F(V(J;f))  \subset \cup_{i \in I_g^{(k<1)}} V(L_{Q_i,g};g). 
\end{equation}
 Therefore, applying similar arguments, using Equation \eqref{eq:4lemmazones} and that \(\operatorname{Cont}(V(J;f), V(K;f))>1 \) for  any  $K \subset \cup_{i \in I_f^{(k>1)}} L_{P_i,f}$, we conclude, by Proposition~\ref{nontangencycriteria}, that 
 \begin{equation}\label{eq:5lemmazones} 
  F(\cup_{i \in I_f^{(k>1)}} V(L_{P_i,f};f))  \subset \cup_{i \in I_g^{(k<1)}} V(L_{Q_i,g};g).
\end{equation}

Following similar reasoning with $F^{-1}$ and Condition \eqref{Eq:propertyD} of $g$, along with Equation~\eqref{eq:5lemmazones}, we
obtain 
\begin{equation}\label{eq:6lemmazones} 
	\begin{aligned}
		F(\cup_{i \in I_f^{(k>1)}} V(L_{P_i,f};f))  = \cup_{i \in I_g^{(k<1)}} V(L_{Q_i,g};g).
	\end{aligned}
\end{equation}
By applying the same approach, we obtain
\begin{equation}\label{eq:7lemmazones} 
	\begin{aligned}
		F(\cup_{i \in I_f^{(k<1)}} V(L_{P_i,f};f))  = \cup_{i \in I_g^{(k>1)}} V(L_{Q_i,g};g).
	\end{aligned}
\end{equation}
 Thus, Equation~\eqref{eq:6lemmazones} and \eqref{eq:7lemmazones} imply Equation~\eqref{Equation2lemma43}.
 
In the case where Equation \eqref{eq:1lemmazones} holds, we proceed in a similar manner and obtain Equation~\eqref{Equation1lemma43}.
  \end{proof}
\begin{lemma}\label{Lemma:2MainThm:1.3}
	Let $f$ and $g$ be as Lemma~\ref{lemma:1MainThm:1.3}. Then,  there is $\chi \in \{-1,1\}$ such that, for any $i \in I_f$, there exists a unique $\mathrm{j}(i) \in I_g$ such that
	\begin{equation}\label{eq:lemma2necessaaryconditions}
	 F(V(L_{P_i,f};f))=V(L_{Q_{\mathrm{j}(i)},g};g),  \text{ with }k_{i,f}= k_{{\rm j}(i),g}^{\chi}.
	 \end{equation} 
\end{lemma}
\begin{proof}
	By Lemma~\ref{lemma:1MainThm:1.3}, if $I_f^{(k=1)}\neq \emptyset$, then  $I_g^{(k=1)}\neq \emptyset$ and 
	\begin{equation}\label{eq:1lemmabijection} 
		F(V(L_{P_i,f};f))=V(L_{Q_{\mathrm{j}(i)},g};g),\ i \in I_f^{(k=1)} \text{ and } \mathrm{j}(i)\in I_g^{(k=1)}.
	\end{equation}
We assume that Equation~\eqref{Equation1lemma43} holds ($\chi=1$), as the proof when Equation~\eqref{Equation2lemma43} holds ($\chi=-1$) is deduced analogously. By Condition \eqref{Eq:propertyC}, $L_{P_{i_{\mathrm{m}}(f)},f}$ is not a metric 1-braid closure, and thus there exists a component $I$ of $L_{P_{i_{\mathrm{m}}(f)},f}$ and $j \in I_{g}^{k>1}$ satisfying 
\begin{equation}\label{eq:3lemmabijection}
	F(V(I;f)) \subset   V(L_{Q_j,g};g), \text{ with } k_{i_{\mathrm{m}}(f),f}=k_{j,g}.
\end{equation}
To see this, notice that if $L_{P_{i_{\mathrm{m}}(f)},f}$ contains a component that is not a metric 1-braid closure, Corollary~\ref{kinvarianceGammaniceIND} provides Equation~\eqref{eq:3lemmabijection} for that component. Otherwise, if  $L_{P_{i_{\mathrm{m}}(f)},f}$ consists entirely of metric 1-braid closures, we select, by Condition \eqref{Eq:propertyC}, two different components $J_1$ and $J_2$. By Corollary~\ref{kinvarianceGammaniceIND} and Equation~\eqref{Equation1lemma43}, the links of $F(V(J_1;f))$ and $F(V(J_2;f))$ are also metric 1-braid closure. Therefore, by Proposition~\ref{nontangencycriteria},
\begin{equation*} 
k_{i_{\mathrm{m}}(f),f}=\operatorname{Cont} (V(J_1;f),V(J_2;f))=\operatorname{Cont} (F(V(J_1;f)),F(V(J_2;f)))=k_{j,g},
\end{equation*}
and either $F(V(J_1;f))$ or $F(V(J_1;f))$ is contained in $V(L_{Q_j,g};g)$.
This yields Equation~\eqref{eq:3lemmabijection} for one of these two components.

Now, suppose that the component $I$ in Equation~\eqref{eq:3lemmabijection} satisfies $j>i_{\mathrm{m}}(g)$. Then, we find that $k_{i_{\mathrm{m}}(f),f}=k_{j,g}<k_{i_{\mathrm{m}}(g),g}$. Let $K_1$ be a component of  $L_{Q_{i_{\mathrm{m}}(g)},g}$, so 
$F^{-1}(V(K_1;g)) \subset V(L_{P_i,f};f)$, for some $i \in I_{f}^{k>1}$ with $k_i \leq k_{i_{\mathrm{m}}(f),f}$.  Using Equation~\eqref{Equation1lemma43} and these inequalities, we find  $1<k_i < k_{i_{\mathrm{m}}(g),g}$, which, by Corollary~\ref{kinvarianceGammaniceIND}, implies that both $K_1$ and the link of $F^{-1}(V(K_1;g))$ are metric 1-braid closures. Applying Proposition~\ref{nontangencycriteria},
\begin{equation*}
k_i=\operatorname{Cont} (V(I;f),F^{-1}(V(K_1;g)))=\operatorname{Cont} (F(V(I;f)),V(K_1;g))=k_{i_{m}(f),f}.
\end{equation*}
Since $L_{Q_{i_{\mathrm{m}}(g)},g}$ is not a metric 1-braid closure, there exists a second component $K_2$ of $L_{Q_{i_{\mathrm{m}}(g)},g}$ with the same properties as $K_1$. Thus, both $K_2$ and the link of $F^{-1}(V(K_2;g)) \subset V(L_{i_{\mathrm{m}}(f),f};f)$ are metric 1-braid closures. Calculating the respective contacts, we obtain  
	\begin{equation*}
k_{i_{m}(f),f}=\operatorname{Cont} (F^{-1}(V(K_1;g)),F^{-1}(V(K_2;g)))=
\operatorname{Cont} (V(K_1;g),V(K_2;g))=k_{i_{m}(g),g},
\end{equation*}
which leads to a contradiction. Therefore, the component $I$ of $L_{P_{i_{\mathrm{m}}(f)},f}$ satisfies Equation~\eqref{eq:3lemmabijection} with $j=i_{\mathrm{m}}(g)$ and $k_{i_{\mathrm{m}}(f),f}=k_{i_{\mathrm{m}}(g),g}$.
 
Let $J$ be a component of $L_{i,f}$ with $i\in I_f^{(k>1)}$, thus   $F(V(J;f)) \subset V(L_{Q_{j},g};g)$ for some $j\in I_g^{(k>1)}$. If $J$ is a metric 1-braid closure, then by Corollary~\ref{kinvarianceGammaniceIND}, $F(V(J;f))$ is also a metric 1-braid closure. By Proposition~\ref{nontangencycriteria}, we obtain
\begin{equation*}
k_{i,f}=\operatorname{Cont} (V(I;f),V(J;f))=\operatorname{Cont} (F(V(I;f)),F(V(J;f)))= k_{j,g}.
\end{equation*}
If $J$ is not a metric 1-braid closure, then  by Corollary~\ref{kinvarianceGammaniceIND}, we also find $k_{i,f}=k_{j,g}$. Therefore, for any $i \in I_f^{(k>1)}$,
\[F(V(L_{P_i,f};f))\subseteq  V(L_{Q_{j},g};g) \text{ with }k_{i,f}= k_{j,g}. \]Using the same arguments above with $F^{-1}$, we establish Equation \eqref{eq:lemma2necessaaryconditions} for $i \in I_f^{(k>1)}$ and $j \in I_g^{(k>1)}$. Applying Condition \eqref{Eq:propertyD} of $f$ and Condition \eqref{Eq:propertyC} of $g$, we obtain Equation \eqref{eq:lemma2necessaaryconditions} for $i \in I_f^{(k<1)}$ and $j\in I_g^{(k<1)}$. Therefore, considering also Equation~\eqref{eq:1lemmabijection}, the result follows

\end{proof}

\begin{theorem}\label{MainThm:1.4}
	Let $f$ and $g$ be IND mixed polynomials that are $\Gamma_{\mathrm{inn}}$-nice and satisfy Conditions \eqref{Eq:propertyC} and \eqref{Eq:propertyD}. If $f$ is bi-Lipschitz $V$-equivalent to $g$, then $\mathcal{C}(f)=\mathcal{C}(g).$ Furthermore, if additionally $f$ and $g$ are $\Gamma_{\mathrm{inn}}$-true, then  $\mathcal{NC}(f)=\mathcal{NC}(g).$ 
\end{theorem}

\begin{proof}
The result follows from Lemma~\ref{Lemma:2MainThm:1.3}, as there exists a bijection between $I_f$ and $I_g$, taking $i$ to ${\rm j}(i)$ and satisfying Equation \eqref{eq:lemma2necessaaryconditions}. Moreover,  $k_{i,f}= k_{{\rm j}(i),g}$ or $k_{i,f}=(k_{{\rm j}(i),g})^{-1}$. If both mixed polynomials $f$ and $g$ are $\Gamma_{\mathrm{inn}}$-true, then the numbers $|I_f|$ and  $|I_g|$ coincide with the numbers of 1-faces of $\Gamma_{\mathrm{inn}}(f)$ and $\Gamma_{\mathrm{inn}}(g)$, respectively. Thus, $\mathcal{NC}(f)=\mathcal{NC}(g)$. 
	\end{proof}
Next, we present sufficient conditions for a mixed polynomial to satisfy Conditions~\eqref{Eq:propertyC} and \eqref{Eq:propertyD}. This result demonstrates the significance of Theorem~\ref{MainThm:1.4}, since a wider class of mixed polynomials satisfies such hypotheses.
    \begin{proposition}\label{suficientcondforCandD}
	 Let $f$ be an IND mixed polynomial that is $\Gamma_{\mathrm{inn}}$-nice with $\mathcal{P}_{\mathrm{inn}}(f)=\{P_1,\dots, P_N\}$ and satisfies the following two conditions:  
	 \begin{itemize}
	 	\item[(i)] $I_f^{(k>1)}\neq \emptyset$ implies that the function $f_{P_1}$ is $u$- (or $\bar u$-) semiholomorphic and $\deg_{|u|} f_{P_1} \geq 2$; 
        \item[(ii)] $I_f^{(k<1)}\neq \emptyset$ implies that the function $f_{P_N}$ is $v$- (or $\bar v$-) semiholomorphic and $\deg_{|v|} f_{P_N}\geq 2$.
	 \end{itemize}
     Then, $f$ satisfies conditions \eqref{Eq:propertyC} and \eqref{Eq:propertyD}.
    \end{proposition}
    \begin{proof}
    If $I_f^{(k>1)}= \emptyset$, then  condition \eqref{Eq:propertyC} is satisfied by vacuity. If $I_f^{(k>1)}\neq  \emptyset$, then  the function $(f_{P_1})_1$ is holomorphic in variable $u$ (or $\bar{u}$ in the $\bar{u}$-semiholomorphic case), and $k_1>1$. Due to holomorphicity, the Fundamental Theorem of Algebra guarantees that, for any value of $t \in [0,2\pi]$, $(f_{P_1})_1$ possesses zeros parametrised by $(u_i(t),\rme^{\rmi t}),\ i=1,\dots M$, where $M$ is defined as the greatest degree of $(f_{P_1})_1$ in variable $u$ (or $\bar{u}$). From this, we get that  $i_m(f)=1$ and $\operatorname{proj}_2(L_1)=\mathbb{S}^1$. Furthermore, since $\deg_{|u|} f_{P_1} \geq 2$, it follows that  $M\geq 2$. Since $k_1>1$, we conclude that $L_{P_1}$ cannot be a metric 1-braid closure. Therefore, $f$ satisfies Condition~\eqref{Eq:propertyC}. Analogously, using vacuity or (ii), we obtain that $f$ satisfies Condition~\eqref{Eq:propertyD}. 
    \end{proof}
    
	\begin{example}\label{Ex:th4.1}
     Consider the IND mixed polynomials $$f(u,v)=u^{14}+u^{10} v^2+u^7(v^2\bar v^2)+u^5(v^3\bar v^3)+u^3(v^5\bar v^4)+(v^{9}\bar v^{9})$$ and 
    $$g(u,v)=u^{11}+u^{7}v^2+u^4(v^2\bar v^2)+u^2(v^3\bar v^3)+(v^6\bar v^6).$$
Since both mixed polynomials are $u$-semiholomorphic (see Remark \ref{braidclosure}), these are $\Gamma_{\mathrm{inn}}$-nice and $\Gamma_{\mathrm{inn}}$-true. Using  Proposition~\ref{suficientcondforCandD}, we conclude that $f$ and $g$ satisfy  conditions~\eqref{Eq:propertyC} and \eqref{Eq:propertyD}. The sets $\mathcal{NC}(f)$ and $\mathcal{NC}(g)$ are:
     $$\mathcal{NC}(f)=(5;(3,1),(2,1),(\tfrac{3}{2},2),(1,1)) \text{ and } \mathcal{NC}(g)=(4;(3,1),(2,1),(\tfrac{3}{2},1),(1,1)).$$
     Therefore, by Theorem~\ref{MainThm:1.4}, $f$ and $g$ are not bi-Lipschitz $V$-equivalent, although the set of their $k_i,k_i^{-1}$ (without multiplicity) is the same (namely $\{3,2,\frac{3}{2},1\}$).
		\end{example}

\section{Further examples and remarks about bi-Lipschitz V-triviality}\label{section7}

We finish the paper by providing two interesting examples. The first one shows that bi-Lipschitz triviality is false in general, even when we have topological triviality. The second one shows that it is possible to obtain ambient bi-Lipschitz triviality in families that are not in the form $\{f+\varepsilon\theta \}_{\varepsilon \in I}$. Such examples show that although the topology of mixed polynomials shares some topological rigidity properties with holomorphic polynomials, their Lipschitz geometry behaves differently. That is why we had to develop proper methods to study their Lipschitz geometry, instead of using the techniques developed in \cite{neumann-pichon} and \cite{Kerner-Mendes} directly. 

\begin{example}\label{contraexemploimportante}
     By Theorem~\ref{thm:link-triviality} and Remark~\ref{rem-topological}(b), any family \(\{f + \varepsilon \theta\}_{\varepsilon \in I}\) that is $\Gamma_{\rm{inn}}$-nice and \( d(P_i; \theta) \geq d(P_i; f) \) for \( P_i \in \mathcal{P}_{\mathrm{inn}}(f)\) is ambient topologically $V$-trivial. Here, we present a family that proves that such conditions alone are insufficient to ensure bi-Lipschitz \( V \)-triviality. This highlights the importance of the semi-radial condition within the framework of bi-Lipschitz \( V \)-triviality. Consider the family
		\[
		\{f + \varepsilon \theta\}_{\varepsilon \in I} = \{(u^3 - (1 + \varepsilon \rmi
		)v\bar{v})(u^5 - v\bar{v})\}_{\varepsilon \in I}.
		\]

       Notice that, for $\varepsilon=0$, the mixed polynomial
		\(f=f(u,v) = (u^3 - v\bar{v})(u^5 - v\bar{v})
		\) is in the family. A direct calculation shows that \( f \) is IND. Furthermore, since \( f \) is \( u \)-semiholomorphic, it is also \( \Gamma_{\rm{inn}} \)-nice. We also have \[\mathcal{P}_{\mathrm{inn}}(f) = \{P_1, P_2\} \, ; \, P_1 = (2,3), P_2 = (2,5).\]
		
		The links \( L_{\underline{1},f} \) and \( L_{\underline{2},f} \) have 3 and 5 components, respectively. Let \( L^1_{\underline{1},f} \) and \( L^1_{\underline{2},f} \) denote the components of \( L_{\underline{1},f} \) and \( L_{\underline{2},f} \), respectively, whose parametrizations are given by \((1, \rme^{\tau \rmi}), \tau \in [0,2\pi]\) for both. Notice that \(
\mathrm{proj}_1(L^1_{\underline{1},f}) \cap \mathrm{proj}_1(L^1_{\underline{2},f}) \neq \emptyset,
		\) and both projections have empty interiors. Thus, \( f \) does not satisfy Condition \eqref{Eq:propertyB}. The corresponding surfaces \( V(L^1_{\underline{1},f};f) \) and \( V(L^1_{\underline{2},f};f) \) can be parametrized by \((R, R^{\frac{3}{2}}\rme^{\tau \rmi})\) and \((R, R^{\frac{5}{3}}\rme^{\tau \rmi})\), respectively. By Proposition~\ref{nontangencycriteria}~(iv), we have
		\[
		1 \leq \operatorname{Cont}(V(L^1_{\underline{1},f};f), V(L^1_{\underline{2},f};f)) \leq \frac{3}{2}.
		\]
		Since the real half-branches \(\gamma_1(\lambda) = (\lambda, \lambda^{\frac{3}{2}})\) and \(\gamma_2(\lambda) = (\lambda, \lambda^{\frac{5}{2}})\) satisfy \(\operatorname{tord}(\gamma_1, \gamma_2) = \frac{3}{2}\), it follows that
		\[
		\operatorname{Cont}(V(L^1_{\underline{1},f};f), V(L^1_{\underline{2},f};f)) =\frac{3}{2}.
		\]
		Similarly, using Proposition~\ref{nontangencycriteria}~(iv), we can show that the contact at the origin between pairwise distinct components of \( V(f) \), except for the pair \((V(L^1_{\underline{1},f};f), V(L^1_{\underline{2},f};f))\), is equal to 1.
        
On the other hand, by Theorem~\ref{thm:link-triviality}, the links \( L_f \) and \( L_{f+\varepsilon \theta} \) are isotopic. Denote by \( L^1_{\underline{1},f+\varepsilon\theta}\) and \( L^1_{\underline{2},f+\varepsilon\theta}\) the respective images of \( L^1_{\underline{1},f} \) and \( L^1_{\underline{2},f} \) under this isotopy. These components can be parametrized as \[\left( \rme^{\frac{-\arctan(\varepsilon)}{2}\rmi},\dfrac{1}{|1+\varepsilon \rmi|^{\frac{1}{2}}}\rme^{\tau \rmi}\right) \text{ and } (1, \rme^{\tau \rmi}),\] respectively. For \( \varepsilon \neq 0 \), we have
		\(
		\mathrm{proj}_1(L^1_{\underline{1},f+\varepsilon\theta}) \cap \mathrm{proj}_1(L^1_{\underline{2},f+\varepsilon\theta}) = \emptyset.
		\) Thus, by Proposition~\ref{nontangencycriteria}~(iv), 
		\[
		\operatorname{Cont}(V(L^1_{\underline{1},f+\varepsilon\theta};f+\varepsilon\theta), V(L^1_{\underline{2},f+\varepsilon\theta};f+\varepsilon\theta)) = 1.
		\]
		The same contact at the origin holds for all other pairwise components of \( V(f+\varepsilon\theta) \). Since the contact at the origin is an invariant of bi-Lipschitz \( V \)-equivalence, and \( V(f) \) has pairwise components with contact at the origin equal to \( \frac{3}{2} \), \( V(f) \) and \( V(f+\varepsilon\theta) \) cannot be bi-Lipschitz \( V \)-equivalent for \( \varepsilon \neq 0 \). Consequently, the family is not bi-Lipschitz \( V \)-trivial.
\end{example}
The following example demonstrates that families, not necessarily of the form \( \{f(u,v) + \varepsilon \theta(u,v)\}_{\varepsilon \in \R} \), can still be ambient bi-Lipschitz \( V \)-trivial. Additionally, it shows that the parameter of ambient bi-Lipschitz \( V \)-trivial families can be a complex parameter and does not need to be sufficiently small, even when the deformation terms are on the Newton boundary of \( f \).

\begin{example}
We present a specific example: a family of radial mixed polynomials, $\{f_{\omega}\}_{\omega \in \mathbb{C}}$. This family is characterised by $\mathcal{P}_{\mathrm{inn}}(f_{\omega})=\{P_1\}$. We explicitly construct the family of ambient bi-Lipschitz homeomorphisms that trivialises this family. This construction is inspired by the works in \cite{AraujoSanchez2024} and \cite{Bode2019}. 

Consider $g(u,\rme^{\rmi t})=u^3+\rme^{2\rmi t}u$. The zero set of the function $g$ forms the closure of a braid  on $3$-strands. Consider the functions formed by the translations by $\omega \in \C$ of the zero set of $g$, i.e.,  
$$g_{\omega}(u,\rme^{\rmi t})=(u+\omega)^3+\rme^{2\rmi t}(u+\omega).$$
The isotopic class of the zero set of $g_\omega$ does not depend on $\omega$.  
 For a positive integer number $\kappa$, define the semiholomorphic polynomial 
\begin{equation*}\label{rescaling}
f_{\kappa,\omega}(u,v)=|v|^{6\kappa}g_\omega(u/|v|^{2\kappa},v/|v|).
\end{equation*}
A direct calculation shows that
\begin{equation*}
    f_{\kappa,\omega}(u,v)
    =u^3+3\omega (v\bar{v})^{\kappa} u^2+(3  \omega^2 (v\bar{v})^{2\kappa}+ (v\bar{v})^{2\kappa-1} v^2)u +\omega^3(v\bar{v})^{3\kappa} + \omega (v\bar{v})^{3\kappa-1} v^2.
\end{equation*}
 Then, $f_{\kappa,\omega}$ is a radial mixed polynomial of radial-type \((P;d)=(2\kappa,1;6\kappa)\),  and it is \(u\)-convenient. The mixed polynomials $f_{\kappa, \omega}$ can be seen as a family of $f_{\kappa,0}$ deformed by terms $\theta_{\kappa}(\omega, u,v, \bar{v})$ with  $d(P;\theta_{\kappa})= d(P;f_{\kappa,0})$. We prove that this family is ambient bi-Lipschitz $V$-trivial, and the trivialisation holds for any $\omega\in \C$. Indeed, the map $\phi_{\omega}: \C^2 \to \C^2$  defined by $\phi_{\omega}(u,v)=(u-|v|^{2\kappa}\omega,v)$ satisfies  $\phi_{\omega}(V(f_{\kappa,0}))=V(f_{\kappa,\omega})$, and  
\begin{equation*} 
||\phi_{\omega}(u_1,v_1)-\phi_{\omega}(u_2,v_2)||
 =||(u_1,v_1)-(u_2,v_2)+((|v_2|^{2\kappa}-|v_1|^{2\kappa})\omega,0)||.  
\end{equation*}
For \( v_1,v_2 \in \C\) satisfying \(|v_1|, |v_2|< K \), we obtain 
\[ ||v_1|^{2\kappa}- |v_2|^{2\kappa}|=||v_1|-|v_2||\cdot\left|\sum_{s=0}^{2\kappa-1}|v_1|^s|v_2|^{2\kappa-1-s}\right|\le2\kappa K^{2\kappa-1} |v_1-v_2|. \]
Then,
\begin{align*}
\|\phi_{\omega}(u_1,v_1)-\phi_{\omega}(u_2,v_2)\|\leq&  \|(u_1,v_1)-(u_2,v_2)\|+\|((|v_2|^{2\kappa}-|v_1|^{2\kappa})\omega,0)\|\\ 
\leq& ||(u_1,v_1)-(u_2,v_2)||+ 2\kappa |\omega| K^{2\kappa-1} |v_1-v_2| \\
\leq&  (1+ 2\kappa|\omega| K^{2\kappa-1})||(u_1,v_1)-(u_2,v_2)||. 
\end{align*}
Analogously, $\phi_{\omega}^{-1}: \C^2 \to \C^2$  defined by \(\phi_{\omega}^{-1}(u,v)=(u+|v|^{2\kappa}\omega,v)\) is Lipschitz. Therefore,
$\phi_{\omega}$ is a ambient bi-Lipschitz map and \(\phi_{\omega}(V(f_{\kappa,0}))=V(f_{\kappa,\omega})\) for any \(\omega\in \C\).
\end{example}

For any two complex numbers $\omega_1$ and $\omega_2$, the two Type-I semi-radial mixed polynomials, $f_{\kappa,\omega_1}$ and $f_{\kappa,\omega_2}$, are ambient bi-Lipschitz $V$-equivalent because they belong to an ambient bi-Lipschitz trivial family. However, they are not equivalent to $f_{\kappa',\omega_1}$ and $f_{\kappa',\omega_2}$ for $\kappa' \neq \kappa$ (Proposition~\ref{kinvarianceI}). Our ongoing research focusses on studying the (ambient) bi-Lipschitz $V$-equivalence for pairs of IND mixed polynomials $f$ and $g$ that satisfy $\mathcal{NC}(f)=\mathcal{NC}(g)$ and that are not necessarily part of a family of mixed polynomials. The goal is to obtain new, effective invariants derived from their $C$-diagrams that can fully classify these equivalence classes.

\end{document}